\documentclass[10pt]{amsart} 

\usepackage{soul} 
\usepackage[dvipsnames]{xcolor}
\usepackage[bookmarks,colorlinks=true,citecolor=OliveGreen,linkcolor=RoyalBlue]{hyperref}
\usepackage{amssymb,amsthm}
\usepackage[inline]{enumitem}
\usepackage{mathrsfs}
\usepackage{mathtools}
\usepackage{xparse}
\usepackage{mathabx}
\usepackage{nicefrac}
\usepackage{thmtools}
\usepackage[capitalise]{cleveref}		
\usepackage{bm}

\usepackage{lmodern}

\makeatletter
\def\thmt@refnamewithcomma #1#2#3,#4,#5\@nil{%
  \@xa\def\csname\thmt@envname #1utorefname\endcsname{#3}%
  \ifcsname #2refname\endcsname
    \csname #2refname\expandafter\endcsname\expandafter{\thmt@envname}{#3}{#4}%
  \fi
}
\makeatother

\declaretheorem[numberwithin=section]{theorem}
\declaretheorem[sibling=theorem]{proposition}
\declaretheorem[sibling=theorem]{corollary}
\declaretheorem[sibling=theorem]{question}
\declaretheorem[sibling=theorem]{lemma}

\declaretheorem[sibling=theorem,style=definition]{definition}

\declaretheorem[sibling=theorem,style=remark]{remark}

\newcounter{claimCounter}[theorem]

\newcounter{subClaimCounter}[claimCounter]

\newlist{equivalent}{enumerate}{1}
\setlist[equivalent,1]{label=\textup{(\arabic*)}}

\newlist{sublemma}{enumerate}{1}
\setlist[sublemma,1]{label=\textup{(\alph*)}}

\newlist{sublemma*}{enumerate*}{1}
\setlist[sublemma*,1]{label=\textup{(\alph*)},afterlabel=\hspace{5pt}}

\newlist{orderedlist}{enumerate}{1}
\setlist[orderedlist,1]{label=\textup{(\roman*)}}

\newlist{orderedlist*}{enumerate*}{1}
\setlist[orderedlist*,1]{label=\textup{(\roman*)},afterlabel=\hspace{3pt}}

\newcommand{\seq}[1]{{\left\langle{#1}\right\rangle}}
\newcommand{\smallseq}[1]{\langle{#1}\rangle}

\newcommand{\rest}[1]{\restriction{#1}} 
\newcommand{\tth}{{}^{\textup{th}}}
\newcommand{\conc}{\hat{\,\,}}
\newcommand{\cat}{\widehat{\phantom{\alpha}}}
\newcommand{\andd}{\,\,\,\&\,\,\,}
\newcommand{\orr}{\,\vee\,}
\newcommand{\Then}{\,\Longrightarrow\,}
\newcommand{\Iff}{\,\,\Longleftrightarrow\,\,}
\renewcommand{\iff}{\leftrightarrow}
\newcommand{\then}{\,\,\rightarrow\,\,}

\DeclareMathOperator{\dom}{dom}

\DeclareMathOperator{\rk}{rk}

\newcommand{\converge}{\!\!\downarrow}
\newcommand{\diverge}{\!\!\uparrow}

\newcommand{\w}{\omega}
\newcommand{\s}{\sigma}
\newcommand{\vphi}{\varphi}
\renewcommand{\epsilon}{\varepsilon}
\renewcommand{\setminus}{\smallsetminus}

\renewcommand{\le}{\leqslant}
\renewcommand{\ge}{\geqslant}

\renewcommand{\preceq}{\preccurlyeq}
\renewcommand{\succeq}{\succcurlyeq}

\newcommand{\Nat}{\mathbb N}
\newcommand{\RR}{\mathbb R}

\DeclarePairedDelimiterX{\norm}[1]{\lVert}{\rVert}{#1}

\newcommand{\bSigma}{\boldsymbol{\Sigma}}

\newcommand{\bPi}{\boldsymbol{\Pi}}

\newcommand{\force}{\Vdash}
\newcommand{\nforce}{\nVdash}

\newcommand{\rooot}{\seq{}}
\newcommand{\Leaves}{\mathscr{L}}
\newcommand{\XXX}{\mathbb{X}}
\newcommand{\YYY}{\mathbb{Y}}
\newcommand{\WWW}{\mathbb{W}}
\newcommand{\symdiff}{\bigtriangleup}
\newcommand{\PP}{\mathbb{P}}
\newcommand{\QQ}{\mathbb{Q}}

\newcommand{\PowerSet}{\mathscr P}

\newcommand{\Dpoint}{\+D_{\textup{point}}}
\DeclareMathOperator{\trk}{rk}

\newcommand{\SSS}{\mathbb{S}}
\newcommand{\TTT}{\mathbb{T}}
\newcommand{\UUU}{\mathbb{U}}
\newcommand{\AAA}{\mathbb{A}}
\newcommand{\BBB}{\mathbb{B}}

\title{Topology, forcing, and graph colourings}

\author[N.~Greenberg]{Noam Greenberg}
\address{School of Mathematics and Statistics\\
Victoria University of Wellington\\
P.O.~Box 600, Wellington, New Zealand.}
\email{greenberg@sms.vuw.ac.nz}

\author[D.~Lecomte]{Dominique Lecomte}
\address{Sorbonne Universit\'e and Universit\'e Paris Cit\'e, CNRS, IMJ-PRG, F-75005 Paris, France\\
and Universit\'e de Picardie, I.U.T. de l'Oise, site de Creil, 13 all\'ee de la faïencerie, 60100 Creil, France.}
\email{dominique.lecomte@upmc.fr}

\author[D.~Turetsky]{Dan Turetsky}
\address{School of Mathematics and Statistics\\
Victoria University of Wellington\\
P.O.~Box 600, Wellington, New Zealand.}
\email{dan.turetsky@vuw.ac.nz}

\author[M.~Zelen\'y]{Miroslav Zelen\'y}
\address{Charles University, Faculty of Mathematics and Physics, Department of Mathematical Analysis \\ Sokolovsk\'a 83, 186 75 Prague, Czech Republic.}
\email{zeleny@karlin.mff.cuni.cz}

\begin{document}

\begin{abstract}
We introduce a family of forcing notions that are helpful in showing that certain graphs do not have countable $\bSigma^0_\alpha$ colourings. We construct graphs that are ``weakly minimal'' for such colourings. 
\end{abstract}

\maketitle

\section{Introduction}

One of the major results in descriptive set theory is the \emph{$\mathbb{G}_0$ dichotomy}:

\begin{theorem}[Kechris,Solecki,Todor\v{c}evi\'{c}, \cite{KST:G_0}]\label{thm:G_0_dichotomy}
  There is a Borel directed graph $\mathbb{G}_0$ on Cantor space such that for any analytic directed graph $G$ on a Polish space~$X$, exactly one of the following holds:
  \begin{equivalent}
    \item $G$ has a countable Borel colouring; 
    \item There is a continuous homomorphism from $\mathbb{G}_0$ to~$G$. 
  \end{equivalent}
\end{theorem}

This result has found a large number of applications (for a survey, see for example \cite{BenMiller}). It is natural to ask for a level-by-level version of this result, with respect to the Borel hierarchy. This work was initiated in \cite{LecomteZeleny}, where the authors prove the following. 

\begin{theorem}[Lecomte,Zelen\'y, \cite{LecomteZeleny}] \label{thm:LZ}
 Let $\alpha\in \{1,2,3\}$. 
There is a zero-dimensional Polish space $\XXX_\alpha$, and an analytic directed graph $A_\alpha$ on~$\XXX_\alpha$, such that for any Polish space~$X$ (zero dimensional if $\alpha=1$), for any analytic directed graph~$G$ on~$X$, exactly one of the following holds:
\begin{equivalent}
  \item There is a countable $\bSigma^0_\alpha$ colouring of~$G$. 
  \item There is a continuous graph homomorphism from $A_\alpha$ to~$G$. 
\end{equivalent}
\end{theorem}

The known proofs of \cref{thm:LZ} (in \cite{LecomteZeleny} and in \cite{LZ2:rectangles}) for $\alpha = 3$ are quite technical, and thus difficult to generalise. In this paper we introduce a method for defining families of graphs without $\bSigma^0_\alpha$ colourings. For these families we obtain a weak version of \cref{thm:LZ} (see \cref{prop:dichotomy_alpha_alpha} below), in which the graph homomorphisms are not continuous, but in a class which is sufficient to preclude such colourings. Such graphs are relatively easy to define (showing that they are not colourable is a different matter). However, for $\alpha=3$, the graph defined in \cref{def:K_alpha} is not minimal for non-colourability (see \cref{prop:no_embedding_of_K_3_into_L_3}). In two steps, we devise a family of graphs~$H_\alpha$ (\cref{def:Y_alpha}) that does generalise the graphs from \cref{thm:LZ}, therefore giving yet a new proof of \cref{thm:LZ} (see \cref{prop:H_alpha_is_not_colourable} and \cref{thm:minimality_of_H_3}). 

The definition of the graphs $H_\alpha$ involves a notion of \emph{true stages}, also known as a \emph{representation theorem}, for specific $\Sigma^0_\alpha$ subsets of Cantor space. The representation theorem for Borel sets of Debs and Saint Raymond \cite{DS:representations} was used several times in investigations of close topics (see for example \cite{Lecomte:potential,Lecomte:SeparationRectangles,L5}). An effective version of this method was independently introduced by Montalb\'an \cite{Montalban:TrueStages}, articulating dynamic aspects of iterated priority arguments in computable structure theory, originally designed by Ash \cite{Ash:Metatheorem}. Montalb\'an's method was used in effective descriptive set theory \cite{DDW,DGHT:BSL} and computability theory to analyse variants of Wadge reducibility \cite{DQT2}, and their proof-theoretic strength as measured by reverse mathematics \cite{DGHT:JEMS}. In this paper we develop a particular system of true stage relations suited for our purposes. In future work, we aim to present a more general framework using general systems of representations (see for example \cite{Lecomte:potential}). 

The main tool that we introduce, to prove that certain graphs do not have $\bSigma^0_\alpha$ colourings, is a family of notions of forcing that all share an ``untagging'' property, inspired by Steel's method of forcing with tagged trees \cite{Steel:tagged}. One of our aims is to explain how this method can be also presented in the language of topology, using Baire category as a main tool. In \cref{sec:topology_and_forcing} we give a detailed development of the simplest notion of forcing in our family, and explain how to view it topologically; we explain the connection to a result of M\'atrai \cite{Matrai}, that relates descriptive complexity and Baire category. In \cref{sec:a_weak_dichotomy} we introduce the first family of graphs, $K_\alpha$, and prove the weak dichotomy theorem \cref{prop:dichotomy_alpha_alpha}. Toward defining the family of graphs $H_\alpha$, in \cref{sec:_smaller_non_colourable_graphs} we define an intermediate family (\cref{def:W_alpha_and_L_alpha}), and present a more elaborate notion of forcing to show non-colourability of these graphs. The most complicated family of graphs, $H_\alpha$, is defined in \cref{sec:candidates_for_minimal_graphs}, where we first develop the true stage machinery required to define these graphs. In \cref{sec:an_embedding_result} we show the minimality of~$H_3$. 

In \cref{sec:separators_of_iterated_fr_echet_ideals} we give a new proof of a result of Debs and Saint Raymond regarding separators of the iterated Fr\'echet filters and ideals, using yet another variant of our forcing with untagging. Finally, in \cref{sec:questions} we list some open questions, including one regarding separating subsets of product spaces by countable unions of Borel rectangles. This is a family of results and problems closely related to graph colourings; we mention some of the background in that section.

\section{Topology and forcing} \label{sec:topology_and_forcing}

\subsection{The space}

Let $\alpha$ be a computable ordinal. (We work with a computable ordinal since some of our results rely on lightface arguments; however, all results relativise to an oracle, so apply to all countable ordinals.)

For the purposes of the following definition, and the rest of the section, we identify~$\alpha$ with a computable well-ordering of a computable subset of~$\Nat$, for which the successor relation and the set of limit points are both computable (roughly, an ``ordinal notation''). Note that this computable presentation gives, for every limit $\delta\le \alpha$, uniformly, a computable increasing and cofinal sequence $(\delta_k)$ in~$\delta$.

\begin{definition} \label{def:the_space_L_alpha}
  We 
  let $T_\alpha\subset \Nat^{<\w}$ be a computable well-founded tree of rank~$\alpha$, which has a computable rank function, defined as follows:
  \begin{itemize}
    \item The root~$\rooot$ of the tree~$T_\alpha$ has rank $\alpha$; 
    
    \item If $\delta\le \alpha$ is a limit, with a chosen computable cofinal and increasing sequence $(\delta_k)$, and $\s\in T_\alpha$ has rank~$\delta$, then for all~$k$, $\s\conc k\in T_\alpha$, and has rank~$\delta_k$. 

    \item If $\beta<\alpha$ and $\s\in T_\alpha$ has rank $\beta+1$, then for all~$k$, $\s\conc k\in T_\alpha$ and has rank~$\beta$. 
  \end{itemize}

  Note that $T_\alpha$ is not ``saturated''; if $\s\in T_\alpha$ and $\rk(\s)>1$, then not all $\beta<\rk(\s)$ are realised as ranks of children of~$\s$ (nodes of the form $\s\conc k$). Indeed, we will make use of the following fact: for all $\s\in T_\alpha$, for all $\beta<\rk(\s)$, there are only finitely many~$k$ such that $\rk(\s\conc k)<\beta$. Indeed, if $\rk(\s)$ is a successor, then there are no such~$k$; if $\rk(\s)=\delta$ is a limit, then we use the fact that $(\delta_k)$ is inceasing and cofinal in~$\delta$. 

  We let~$\Leaves_\alpha$ denote the collection of leaves of~$T_\alpha$, which is a computable set (the leaves of~$T_\alpha$ are the nodes of rank~$0$). Any computable 
  bijection between $\Leaves_\alpha$ and~$\w$ induces a computable isomorphism of the space $2^{\Leaves_\alpha}$ and Cantor space~$2^\Nat$. 
\end{definition}

Henceforth, we often suppress mention of the ordinal~$\alpha$ in some subscripts.

\begin{definition} \label{def:the_tree_determined_by_a_labelling_of_the_leaves}
  Suppose that $x\in 2^{\Leaves_\alpha}$. We define a $\{0,1\}$-valued labelling $T^x= T^x_\alpha$ of~$T_\alpha$, extending~$x$, by transfinite recursion on the nodes of~$T_\alpha$, as follows:
  \begin{itemize}
    \item For each $\s\in \Leaves_\alpha$, $T^x(\s) = x(\s)$;
    \item For each $\s\in T_\alpha\setminus \Leaves_\alpha$, 
    \[
      T^x(\s) = 0 \qquad \Iff \qquad (\exists k)\,\,T^x(\s\conc k)=1. 
    \]
  \end{itemize}
\end{definition}

The definition of $T^x(\s)=0$ mimics an existential quantifier. By effective transfinite recursion, we obtain:

\begin{lemma} \label{lem:the_complexity_of_a_label_of_a_node}
  If $\s\in T_\alpha$ has rank~$\beta$, then 
  \[
    \left\{ x\in 2^{\Leaves_\alpha} \,:\,   T^x(\s) = 0   \right\}
  \]
  is $\Sigma^0_{\beta}$, uniformly in~$\beta$. 
\end{lemma}

If $\alpha=1$, then $T_\alpha$ consists of a root and the set of leaves (which all have height~1). The space $2^{\Leaves_1}$ is identified with Cantor space in a straightforward way: $x\in 2^{\Leaves_1}$ is identified with $k\mapsto x(\smallseq{k})$. (We say that the \emph{location}~$k$ for Cantor space, which is an element of~$\Nat$, is identified with the \emph{location} $\smallseq{k}$ for the space $2^{\Leaves_1}$, which is an element of $\Leaves_1$.) The value of $T^x$ at the root records the fact if, considered as a subset of~$\Nat$ via identification with its characteristic function, $x$ is empty or not, with the value~0 indicating the latter. 

If $\alpha=2$, then the space $2^{\Leaves_2}$ is naturally identified with $(2^\Nat)^\Nat$, which in turn is identified with Cantor space using a pairing function. A location in~$\Nat$ coding the pair $(k,l)$ is identified with the leaf $\smallseq{k,l}\in \Leaves_2$. The value of~$T^x$ on a rank 1 node $\smallseq{k}$ records whether the $k\tth$ column of~$x$ is empty or not; while the value of~$T^x$ at the root records whether $x$ has an empty column or not. 

This continues to higher ranks. When $\alpha=3$, the space is naturally identified with $((2^\Nat)^\Nat)^\Nat$, and the value of~$T^x$ at the root records if, when considering~$x$ as built up of columns of columns, it has a column, each sub-column of which is nonempty; and so on. At limit levels~$\delta$, the space $2^{\Leaves_\delta}$ is naturally identified with the product $\prod 2^{\Leaves_{\delta_k}}$. 

We remark that these ``versions of Cantor space'' were used in the past, in particularly for understanding the iterated Fr\'echet ideals and filters; see \cite{DebsSaintRaymond,DayMarks}.

\subsection{The notion of forcing, and the associated topology}

\begin{definition} \label{def:simple_forcing_notion}
   We let $\QQ$ be the collection of all partial functions $p\colon T_\alpha\to \{0,1\}$ satisfying:
   \begin{orderedlist}
     \item $\dom p$ is finite; and
     \item If $\s\in T_\alpha\setminus \Leaves_\alpha$, then $p(\s)=0$ if and only if there is some~$k$ such that $p(\s\conc k)=1$. 
    \end{orderedlist}
   The set~$\QQ$ is partially ordered by reverse extension: for $p,q\in \QQ$, $q\le p$ if and only if $p\subseteq q$. 
\end{definition}

The partial ordering $\QQ$ is called a \emph{notion of forcing}, and its elements are called \emph{forcing conditions}. If $q\le p$ then we say that~$q$ \emph{extends}~$p$.  We also say that two conditions $p$ and~$q$ are \emph{compatible} if they have a common extension in~$\QQ$. 

\begin{lemma} \label{lem:simple_forcing:compatibility_is_being_a_function}
  Conditions $p,q\in \QQ$ are compatible if and only if $p\cup q$ is a function, in which case $p\cup q\in \QQ$. 
\end{lemma}

\begin{proof}
  If $r$ extends both~$p$ and~$q$ then $p\cup q\subseteq r$, implying that $p\cup q$ is a function. On the other hand, suppose that $s = p\cup q$ is a function. For all $\s\in T_\alpha\setminus \Leaves_\alpha$, $s(\s)=0$ if and only if $p(\s)=0$ or $q(\s)=0$ if and only if there is some~$k$ such that $p(\s\conc k)=1$ or $q(\s\conc k)=1$ if and only if there is some~$k$ such that $s(\s\conc k)=1$; so $s\in \QQ$, and extends both~$p$ and~$q$. 
\end{proof}

Viewed topologically, a condition $p\in \QQ$ can be thought of as a code of a $\Delta^1_1$ subset of our space $2^{\Leaves_\alpha}$.

\begin{definition} \label{def:simple_forcing:topology_of_space}
  For $p\in \QQ$ we let 
\[
  [p] = \left\{ x\in 2^{\Leaves_\alpha} \,:\, T^x\supset p  \right\}. 
\]
We let $\tau_\QQ$ denote the topology on $2^{\Leaves_\alpha}$ generated by the sets $[p]$ for $p\in \QQ$. 
\end{definition}

Observe that if $q$ extends $p$ then $[q]\subseteq [p]$ (this explains writing $q\le p$ when~$q$ extends~$p$, rather than $p\le q$). The converse holds, but this is not (yet) immediate, nor is it important. 

\smallskip

We will shortly verify that the $\tau_\QQ$-topology satisfies the Baire category theorem. To do so, we will show that we can view points in the space $2^{\Leaves_\alpha}$ as {filters} on~$\QQ$. We recall the following definitions. 

\begin{definition} \label{def:simple_forcing:generic_filters} \ 
  \begin{sublemma}
  \item A \emph{filter} on~$\QQ$ is a nonempty subset $G\subseteq \QQ$ that is:
  \begin{itemize}
      \item directed, i.e., for all $p,q\in G$ there is some $r\in G$ that extends both~$p$ and~$q$; and
      \item upwards closed, i.e., for all $p, q\in \QQ$, if $q$ extends~$p$ and $q\in G$ then $p\in G$. 
    \end{itemize}  

  \item A set $D\subseteq \QQ$ is called \emph{dense} if every $p\in \QQ$ has an extension in~$D$. 

  \item Let $\+D$ be a collection of dense subsets of~$\QQ$. We say that a filter $G$ of~$\QQ$ is \emph{$\+D$-generic} if $G$ intersects every $D\in \+D$.   
  \end{sublemma}
\end{definition}

\begin{definition} \label{def:simple_forcing:Dpoint}
  We define some subsets of $\QQ$:
\begin{itemize}
  \item For $\s\in T_\alpha$, let $D_\s$ be the set of $p\in \QQ$ such that $\s\in \dom p$. 


  \item For $p\in \QQ$, let $D_p$ be the set $\left\{ q\in \QQ \,:\,   q\le p, \text{ or } p \text{ and } q \text{ are incompatible}\right\}$. 
\end{itemize}

We let $\Dpoint$ be the collection of all sets $D_\s$ for $\s\in T_\alpha$ and $D_p$ for $p\in \QQ$. 
\end{definition}

\begin{lemma} \label{lem:simple_forcing:density_of_the_sets_in_Dpoint}
  Each set in $\Dpoint$ is dense. 
\end{lemma}

\begin{proof}
  Let $\s\in T_\alpha$, and let $p\in \QQ$; we find an extension of~$p$ in $D_\s$. If $\s\notin \dom p$, we extend~$p$ to $p'$ by setting $p'(\s)=0$, and further, if~$\s$ is not a leaf of~$T_\alpha$, then we set $p'(\s\conc k)=1$ for some large~$k$ (so that $\dom p$ contains no extensions of $\s\conc k$). This choice of~$k$ ensures that $p'\in \QQ$, so $p'\in D_\s$. 

  For $D_p$, the argument is quick. Let $q\in \QQ$. If $p$ and $q$ are incompatible  then $q\in D_p$. Otherwise, $q$ and~$p$ have a common extension~$r$; then $r$ is an extension of~$q$ in~$D_p$.
\end{proof}

\begin{definition} \label{def:simple_forcing:x_G}
  For a filter $G\subseteq \QQ$, we let 
  \[
    x_G = \big(\bigcup G\big)\rest{\Leaves_\alpha}. 
  \]
\end{definition}

\begin{lemma} \label{lem:simple_forcing:filters_are_points}
  If $G$ is $\Dpoint$-generic, then $x_G\in 2^{\Leaves_\alpha}$, and for all $p\in \QQ$, 
  \[
    x_G \in [p] \Iff p\in G.
  \]
\end{lemma}

\begin{proof}
  Suppose that $G$ is a $\Dpoint$-generic filter. We let 
  \[
    S = \bigcup G.
  \]
  Since $G$ is directed, $S$ is a function on $T_\alpha$. Since $G\cap D_\s\ne\emptyset$ for all $\s\in T_\alpha$, $S$ is a total function, i.e., $\dom S = T_\alpha$. 

  We claim that $S = T^{x_G}$. To do so, we check that~$S$ satisfies the definition of $T^{x_G}$. By definition of~$x_G$, we have $S\rest{\Leaves_\alpha} = x_G$. Suppose that $\s\in T_\alpha\setminus \Leaves_\alpha$. If $S(\s)=0$, let $p\in G$ with $p(\s) = 0$; since $p\in \QQ$, there is some~$k$ such that $p(\s\conc k)=1$, so $S(\s\conc k)=1$. On the other hand, if for some~$k$ we have $S(\s\conc k)=1$, let $p\in G$ such that $p(\s\conc k)=1$; since $p\in \QQ$, $p(\s)=0$, so $S(\s)=0$. 

  Hence, if $p\in G$, then $p\subset S = T^{x_G}$, that is, $x_G\in [p]$. In the other direction, let $p\in \QQ$, and suppose that $x_G\in [p]$, i.e., that $p\subseteq S$. For all $q\in G$, $q\subseteq S$, so $p\cup q$ is a function. By \cref{lem:simple_forcing:compatibility_is_being_a_function}, $p$ and~$q$ are compatible. That is, $p$ is compatible with all $q\in G$. Since $G\cap D_p\ne\emptyset$, $G$ must contain an extension of~$p$. Since~$G$ is closed upwards in~$\QQ$, $p\in G$.
\end{proof}

\begin{remark} \label{rmk:Stone_duality}
  In fact, the map $G\mapsto x_G$ is a \emph{bijection} between the $\Dpoint$-generic filters and the space $2^{\Leaves_\alpha}$, giving us some kind of Stone duality. We will however not need this fact. It is enough that ``most'' points in $2^{\Leaves_\alpha}$ are~$x_G$ for a generic~$G$, in a sense made precise below.  
\end{remark}

The following is often referred to as the \emph{Rasiowa-Sikorski lemma}. 

\begin{lemma} \label{lem:simple_forcing:Rasiowa-Sikorski}
  For any countable collection $\+D$ of dense subsets of~$\QQ$, for all $p\in \QQ$, there is a $\+D$-generic filter~$G$ such that $p\in G$. 
\end{lemma}

\begin{proof}
  Write $\+D = \left\{ D_n \,:\, n\in \Nat  \right\}$ and define a decreasing sequence of conditions $(p_n)$ as follows: $p_0=p$; given $p_n$, $p_{n+1}$ is some condition extending~$p_n$ such that  $p_{n+1}\in D_n$, which exists since~$D_n$ is dense. Then $G = \left\{ q \,:\,  (\exists n)\,\,p_n\text{ extends }q \right\}$ is a $\+D$-generic filter containing~$p$. 
\end{proof}

Applying \cref{lem:simple_forcing:Rasiowa-Sikorski} to $\Dpoint$, we get that for all $p\in \QQ$, $[p]\ne\emptyset$. This implies:

\begin{lemma} \label{lem:simple_forcing:dense_open_and_comeagre} \ 
  \begin{sublemma}
      \item A set $U\subseteq 2^{\Leaves_\alpha}$ is $\tau_\QQ$-dense and open if and only if there is a dense set $D\subseteq \QQ$ such that 
    \[
      U =  \bigcup \left\{ [p] \,:\,  p\in D \right\}. 
        \]

    \item A set $A\subseteq 2^{\Leaves_\alpha}$ is $\tau_\QQ$-comeagre if and only if there is a countable collection $\+D\supseteq \Dpoint$ of dense subsets of~$\QQ$ such that 
  \[
    A\supseteq \left\{ x_G \,:\,  G\text{ is a $\+D$-generic filter} \right\}.
  \]
  \end{sublemma}
\end{lemma}

Now \cref{lem:simple_forcing:dense_open_and_comeagre} and \cref{lem:simple_forcing:Rasiowa-Sikorski} imply: 

\begin{proposition} \label{prop:simple_forcing:Baire_category_theorem}
  The topology $\tau_\QQ$ satisfies the Baire category theorem: for all $p\in \QQ$ and $\tau_\QQ$-comeagre set $A\subseteq 2^{\Leaves_\alpha}$, $[p]\cap A\ne \emptyset$. 
\end{proposition}

\begin{remark} \label{rmk:tau_Q_is_Polish}
  The topology $\tau_\QQ$ is Polish. To get a complete metric for this topology, fix an $\w$-ordering $(\sigma_n)$ of~$T_\alpha$; for distinct $x,y\in 2^{\Leaves_\alpha}$, the distance between~$x$ and~$y$ is $2^{-n}$, where~$n$ is least such that $T^x(\s_n)\ne T^y(\s_n)$. 
\end{remark}

\begin{definition} \label{def:the_sufficiently_generic_quantifier}
   The quantifier ``for all sufficiently generic~$G$, \dots'' means: there is a countable collection~$\+D$ of dense subsets of~$\QQ$ such that for every $\+D$-generic filter~$G$, \dots

 By \cref{lem:simple_forcing:dense_open_and_comeagre}, for any $A\subseteq 2^{\Leaves_\alpha}$, the following are equivalent: (1) For all sufficiently generic~$G$, $x_G\in A$. (2) $A$ is $\tau_\QQ$-comeagre.
\end{definition}

\subsection{The strong forcing relation}

The (very restricted) \emph{forcing language} that we will use consists of codes of Borel subsets of our space.

\begin{definition} \label{def:simple_forcing:Borel_codes}
  We define, by recursion, the collection of \emph{Borel codes} of subsets of $2^{\Leaves_\alpha}$:
  \begin{itemize}
    \item Any finite partial function from $\Leaves_\alpha$ to $\{0,1\}$ is a Borel code;
    \item If $\vphi$ is a Borel code, then $\lnot \vphi$ is a Borel code;
    \item If $(\vphi_n)_{n<\w}$ is an $\w$-sequence of Borel codes, then $\bigvee_n \vphi_n$ is a Borel code.
  \end{itemize}
\end{definition}

Thus, Borel codes can be identified as formulas in an infinitary propositional logic, or alternatively, as labelled well-founded trees (where the leaves are labelled by finite functions as described,  and each non-leaf is labelled by either $\lnot$ or $\vee$; a node labelled by~$\lnot$ has one child only). We omit conjunction to make some following definitions and arguments shorter. 

For a Borel code~$\vphi$, we let $[\vphi]$ denote the subset of $2^{\Leaves_\alpha}$ that is coded by~$\vphi$, namely:
\begin{itemize}
  \item If $r\colon \Leaves_\alpha \to \{0,1\}$ is finite, then $[r] = \left\{ x\in 2^{\Leaves_\alpha} \,:\, r\subset x  \right\}$ (this is a special case of \cref{def:simple_forcing:topology_of_space});
  \item $[\lnot\vphi] = [\vphi]^\complement = 2^{\Leaves_\alpha}\setminus [\vphi]$; 
  \item $[\bigvee \vphi_n] = \bigcup_n [\vphi_n]$. 
\end{itemize}

Note that at the basic level we take finite partial functions on $\Leaves_\alpha$, not all elements of~$\QQ$. Thus, we are considering Borel sets as they are generated from the standard topology on $2^{\Leaves_\alpha}$, rather than the $\tau_\QQ$-topology. 

Certainly, every Borel subset of $2^{\Leaves_\alpha}$ has a code, indeed many different codes. The connection between generic points and Borel sets is given by the \emph{forcing relation}, which is the heart of the theory. The standard forcing relation $p\force \vphi$ (between conditions and Borel codes) is defined combinatorially, but is equivalent to $[\vphi]$ being $\tau_\QQ$-comeagre in~$[p]$. For our purposes, we will need a variation~$\force^*$, a stronger version of the usual forcing relation. 

\begin{definition} \label{def:simple_forcing:strong_forcing_relation}
   We define the relation $\force^*$ between~$\QQ$ and Borel codes, by recursion on the complexity of the Borel code. Let $p\in \QQ$.
   \begin{itemize}
     \item If $r\colon \Leaves_\alpha\to 2$ is a finite partial function, then $p\force^* r$ if $p$ extends~$r$.

     \item $p\force^* \bigvee_n \vphi_n$ if there is some~$n$ such that $p\force^* \vphi_n$. 

     \item $p\force^* \lnot \vphi$ if there is no $q$ extending~$p$ such that $q\force^* \vphi$. 
   \end{itemize}
   If $p\force^* \vphi$, we say that~$p$ \emph{strongly forces~$\vphi$}. 
\end{definition}

We remark that the standard forcing relation can be defined as follows: $p\force \vphi$ if the collection of $q\le p$ that strongly force~$\vphi$ is dense below~$p$ (every extension of~$p$ has an extension that strongly forces~$\vphi$). The difference between the two notions lies in disjunctions: suppose that densely below some~$p$, we have conditions that strongly force some~$\vphi_n$. Then $p$ forces~$\vphi$, but~$p$ itself may not force any particular~$\vphi_n$, so does not strongly force~$\vphi$.  In contrast, $p\force^* \lnot \vphi$ if and only if $p$ forces $\lnot\vphi$. 

\begin{lemma} \label{lem:basic_forcing:basic_properties_of_strong_forcing}
  Let $\vphi$ be a Borel code. 
  \begin{sublemma}
    \item There is no $p\in \QQ$ which strongly forces both $\vphi$ and $\lnot\vphi$. 
    \item $\left\{ p\in \QQ \,:\, p\force^* \vphi \orr  p\force^*\lnot \vphi   \right\}$ is dense. 
    \item If $q$ extends $p$ and $p\force^* \vphi$ then $q\force^* \vphi$. 
  \end{sublemma}
\end{lemma}

\begin{proof}
  (a) and~(b) are immediate, from the definition of strongly forcing $\lnot\vphi$. (c) is immediate except for disjunctions, and so is proved by induction on the complexity of the Borel code~$\vphi$. 
\end{proof}

The \emph{forcing theorem} shows how Borel sets have the property of Baire for the topology $\tau_\QQ$.

\begin{proposition} \label{prop:simple_forcing:forcing_theorem}
  For each Borel code~$\vphi$, for all sufficiently generic~$G$, 
  \[
    x_G\in [\vphi]\,\,\Iff\,\,(\exists p\in G)\,\,(p\force^* \vphi).
  \]
\end{proposition}

\begin{proof}
This is proved by induction on the complexity of~$\vphi$. For each~$\vphi$, we define a countable collection $\+D_{\vphi}\supseteq \Dpoint$ of dense subsets of~$\QQ$ such that the equivalence above holds for every $\+D_{\vphi}$-generic filter~$G$.

\smallskip

If $\vphi = r$ is a finite partial function from $\Leaves_\alpha$ to~$\{0,1\}$, then we can take $\+D_{r} = \Dpoint$. The desired equivalence follows from \cref{lem:simple_forcing:filters_are_points}. 

  \smallskip

  If $\vphi$ is $\bigvee_n \vphi_n$, then we let 
  \[
    \+D_{\vphi} = \bigcup_n \+D_{\vphi_n}. 
  \]
  Suppose that $G$ is $\+D_{\vphi}$-generic. Then $x_G\in [\vphi]$ if and only if $x_G\in [\vphi_n]$ for some~$n$. By induction, this holds if and only if there is some~$n$ and some $p\in G$ such that $p\force^* \vphi_n$. By definition of strong forcing, this holds if and only if there is some $p\in G$ such that $p\force^* \vphi$. 

  \smallskip

  If $\vphi$ is $\lnot \psi$, then we let $\+D_\vphi$ be $\+D_\psi$, together with the  set 
   \[
   \left\{ p\in \QQ \,:\, p\force^* \psi\orr p\force^*\lnot \psi   \right\}, 
   \]
   which is dense by \cref{lem:basic_forcing:basic_properties_of_strong_forcing}(b). 

   Suppose that $G$ is $\+D_{\vphi}$-generic.  In one direction, suppose that $p\in G$ and $p\force^* \vphi$. That is, there is no $q\le p$ with $q\force^* \psi$. If $x_G\in [\psi]$ then by induction, there is some $r\in G$ such that $r\force^* \psi$. Since~$G$ is a filter, there is some $q\in G$ extending both~$p$ and~$r$. Since~$q$ extends~$r$, by \cref{lem:basic_forcing:basic_properties_of_strong_forcing}, $q\force^*\psi$, contradicting the assumption on~$p$. Hence, $x_G\notin [\psi]$, i.e., $x_G\in [\vphi]$. 

  In the other direction, suppose that $x_G\in [\vphi]$. Since $G$ is $\+D_\vphi$-generic, there is some $p\in G$ such that $p\force^*\psi$ or $p\force^*\vphi$. However, since $x_G\notin [\psi]$, by induction, $p\force^*\psi$ is impossible, so $p\force^*\vphi$. 
\end{proof}

For each Borel code~$\vphi$ we let 
\[
  U_\vphi  = \bigcup \left\{ [p] \,:\,  p\force^* \vphi \right\}, 
\]
which is $\tau_\QQ$-open. Translated, \cref{prop:simple_forcing:forcing_theorem} says:

\begin{proposition} \label{prop:simple_forcing:forcing_theorem_topologically}
  For any Borel code~$\vphi$, $[\vphi]$ and $U_\vphi$ are equivalent modulo a $\tau_\QQ$-meagre set. 
\end{proposition}

\subsection{Untagging lemma}

The development so far is quite general, and can be applied to many notions of forcing other than our particular~$\QQ$. What comes next, however, is special to~$\QQ$ and the similar notions of forcing that we will use in this paper. We consider a fine-grained connection between strong forcing of statements in the Borel hierarchy, and intermediate topologies between the standard one and~$\tau_\QQ$. 

First, we recall that codes can be placed in a syntactic hierarchy that mirrors the Borel hierarchy:
\begin{itemize}
  \item The $\Pi_0$ codes are the codes $r$ for clopen sets (finite functions from $\Leaves_\alpha$ to~$\{0,1\}$). 
  \item For $\beta>0$, a code $\vphi$ is $\Sigma_\beta$ if it is of the form $\bigvee_n \vphi_n$, where each $\vphi_n$ is $\Pi_\gamma$ for some $\gamma<\beta$. 
  \item For $\beta> 0$, a code $\vphi$ is $\Pi_\beta$ if it is of the form $\lnot \psi$, where~$\psi$ is $\Sigma_\beta$. 
\end{itemize}

We observe that if $\vphi$ is a $\Sigma_\beta$ code then $[\vphi]$ is a $\bSigma^0_\beta$ subset of $2^{\Leaves_\alpha}$ (according to the standard topology), and similarly for~$\Pi_\beta$.

\begin{definition} \label{def:simple_forcing:restriction_of_condition_to_beta}
  For $p\in \QQ$ and $\beta\le \alpha+1$, we define a condition $p\rest{\beta}\subseteq p$ as follows: for all~$\s$, $(p\rest\beta)(\s)\converge$ if and only if $p(\s)\converge$,\footnote{Recall that $p(\s)\converge$ means $\s\in \dom p$, and $p(\s)\diverge$ means $\s\notin \dom p$.} and further, either 
  \begin{itemize}
    \item $\rk(\s)<\beta$; or
    \item there is some~$k$ such that $p(\s\conc k)=1$ and $\rk(\s\conc k)<\beta$. 
  \end{itemize}
\end{definition}

That is, $p\rest{\beta}$ records what $p$ says about nodes of rank $<\beta$, and the consequences of that information to parents of nodes.

\begin{lemma} \label{lem:simple_forcing:restricted_condition_is_a_condition}
  Let $p\in \QQ$ and $\beta\le \alpha+1$.
  \begin{sublemma}
    \item $p\rest{\beta}\in \QQ$ and $p$ extends $p\rest{\beta}$. 
    \item If $0<\gamma<\beta\le \alpha+1$ then $p\rest\gamma = (p\rest\beta)\rest\gamma$, so $p\rest\beta$ extends $p\rest\gamma$. 
  \end{sublemma}
\end{lemma}

\begin{definition} \label{def:simple_forcing:beta-complete}
  Let $\beta \le \alpha+1$. We say that $p\in \QQ$ is \emph{$\beta$-complete} if for all $\s\in \dom p$, if $\rk(\s)>\beta$ and $p(\s)=1$, then $p(\s\conc k)\converge$ for all~$k$ such that $\rk(\s\conc k)<\beta$. 
\end{definition}

Note that for all such~$k$, we must have $p(\s\conc k)=0$. That is, for all~$k$, $p$ ``knows'' that $T^x(\s\conc k) = 0$ for all $x\in [p]$, but since $\dom p$ is required to be finite, we cannot actually have $p(\s\conc k)=0$ for all~$k$. The requirement of~$\beta$-completeness is that at least below rank~$\beta$, this knowledge of~$p$ is recorded in~$p$ itself (meaning that~$p$ determines witnesses for $p(\s\conc k)=0$). Note again that the definition only cares about~$\s$ such that $\rk(\s)$ is a limit; otherwise, if $\rk(\s)>\beta$, then there are no~$k$ such that $\rk(\s\conc k)<\beta$. The key point, for the next lemma, is our observation early on that even in the limit case there will be only finitely many~$k$ such that $\rk(\s\conc k)<\beta<\rk(\s)$.

\begin{lemma} \label{lem:simple_forcing:beta_complete_is_dense}
  Let $\beta\le \alpha+1$. For all $p\in \QQ$ there is a $\beta$-complete~$p'$ extending~$p$.
\end{lemma}

\begin{proof}
  Do the obvious: define~$p'$ extending $p$ by letting $p'(\s\conc k)=0$ whenever $p(\s)=1$ and $\rk(\s\conc k)<\beta < \rk(\s)$. For each such~$k$, also define $p'(\s\conc k\conc l)=1$ for some large~$l$. As just discussed, $\dom p'$ is finite. By design, $p'\in \QQ$. If $\tau\in \dom p'\setminus \dom p$ and $p'(\tau)=1$ then $\rk(\tau)<\beta$; this implies that $p'$ is $\beta$-complete. 
\end{proof}

The following is the key technical lemma.

\begin{lemma} \label{lem:simple_forcing:untagging:extension_lemma}
  Let $0<\beta\le \alpha$. If $p\in \QQ$ is $\beta$-complete and~$r$ extends $ p\rest{(\beta+1)}$ then $p$ and $(r\rest \beta)$ are compatible in~$\QQ$. 
\end{lemma}

\begin{proof}
  By \cref{lem:simple_forcing:compatibility_is_being_a_function}, we need to show that $p\cup (r\rest\beta)$ is a function. Let $\s\in T_\alpha$, and suppose that $p(\s)\converge$ and $(r\rest\beta)(\s)\converge$; we need to show that $p(\s) = (r\rest\beta)(\s)$.

  First, suppose that $\rk(\s)\le \beta$. Then $(p\rest{(\beta+1)})(\s) = p(\s)$, and since $r\le p\rest{(\beta+1)}$, we have $r(\s)= p(\s)$; and $(r\rest\beta)(\s) = r(\s)$. 

  Suppose that $\rk(\s)>\beta$. Since $(r\rest\beta)(\s)\converge$, we must have $r(\s)=0$ and $r(\s\conc k)=1$ for some~$k$ such that $\rk(\s\conc k)<\beta$. We need to show that $p(\s)=0$ as well. But if $p(\s)=1$, then as $p$ is $\beta$-complete, we would have $p(\s\conc k)=0$ and as $r\le p\rest{\beta}$, we would have $r(\s\conc k)=0$, which is not the case. 
\end{proof}

The following is the ``untagging lemma''.

\begin{proposition} \label{prop:simple_forcing:Pi_untagging_lemma}
  Let $\gamma\le \alpha$, and let $\vphi$ be a $\Pi_\gamma$ Borel code. If $p\in \QQ$ is $\gamma$-complete and $p\force^*\vphi$ then $p\rest{(\gamma+1)}\force^* \vphi$. 
\end{proposition}

\begin{proof}
  The proposition is proved by induction on~$\gamma$. 

  \medskip

  The base case $\gamma = 0$ follows from the definition of strong forcing; $p\force^* r$ means $r\subseteq p$, and since $r$ is only defined on $\s\in \Leaves_\alpha$, this implies that $r\subseteq p\rest{1}$, so $p\rest{1}\force^* r$. (Note that every condition is 0-complete.)

  \medskip

  Suppose that $\gamma>0$, and that the proposition has been verified for all $\gamma'<\gamma$. Let $\vphi$ be a $\Pi_\gamma$ Borel code, and let $p\in \QQ$ be $\gamma$-complete. We prove the contrapositive: if $p\rest{(\gamma+1)}\nforce^* \vphi$ then $p\nforce^* \vphi$. 

  Suppose that $p\rest{(\gamma+1)}\nforce^* \vphi$. Since $\vphi$ is $\lnot \psi$ (where $\psi$ is $\Sigma_\gamma$), by definition, this means that there is some~$r$ extending $ p\rest{(\gamma+1)}$ such that $r\force^* \psi$. Now $\psi = \bigvee_n \psi_n$, and by definition, there is some~$n$ such that $r\force^* \psi_n$; and $\psi_n$ is $\Pi_{\gamma'}$ for some $\gamma'<\gamma$. By \cref{lem:basic_forcing:basic_properties_of_strong_forcing}(c) and \cref{lem:simple_forcing:beta_complete_is_dense}, we may assume that $r$ is $\gamma'$-complete. Hence, by induction, $r\rest{(\gamma'+1)}\force^* \psi_n$, so $r\rest{(\gamma'+1)}\force^* \psi$. By \cref{lem:simple_forcing:restricted_condition_is_a_condition}, $r\rest{\gamma}$ extends $r\rest{(\gamma'+1)}$, so by \cref{lem:basic_forcing:basic_properties_of_strong_forcing}(c) again, $r\rest{\gamma}\force^* \psi$. 

  By \cref{lem:simple_forcing:untagging:extension_lemma}, $q = p\cup (r\rest{\gamma})$ extends both~$p$ and $r\rest{\gamma}$. Since $q$ extends~$r\rest\gamma$, $q\force^* \psi$. Since~$q$ extends~$p$, by definition, $p$ does not strongly force~$\vphi$, as required. 
\end{proof}

\subsection{Interpretations of untagging}

We obtain a refinement of the forcing theorem (\cref{prop:simple_forcing:forcing_theorem}).

\begin{definition} \label{def:simple_forcing:Q_beta}
  Let $\beta \le \alpha+1$. We let
  \[
    \QQ_\beta = \left\{ p\rest\beta \,:\,  p\in \QQ \right\}. 
  \]
\end{definition}

In other words, $\QQ_\beta$ is the collection of all $p\in \QQ$ such that for all $\s\in \dom p$, either $\rk(\s)<\beta$, or $p(\s)=0$ and $p(\s\conc k)=1$ for some~$k$ with $\rk(\s\conc k)<\beta$. 

\begin{proposition} \label{prop:simple_forcing:Sigma_untagging_lemma}
  Suppose that $0<\beta \le \alpha+1$, and that~$\vphi$ is a $\Sigma_\beta$ code. For every sufficiently generic~$G$, 
  \[
    x_G \in [\vphi]\,\,\Iff \,\, (\exists p\in G\cap \QQ_\beta)\,\,p\force^*\vphi.
  \]
\end{proposition}

\begin{proof}
  For each $p\in \QQ$ and $\gamma \le \alpha+1$, let $E_{p,\gamma}$ be the collection of $q\in \QQ$ such that either 
  \begin{itemize}
    \item $q$ extends $p$ and~$q$ is $\gamma$-complete; or
    \item $q$ and~$p$ are incompatible. 
  \end{itemize}
  By \cref{lem:simple_forcing:beta_complete_is_dense}, each $E_{p,\gamma}$ is dense. Let $\+E = \left\{ E_{p,\gamma} \,:\,  p\in \QQ \andd \gamma \le \alpha+1 \right\}$. By \cref{prop:simple_forcing:forcing_theorem}, let $\+D$ be a countable collection of dense subsets of~$\QQ$ such that for all $\+D$-generic~$G$, $x_G\in [\vphi]$ if and only if $p\force^* \vphi$ for some $p\in G$. We claim that if $G$ is $\+E\cup \+D$-generic then the equivalence above holds. 

  In the direction which is not immediate, suppose that $x_G\in [\vphi]$. Let $p\in G$ be such that $p\force^* \vphi$. Write $\vphi = \bigvee \vphi_n$. By definition of strong forcing, $p\force^* \vphi_n$ for some~$n$. Let $\gamma<\beta$ such that $\vphi_n$ is $\Pi_\gamma$. Since $G$ is $\+E$-generic, find some $q\le p$ in~$G$ that is $\gamma$-complete. By \cref{prop:simple_forcing:Pi_untagging_lemma}, $q\rest{(\gamma+1)}$ strongly forces~$\vphi_n$, and so strongly forces~$\vphi$. Since $\gamma<\beta$, $q\rest\beta$ strongly forces~$\vphi$. Since $q$ extends $q\rest{\beta}$, $q\rest{\beta}\in G$, and so is the desired condition. 
\end{proof}

We interpret these results in the language of topology.

\begin{definition} \label{def:simple_forcing:tau_beta}
  We let  $\tau_\beta = \tau_{\QQ,\beta}$ denote the topology generated by $[p]$ for $p\in \QQ_{\beta}$.
\end{definition}

Hence, $\tau_{1}$ is the standard topology on $2^{\Leaves_\alpha}$, and $\tau_{\alpha+1}$ is~$\tau_\QQ$.  Note that by \cref{lem:the_complexity_of_a_label_of_a_node}, the generating sets of the topology $\tau_{\beta}$ are all finite Boolean combinations of $\Sigma^0_{<\beta}$ and $\Pi^0_{<\beta}$ sets, and so are all $\Delta^0_\beta$ sets; so each $\tau_{\beta}$-open set is~$\bSigma^0_\beta$. 

Recall that we let $U_\vphi = \bigcup \left\{ [p] \,:\,  p\force^* \vphi \right\}$. 

\begin{proposition} \label{prop:simple_forcing:untagging_lemma_topologically}
  Let $0<\beta \le \alpha+1$. For any $\Sigma_\beta$ Borel code~$\vphi$, the set $U_\vphi$ is $\tau_{\beta}$-open, and is equivalent to~$[\vphi]$ modulo a $\tau_\beta$-meagre set. 
\end{proposition}

\begin{proof}
Fix nonzero $\beta\le \alpha+1$. We can repeat the development of the beginning of this section, with $\QQ_\beta$ replacing~$\QQ$:
\begin{itemize}
  \item We define the notion of a filter $G\subset \QQ_\beta$, the notion of a dense subset of~$\QQ_\beta$, and of a $\+D$-generic filter, where~$\+D$ is a countable family of dense subsets of~$\QQ_\beta$ (\cref{def:simple_forcing:generic_filters}). 

  \item For a sufficiently generic filter $G\subset \QQ_\beta$, we define $x_G$ as in \cref{def:simple_forcing:x_G}. The analogue of \cref{lem:simple_forcing:filters_are_points} holds for~$\QQ_\beta$, with the same proof, except that $\bigcup G$ is the restriction of $T^{x_G}$ to~$\s$ of rank $<\beta$, and parents of 1-labelled nodes of rank $<\beta$. 

  \item The Rasiowa-Sikorski lemma (\cref{lem:simple_forcing:Rasiowa-Sikorski}) holds for any notion of forcing, including $\QQ_\beta$. This gives the analogues of \cref{lem:simple_forcing:dense_open_and_comeagre} and \cref{prop:simple_forcing:Baire_category_theorem} for the topology~$\tau_\beta$. 

  \item We define the strong forcing relation, where for negation, we only search for extensions in~$\QQ_\beta$. Denote this notion by $\force^*_\beta$. The forcing theorem, \cref{prop:simple_forcing:forcing_theorem}, holds for~$\QQ_\beta$ as well, with the same proof. 
\end{itemize}

Hence, $[\vphi]$ is equivalent to the $\tau_\beta$-open set $\bigcup \{ [p] \,:\,  p\force^*_\beta \vphi \}$ modulo a $\tau_\beta$-meagre set. The real content of \cref{prop:simple_forcing:untagging_lemma_topologically} is that this open set is the same as the one that we get when we consider~$\tau_\QQ$. That is, if $\vphi$ is $\Sigma_\beta$ then
  
   \begin{description}
      \item[$(*)$] $U_\vphi = \bigcup \big\{ [p] \,:\,  p\in \QQ_\beta\andd p\force^*_\beta  \vphi \big\}. $
   \end{description}
  
To see this, we observe that the argument proving \cref{prop:simple_forcing:Pi_untagging_lemma} gives:
\begin{itemize}
  \item If $\vphi$ is $\Sigma_\beta$, then for all $p\in \QQ_\beta$, $p\force^* \vphi$ if and only if $p\force^*_\beta \vphi$. 
\end{itemize}
Now~$(*)$ follows by an argument similar to that giving \cref{prop:simple_forcing:Sigma_untagging_lemma}. Suppose that $x\in U_\vphi$; we need to show that there is some $r\in\QQ_\beta$ such that $r\force^* \vphi$ and $x\in [r]$. Let $p\in \QQ$ such that $p\force^* \vphi$ and $x\in [p]$. Then $p\force^*\vphi_n$ for some~$n$, and $\vphi_n$ is $\Pi_\gamma$ for some $\gamma<\beta$. There is some $q\le p$ such that $x\in [q]$ and $q$ is $\gamma$-complete (define~$q$ as in the proof of \cref{lem:simple_forcing:beta_complete_is_dense}, adding witnesses according to $T^x$). Now $r = q\rest{\beta}$ (which extends $q\rest{(\gamma+1)}$) is as required. 
\end{proof}

\subsection{M\'atrai's result}

The work above is closely related to a result of M\'atrai \cite{Matrai}:

\begin{theorem}[M\'atrai] \label{thm:Matrai}
  Let $\alpha \ge 1$ be a countable ordinal. There is a $\bPi^0_\alpha$ set $P\subseteq 2^\w$ and a Polish topology $\tau_\alpha$ on $2^\w$ such that:
  \begin{orderedlist}
    \item $\tau_\alpha$ is finer than the standard topology on $2^\w$; 

    \item $P$ is $\tau_\alpha$-closed and $\tau_\alpha$-nowhere dense;

    \item If $B$ is a basic $\tau_\alpha$-open set meeting $P$, $D\subseteq 2^\w$ is $\bPi^0_{<\alpha}$ (in the standard topology), and $P\cap B \cap D$ is comeagre in $(P\cap B,\tau_\alpha)$, then there is a $\tau_\alpha$-open set $B'$ such that $P\cap B = P\cap B'$ and $D\cap B'$ is comeagre in $(B',\tau_\alpha)$. 
  \end{orderedlist}
\end{theorem}

\begin{proof}
  We work in $2^{\Leaves_\alpha}$, which, as discussed above, is computably isomorphic to Cantor space. We let $\tau_\alpha = \tau_{\QQ,\alpha}$ defined above (rather than using $\tau_\QQ = \tau_{\alpha+1}$). Note that $\QQ_\alpha$ is the collection of $p\in \QQ$ such that it is not the case that $p(\rooot)=1$.  We let 
  \[
    P = \left\{ x\in 2^{\Leaves_\alpha} \,:\, T^x(\rooot) = 1  \right\}, 
  \]
  which is $\Pi^0_\alpha$ by \cref{lem:the_complexity_of_a_label_of_a_node}. The set $P$ is $\tau_\alpha$-closed since it equals the intersection of the sets 
  \[
   \left\{ x\in 2^{\Leaves_\alpha}  \,:\,  T^x(\smallseq{k}) = 0   \right\},
  \]
  each of which is $\tau_\alpha$-clopen. The set~$P$ is $\tau_\alpha$-nowhere dense since any $p\in \QQ_\alpha$ can be extended to some $p'\in \QQ$ with $p'(\rooot)=0$; so every $\tau_\alpha$-open set intersects the complement of~$P$. Note that for $q\in \QQ_\alpha$, $[q]$ meets~$P$ if and only if $q(\smallseq{k})=0$ whenever $q(\smallseq{k})\converge$. 

  Let $B = [p]$ for $p\in \QQ_\alpha$ be a basic $\tau_\alpha$-open set that meets~$P$. Let $D$ be $\bPi^0_{<\alpha}$; let $\vphi$ be a $\Pi_{<\alpha}$ Borel code of~$D$ (that is, $D = [\vphi]$). We suppose that $P\cap B \cap D$ is comeagre in $(P\cap B,\tau_\alpha)$. This means that for any $q\in \QQ_\alpha$ extending~$p$, if $[q]$ intersects~$P$ then there is some $r$ extending $q$ in $\QQ_\alpha$ which also meets~$P$ and strongly forces~$\vphi$. Then 
  \[
    B' = [p] \cup \bigcup \left\{[q]\,:\,  q\in \QQ_\alpha, q\le p, [q]\cap P = \emptyset, \text{ and } q\force^*  \vphi \right\}
  \]
  is as required. 
\end{proof}

\section{A weak dichotomy} \label{sec:a_weak_dichotomy}

Let $\alpha\ge 1$ be a computable ordinal. We let 
\[
\XXX_\alpha = 2^{\Leaves_\alpha} \times \PowerSet(\Nat)
\] (both factors can be identified with Cantor space). 

\begin{definition} \label{def:K_alpha}
  The directed graph $K_\alpha$ on $\XXX_\alpha$ is defined as follows: a pair $(x,y)$ is connected to a pair $(x',y')$ if $x'=x$, $T^x(\rooot)=0$, and for the least~$n$ such that $T^x(\seq{n})=1$, we have $y\symdiff y' = \left\{ n   \right\}$, and $n\notin y$.
\end{definition}

The directed graph $K_\alpha$ is Borel, indeed the collection of edges is $\Sigma^0_\alpha$. Note that~$K_\alpha$ is the graph of a bijection between two disjoint~$\Sigma^0_\alpha$ sets.

\begin{theorem} \label{prop:dichotomy_alpha_alpha}
    Let $G$ be a $\Sigma^1_1$ directed graph on a computably presented Polish space~$Y$. The following are equivalent:
    \begin{equivalent}
      \item \label{item:dichotomy:Delta11} 
      There is a homomorphism $f\colon (\XXX_\alpha,K_\alpha)\to (Y,G)$ which is continuous when both spaces are equipped with the topology generated by the $\Sigma^0_\alpha(\Delta^1_1)$ sets.\footnote{That is, sets that are $\Sigma^0_\alpha$ relative to some $\Delta^1_1$ parameter.} 

      
      \item \label{item:dichotomy:boldface} 
      There is a homomorphism $f\colon (\XXX_\alpha,K_\alpha)\to (Y,G)$ such that for every $A\in \Sigma^0_\alpha(\Delta^1_1)$, $f^{-1}(A)$ is $\bSigma^0_\alpha$.

      \item \label{item:dichotomy:no_colouring} 
      There is no countable $\bSigma^0_\alpha$ colouring of~$G$.  
    \end{equivalent}
    For $\alpha=1$, we need to require that~$Y$ be 0-dimensional. 
\end{theorem}

We remark that in conditions \ref{item:dichotomy:Delta11} and~\ref{item:dichotomy:boldface}, we can replace the topology generated by the $\Sigma^0_\alpha(\Delta^1_1)$ sets by the topology generated by sets that are both $\Sigma^1_1$ and $\bPi^0_{<\alpha}$. 

Most of the work is in the following:

\begin{theorem} \label{prop:K_alpha:no_colouring}
  There is no countable $\bSigma^0_\alpha$ colouring of~$K_\alpha$.  
\end{theorem}

Given this result, we can prove the weak dichotomy theorem.

\begin{proof}[Proof of \cref{prop:dichotomy_alpha_alpha}]
    \ref{item:dichotomy:Delta11}$\Then$\ref{item:dichotomy:boldface} %
    is immediate. 

    \ref{item:dichotomy:boldface}$\Then$\ref{item:dichotomy:no_colouring} follows from \cref{prop:K_alpha:no_colouring}, and \cite[Theorem~2.1]{LecomteZeleny}, which shows that if \ref{item:dichotomy:no_colouring} fails, then there is a $\Sigma^0_\alpha(\Delta^1_1)$ colouring of~$G$. 

    For the rest, suppose that~\ref{item:dichotomy:no_colouring} holds. By Proposition~2.4 of \cite{LecomteZeleny}, there is some point $p\in Y$ which is an accumulation point of edges of~$G$ in the $\Sigma^0_\alpha(\Delta^1_1)$ topology: say $a_n,b_n\to p$ and $(a_n,b_n)\in G$. Define the following map~$F$ from $\XXX_\alpha$ to~$Y$: for $(x,y)\in \XXX_\alpha$
  \begin{itemize}
    \item If $T^x(\rooot)= 1$ then $F(x,y)=p$. 
    \item If $T^x(\rooot)=0$, let~$n$ be least such that $T^x(\seq{n})=1$.
    \begin{itemize}
      \item If $n\notin y$, let $F(x,y)=a_n$. 
      \item If $n\in y$, let $F(x,y)=b_n$. 
    \end{itemize}
  \end{itemize}
  The range of~$F$ is a countable set, and for any $\Sigma^0_\alpha(\Delta^1_1)$-open set $A$, $F^{-1}[A]$ has one of two forms. For $n<\w$ and $i<2$ let  
   \[
   B_{n,i} = \left\{ (x,y) \,:\,  n\text{ least s.t.\ }T^x(\seq{n})=1 \andd y(n)=i \right\}.
   \]
   Each set $B_{n,i}$ is $\Pi^0_{<\alpha}(\Delta^1_1)$. 
  
  \begin{itemize}
    \item If $p\notin A$, then $F^{-1}[A]$ is the union of sets among $B_{n,i}$. 
    \item Otherwise, $F^{-1}[A]$ is the complement of the union of finitely many of the $B_{n,i}$. 
  \end{itemize}
  Thus, $F^{-1}[A]$ is $\Sigma^0_\alpha(\Delta^1_1)$ whenever~$A$ is $\Sigma^0_\alpha(\Delta^1_1)$. 
\end{proof}

It remains to prove \cref{prop:K_alpha:no_colouring}. The proof is an elaboration on the forcing argument of the previous section. We force with $\QQ\times$Cohen.

\begin{definition} \label{def:K_alpha_forcing_notion}
   We let $\PP = \QQ\times 2^{<\w}$. For $p = (u,\zeta)$ and $q = (v,\xi)$ in~$\PP$, we write $q\le p$ ($q$ extends~$p$) if $u\subseteq v$ and $\zeta\preceq \xi$. 
\end{definition}

When $p\in \PP$ we often write $p = (u^p,\zeta^p)$. For $\zeta\in 2^{<\w}$ we let $[\zeta] = \left\{ y\in 2^\w \,:\, \zeta\prec y   \right\}$, and for $p\in \PP$ we let
\[
  [p] = [u^p]\times [\zeta^p] \subseteq \XXX_\alpha.
\]

The notions of a filter of~$\PP$ and a dense subset of~$\PP$ are defined as before. For a filter~$G$ of~$\PP$ we define
\[
  x_G = \bigcup \left\{ u^p\rest{\Leaves_\alpha} \,:\,  p\in G \right\}
\]
and
\[
  y_G = \bigcup \left\{ \zeta^p \,:\,  p\in G \right\}.
\]
The proof of \cref{lem:simple_forcing:filters_are_points} gives its analogue for~$\PP$: there is a countable collection ${\Dpoint}(\PP)$ of dense subsets of~$\PP$, such that for any ${\Dpoint}(\PP)$-generic filter~$G$, $(x_G,y_G)\in \XXX_\alpha$, and for all $p\in \PP$, 
\[
  (x_G,y_G)\in [p]\,\,\Iff\,\,p\in G.\footnote{To be specific, ${\Dpoint}(\PP)$ consists of the sets $\left\{ p \,:\,  \s\in \dom u^p \right\}$ for $\s\in T_\alpha$; the sets $\left\{ p \,:\,  |\zeta^p|\ge m \right\}$ for all $m\in \Nat$; and the sets $\left\{ q \,:\,  q\le p, \text{ or $q$ and~$p$ are incompatible} \right\}$ for $p\in \PP$.}
\]
We obtain the same characterisation of dense open and comeagre subsets of $\XXX_\alpha$, equipped with the product topology $\tau_\PP$ (the product of $\tau_\QQ$ and the usual topology on $\PowerSet(\Nat)$). 

We define Borel codes for subsets of $\XXX_\alpha$ analogously to the definition above; the only difference is in the clopen level, where a code is a pair $p = (u,\zeta)\in \PP$ such that $\dom u\subset \Leaves_\alpha$, and its interpretation is~$[p]$ as defined above. The strong forcing relation $p\force^* \vphi$ is defined exactly as in \cref{def:simple_forcing:strong_forcing_relation}. \Cref{lem:basic_forcing:basic_properties_of_strong_forcing} holds. 
The forcing theorem, \cref{prop:simple_forcing:forcing_theorem}, holds, with the same proof. 

Fix some~$\beta$ with $0<\beta \le \alpha+1$. We define 
\[
  \PP_\beta = \QQ_\beta \times 2^{<\w},
\]
and for $p\in \PP$, we let 
\[
  p\rest{\beta} = (u^p\rest{\beta}, \zeta^p). 
\]
Immediately from \cref{lem:simple_forcing:restricted_condition_is_a_condition} we get its analogue for~$\PP$: $p\rest{\beta}\in \PP_\beta$ and $p \le p\rest{\beta}$. We say that~$p$ is \emph{$\beta$-complete} if $u^p$ is $\beta$-complete (\cref{def:simple_forcing:beta-complete}). \Cref{lem:simple_forcing:beta_complete_is_dense} implies its analogue for~$\PP$. We get an analogue of \cref{lem:simple_forcing:untagging:extension_lemma}:

\begin{lemma} \label{lem:K_alpha:untagging:extension_lemma}
  Suppose that $p\in \PP$ is $\beta$-complete and $r$ extends $p\rest{(\beta+1)}$. Then~$p$ and $(r\rest \beta)$ are compatible in~$\PP$. 
\end{lemma}

\begin{proof}
  By \cref{lem:simple_forcing:untagging:extension_lemma}, there is some $v\in \QQ$ that extends both $u^p$ and $u^r\rest{\beta} = u^{r\rest\beta}$ (namely $v=u^p\cup (u^{r}\rest{\beta})$). Then $(v,\zeta^r)$ extends both~$p$ and $r\rest\beta$ (note that $\zeta^r\succeq \zeta^p$ follows from $r\le p\rest{(\beta+1)}$). 
\end{proof}

Finally, we obtain the untagging lemma for~$\PP$. The proofs are identical, and so we obtain:

\begin{lemma} \label{lem:K_alpha:Sigma_untagging_lemma}
Suppose that $0<\beta \le \alpha+1$, and that~$\vphi$ is a $\Sigma_\beta$ code for a subset of $\XXX_\alpha$. For every sufficiently generic $G\subset \PP$, 
  \[
    (x_G,y_G) \in [\vphi]\,\,\Iff \,\, (\exists p\in G\cap \PP_\beta)\,\,\,p\force^*\vphi.
  \]
\end{lemma}

\begin{proof}[Proof of \cref{prop:K_alpha:no_colouring}]
  Let $\+C$ be a countable collection of $\bSigma^0_\alpha$ sets. Let $\+D$ be a countable collection of dense subsets of~$\PP$, so that for every set $C\in \+C$ there is a $\Sigma_\alpha$ code~$\vphi$ of~$C$ such that for any $\+D$-generic filter $G\subset \PP$, 
  \[
    (x_G,y_G) \in C\,\,\Iff \,\, (\exists p\in G\cap \PP_\alpha)\,\,\,p\force^*\vphi.
  \]

  Let $p^*$ be the condition defined by $\zeta^{p^*}$ being the empty string, $\dom u^{p^*}= \{ \rooot \}$, and $u^{p^*}(\rooot)=1$. By the Rasiowa-Sikorski lemma, let~$G$ be a $\+D$-generic filter containing~$p^*$. 

  We claim that for every set $C\in \+C$, if $(x_G,y_G)\in C$ then there are two points in~$C$ connected by an edge of~$K_\alpha$. Thus, $\+C$ cannot be a partition of $\XXX_\alpha$ into $K_\alpha$-independent sets. Note though that $(x_G,y_G)$ itself is not part of an edge of~$K_\alpha$.

  Let $C\in \+C$ and suppose that $(x_G,y_G)\in C$. Let $\vphi$ be a $\Sigma_\alpha$ code for~$C$ such that there is some $p\in G\cap \PP_\alpha$ that strongly forces~$\vphi$. Since $p$ is compatible with~$p^*$,  $u^{p}(\seq{k})=0$ whenever defined. Since $p\in \PP_\alpha$, it is not the case that $u^p(\rooot)=1$.

  Let~$n$ be large, so $n > |\zeta^p|$ and every $\tau\in \dom u^{p}$ 
  extends $\smallseq{k}$ for some $k<n$. Define $p'$ extending $p$ by 
  setting $u^{p'}(\rooot)=0$, $u^{p'}(\seq{k})=0$ for all $k<n$, and $u^{p'}(\seq{n})=1$. Now define two conditions $q_0$ and~$r_0$, both extending~$p'$, 
  by setting $|\zeta^{q_0}| = |\zeta^{r_0}|=n+1$, $\zeta^{q_0}\rest n = \zeta^{r_0}\rest n$, 
  and $\zeta^{q_0}(n)=0$ while $\zeta^{r_0}(n)=1$; $u^{q_0} = u^{r_0} = u^{p'}$. 

  We now define two filters $Q$ and~$R$ starting with $q_0\in Q$ and $r_0\in R$. They will both be $\+D$-generic, but not mutually so. Indeed, we will let~$Q$ be the upward closure of a decreasing sequence $(q_\ell)$ of conditions, and~$R$ be the upward closure of a sequence~$(r_\ell)$ of conditions, and for each~$\ell$ we will ensure:
  \begin{itemize}
    \item $u^{q_\ell} = u^{r_\ell}$; and
    \item $|\zeta^{q_\ell}| = |\zeta^{r_\ell}|$, and for all $m<|\zeta^{r_\ell}|$ other than~$n$, $\zeta^{q_\ell}(m) = \zeta^{r_\ell}(m)$. 
  \end{itemize}
  This is done by an interleaving construction: suppose that $q_\ell$ and $r_\ell$ have been determined. We first extend~$q_\ell$ to some $q'_\ell$ meeting the $\ell\tth$ dense set in~$\+D$.  We let $r'_\ell$ extend~$r_\ell$ by copying over the new values of~$q'_\ell$. We extend~$r'_\ell$ to $r_{\ell+1}$ in the same dense set, and then copy the new 
  values of $r_{\ell+1}$ to define $q_{\ell+1}$. This construction ensures that $(x_Q,y_Q)$ is connected by an edge to $(x_R,y_R)$. Both points lie in~$C$, as both filters contain $p$. 
\end{proof}

\subsection*{Using topological language}

As discussed, one of our aims is to help bridge the gap between practitioners more comfortable with forcing, and ones more comfortable with topology. We therefore give a translation of the proof above of \cref{prop:K_alpha:no_colouring}, using topological notions. 

The main idea of the proof is passing between the two topologies, $\tau_\alpha$ and $\tau_\PP =\tau_{\alpha+1}$. We abuse notation by allowing $\tau_\alpha$ to refer both to the topology on $2^{\Leaves_\alpha}$ and on $\XXX_\alpha$ (using the product with the usual topology on $\PowerSet(\Nat)$). As above, let 
\[
  P = \left\{ x\in 2^{\Leaves_\alpha} \,:\, T^x(\rooot) = 1   \right\}. 
\]
As discussed in the proof of \cref{thm:Matrai}, according to $\tau_\alpha$, $P$ is closed and nowhere dense, whereas it is clopen according to $\tau_{\alpha+1}$. 

Suppose that $\+C$ is a partition of $\XXX_\alpha$ into $\bSigma^0_\alpha$ sets. Since~$P$ is $\tau_{\alpha+1}$-clopen, there is some $C\in \+C$ such that $C\cap (P\times \PowerSet(\Nat))$ is $\tau_{\alpha+1}$-nonmeagre. Let $\tilde p\in \PP$ such that $[u^{\tilde p}]\subseteq P$ and $C$ is $\tau_{\alpha+1}$-comeagre in~$[\tilde p]$. 

The untagging lemma for~$\PP$ ensures that there is some $p\in \PP_\alpha$ such that $[u^p]\cap P\ne \emptyset$, and such that~$C$ is $\tau_\alpha$-comeagre in~$[p]$. Since $[u^p]\cap P\ne \emptyset$, 
we can extend~$p$ to $q_0$ and~$r_0$ as in the proof of \cref{prop:K_alpha:no_colouring}. Letting~$n$ be the same as in that proof, the map $(x,y)\mapsto (x,y\cup \{n\})$ is a homeomorphism between $[q_0]$ and $[r_0]$. Since~$C$ is $\tau_\alpha$-comeagre in both $[q_0]$ and $[r_0]$, there is some $(x,y)\in [q_0]\cap C$ such that $(x,y\cup \{n\})\in C$ as well, showing that~$C$ contains an edge of~$K_\alpha$.

\section{``Smaller'' non-colourable graphs} \label{sec:_smaller_non_colourable_graphs}

The graph~$K_3$ is not a least graph with no $\bSigma^0_3$ colouring, with respect to continuous homomorphisms. To explain why, we define a family of graphs $L_\alpha$ for $\alpha\ge 3$ that are ``smaller'' in some sense than the graphs $K_\alpha$. We will show that $L_\alpha$ has no countable $\bSigma^0_\alpha$ colouring, and that there is no continuous homomorphism from $K_3$ to $L_3$.

\subsection{The graphs $L_\alpha$}

Fix $\alpha\ge 3$. This implies that we may assume that no $\s\in T_\alpha$ of height $\le 2$ is a leaf of~$T_\alpha$.\footnote{We may assume that for all limit $\delta \le\alpha$, every element $\delta_k$ of the cofinal sequence given by the computable presentation of~$\alpha$, is a successor of a successor ordinal. Hence, for example, if $\alpha$ is a limit ordinal, then none of the children of the root are leaves, nor can nodes of height~2 be leaves.}

Below, to avoid excess notation, for $x\in 2^{\Leaves_\alpha}$, we write $T^x(m)$ instead of $T^x(\smallseq{m})$, $T^x(m,k)$ instead of $T^x(\smallseq{m,k})$, etc. We similarly write $\rk(m)$, $\rk(m,k)$, and so on. 

\begin{definition} \label{def:k_x_m}
  Let $x\in 2^{\Leaves_\alpha}$. 
  \begin{sublemma}
    \item We let $n^x = \sup \left\{ n \,:\,  (\forall m<n) \,\,\,T^x({m})=0 \right\}$. 
    \item For $m<n^x$ we let $k^x(m)$ be the least~$k$ such that $T^x({m,k})=1$.
  \end{sublemma}
\end{definition}

Thus, if $T^x(\rooot)=0$, then $n^x$ is the least~$n$ such that $T^x({n})=1$, so $(x,y)\sim (x,y')$ in $K_\alpha$ exactly when $y\symdiff y' = \{n^x\}$. If $T^x(\rooot)=1$ then $n^x = \w$. 

\begin{definition} \label{def:M_k}
  We let $(M_k)$ be a uniformly computable and dense list of elements of Cantor space.
\end{definition}

We think of each number~$k$ as a ``code'' of every finite initial segment of~$M_k$. 

\begin{definition} \label{def:W_alpha_and_L_alpha} 
We let $\WWW_\alpha$ be the collection of $(x,y)\in \XXX_\alpha$ such that for all $m<n^x$, $k^x(m)$ codes $y\rest{(m+1)}$ (meaning that $y\rest{(m+1)}\prec M_{k^x(m)}$).

We let $L_\alpha$ be the restriction of $K_\alpha$ to $\WWW_\alpha$. 
\end{definition}

\begin{definition} \label{def:alpha_star}
  We let $\alpha^* = \alpha-2$. More precisely, 
  \[
    \alpha^* = \begin{cases*}
      \alpha -2, & if $\alpha$ is the successor of a successor; \\
      \alpha -1, & if $\alpha$ is the successor of a limit ordinal; \\
      \alpha, & if $\alpha$ is a limit ordinal. 
    \end{cases*}
  \]
\end{definition}

\begin{proposition} \label{prop:complexity_of_WWW_alpha}
  $\WWW_\alpha$ is $\Pi^0_2(\tau_{\alpha^*})$. 
\end{proposition}

\begin{proof}
  For all~$m$ and~$k$, the collection of $x\in 2^\Leaves$ such that $T^x(m,k)=0$ is $\Sigma^0_1(\tau_{\alpha^*})$. For each finite tuple $\bar k = (k_0, k_1,\dots, k_{m})$ of natural numbers, let $A_{\bar k}$ be the collection of $(x,y)\in \XXX_\alpha$ such that:
  \begin{itemize}
    \item there is some $\ell\le m$ such that $T^x(\ell,k_\ell)=0$; or
    \item there is some $\ell\le m$ and some $k<k_\ell$ such that $T^x(\ell, k)=1$; or
    \item for all $\ell \le m$, $k_\ell$ is a code of $y\rest{(\ell+1)}$. 
  \end{itemize}
  then $A_{\bar k}$ is $\Delta^0_2(\tau_{\alpha^*})$, and $\WWW_\alpha$ is the intersection of all the sets $A_{\bar k}$. 
\end{proof}

In particular, $\WWW_3$ is a $\Pi^0_2$ set, and so equipped with the subspace topology is Polish.

\subsection{Non-colourability of $L_\alpha$}

We cannot use the notion of forcing $\PP$ above to always obtain points in $\WWW_\alpha$; indeed, $\WWW_\alpha\cap \left\{ (x,y) \,:\,  T^x(\rooot)=1 \right\}$ is meagre in the $\tau_{\alpha+1}$ topology. 

\begin{definition} \label{def:k_u_and_n_u}
  For $u\in \QQ$ we let 
  \[
     n^u = \max \left\{ n \,:\,  (\forall m<n)\,\,u({m})=0 \right\}.
  \]
  For each $m<n^u$ we let 
  \[
  k^u(m) = \min \left\{ k \,:\,  u({m,k})=1 \right\}. 
\]
\end{definition}

Here we use the same notation as above: $u(m) = u(\seq{m})$, $u(m,k) = u(\seq{m,k})$, etc.

\begin{definition} \label{def:the_notion_of_forcing_tilde_Q}
  We let $\tilde \QQ$ be the collection of $u\in \QQ$ satisfying:
  \begin{orderedlist}
    \item if $u({n^u})\diverge$ then for all $m>n^u$, $u({m})\diverge$; 
    \item There is no $m\ge n^u$ such that $u({m})=0$;
    \item For all $m<n^u$, for all $k<k^u(m)$, $u({m,k})\converge$; 
    \item For all $m<m'<n^u$, $\rk({m,k^u(m)}) \le \rk({m',k^u(m')})$. 
  \end{orderedlist}
\end{definition}

Note that requirement~(iv) in the definition is only relevant when~$\alpha$ is the successor of a limit ordinal; otherwise, it holds automatically. Note that for $u\in \tilde \QQ$, $u(n^u)\converge$ if and only if $u(\rooot)=0$. If $u(n^u)\diverge$ we can have either $u(\rooot)\diverge$ or $u(\rooot)=1$.

\begin{lemma} \label{lem:tilde_Q_is_closed_under_unions}
  Suppose that $u,v\in \tilde \QQ$ and $u\cup v$ is a function; then $u\cup v\in \tilde \QQ$. 
\end{lemma}

\begin{proof}
  \Cref{lem:simple_forcing:compatibility_is_being_a_function} implies that $u\cup v\in \QQ$; the conditions for being in~$\tilde \QQ$ are easily verified. 
\end{proof}

\begin{lemma} \label{lem:restrictions_of_conditions_in_tilde_Q}
  Let $u\in \tilde \QQ$. For all $\beta \le \alpha+1$, $u\rest{\beta}\in \tilde \QQ$; $n^{u\rest{\beta}}\le n^u$, and for all $m<n^{u\rest\beta}$, $k^{u\rest\beta}(m) = k^u(m)$. 
\end{lemma}

\begin{proof}
  We use the fact that for all~$\s$ and $k<k'$, $\rk(\s\conc k)\le \rk(\s\conc k')$. This implies that for all $m<n^u$, $(u\rest{\beta})({m})\converge$ if and only if $\rk({m,k^u(m)})<\beta$.\footnote{In detail: if $\rk(m,k^u(m))<\beta$ then since $u(m,k^u(m))=1$, we have $(u\rest{\beta})(m)\converge$. If $(u\rest{\beta})(m)]\converge$, then either $\rk(m)<\beta$, in which case $\rk(m,k)<\beta$ for all~$k$; or there is some~$k$ such that $\rk(m,k)<\beta$ and $u(m,k)=1$. The minimality of $k^u(m)$ implies $k^u(m)\le k$, so $\rk(m,k^u(m))\le \rk(m,k)<\beta$.} Requirement~(iv) implies that this is an initial segment of $m<n^u$. If $n^{u\rest{\beta}}<n^u$ then for all $m\ge n^{u}$, $(u\rest{\beta})(m)\diverge$, since $\rk(m)\ge\beta$ and $u(m)\ne 0$. 
\end{proof}

\begin{definition} \label{def:notion_of_forcing_SSS}
  We let $\SSS$ be the collection of pairs $(u,\zeta)\in \tilde \QQ\times 2^{<\w}$ such that $|\zeta|\ge n^u$, and for all $m<n^u$, $k^u(m)$ is a code of $\zeta\rest{(m+1)}$. 
\end{definition}

The collection of conditions~$\SSS$ is partially ordered by co-ordinatewise extension, just like~$\PP$ (it is a sub-ordering of~$\PP$).

\begin{lemma} \label{lem:forcing_S:extension_lemma} \  
  \begin{sublemma}
    \item \label{item:forcing_S:extension_lemma:starting_condition}
    The condition~$p^*$ defined by $\dom u^{p^*} = \{ \rooot \}$, $u^{p^*}(\rooot)=1$, and $\zeta^{p^*} = \seq{}$ is in~$\SSS$.
    \end{sublemma}

Suppose that $(u,\zeta)\in\SSS$. 

    \begin{sublemma}[resume]

    \item \label{item:forcing_S:extension_lemma:extending_zeta}
    For all $\xi\succeq \zeta$, $(u,\xi)\in \SSS$. 

    \item \label{item:forcing_S:extension_lemma:extending_a_1}
    Suppose that $u(n^u)= 1$.  Then for all $v\supseteq u$ in~$\tilde\QQ$,  $(v,\zeta)\in \SSS$.

    \item \label{item:forcing_S:extension_lemma:adding_a_1}
    Suppose that $u(\rooot)\diverge$. Let~$v$ extend~$u$ by defining $v(n^u)=1$ (and $v(\rooot)=0$). Then $(v,\zeta)\in \SSS$. 

    \item \label{item:forcing_S:extension_lemma:extending_a_void}
    Suppose that $u(n^u)\diverge$. Then there is some $v\supseteq u$ in $\tilde \QQ$ such that $n^v = |\zeta|$, $v(n^v)\diverge$, and $(v,\zeta)\in \SSS$. 
  \end{sublemma}
\end{lemma}

\begin{proof}
  For~\ref{item:forcing_S:extension_lemma:extending_a_1}, the point is that $n^v = n^u$, so there are no further coding requirements on $(v,\xi)$. For~\ref{item:forcing_S:extension_lemma:adding_a_1}, note that $u(\rooot)\diverge$ implies $u(n^u)\diverge$; we again get $n^v = n^u$. 

  For~\ref{item:forcing_S:extension_lemma:extending_a_void}: let $l = |\zeta|$. For each $m$ with $n^u\le m <l$, choose some~$k_m$ that codes~$\zeta$, and such that for $m < n^u \le m' < m'' < l$ we have $\rk(m,k^u(m))\le \rk(m',k_{m'})\le \rk(m'',k_{m''})$. Extend~$u$ to~$v$ by setting, for all~$m$ with $n^u\le m <l$:
  \begin{itemize}
    \item $v(m)=0$, and $v(m,k_m)=1$; and
    \item for all $k<k_m$, $v(m,k)=0$, and $v(m,k,\ell)=1$ for some large~$\ell$.
  \end{itemize}
  Then~$v$ is as required. 
\end{proof}

\begin{lemma} \label{lem:forcing_S:totality}
  For all $\s\in T_\alpha$, the collection of $q\in \SSS$ such that $u^q(\s)\converge$ is dense in~$\SSS$. 
\end{lemma}

\begin{proof}
  Let $p = (u,\zeta)\in \SSS$ and let $\s$ such that $u(\s)\diverge$. If $\s = \rooot$, then \ref{item:forcing_S:extension_lemma:adding_a_1} of \cref{lem:forcing_S:extension_lemma} allows us to extend~$p$ to $(v,\zeta)\in \SSS$ with $v(\rooot)\converge$. Suppose that $\s\ne \rooot$. 

  If $u(n^u)=1$, then \ref{item:forcing_S:extension_lemma:extending_a_1} allows us to easily extend~$p$ to $(v,\zeta)\in \SSS$ with $v(\s)\converge$. Suppose that $u(n^u)\diverge$. Let~$m$ such that $\smallseq{m}\preceq \s$. By \ref{item:forcing_S:extension_lemma:extending_zeta}, we may assume that $|\zeta|>m$; by \ref{item:forcing_S:extension_lemma:extending_a_void}, we may assume that $n^u = |\zeta|$, so $m<n^u$. Hence, if $\s = \smallseq{m}$, we are done. If $\s = \smallseq{m,k}$ for some~$k$, and $k\le k^u(m)$, we are also done. In all other cases, we may extend~$u$ to~$v$ by defining $v(\s)=0$ and $v(\s\conc l)=1$ for some large~$l$, and keep $(v,\zeta)\in \SSS$, since $n^v = n^u$. 
\end{proof}

\begin{lemma} \label{lem:forcing_S:compatibility}
  Conditions $p,q\in \SSS$ are compatible if and only if $u^p\cup u^q$ is a function and $\zeta^p$, $\zeta^q$ are comparable. 
\end{lemma}

\begin{proof}
  Suppose that $v=u^p\cup u^q$ is a function and $\zeta^p$, $\zeta^q$ are comparable; without loss of generality, $\zeta^p\preceq \zeta^q$. By \cref{lem:tilde_Q_is_closed_under_unions}, $v\in \tilde \QQ$. To see that $(v,\zeta^q)\in \SSS$, let $m<n^v$. Then either $m<n^{u^q}$ or $m<n^{u^p}$. If the former, then $k^v(m) = k^{u^q}(m)$ codes $\zeta^q\rest{(m+1)}$, as $q\in \SSS$. If the latter, then $k^v(m) = k^{u^p}(m)$ codes $\zeta^p\rest{(m+1)}$, as $p\in \SSS$, but $\zeta^p\preceq \zeta^q$. 
\end{proof}

For sufficiently generic $G\subset \SSS$, we define $(x_G,y_G)$ as above. 

\begin{lemma} \label{lem:forcing_S:end_up_in_WWW}
  If $G\subset \SSS$ is sufficiently generic, then $(x_G,y_G)\in \WWW_\alpha$, and for all $p\in \SSS$, $(x_G,y_G)\in [p]$ if and only if $p\in G$. 
\end{lemma}

\begin{proof}
  \Cref{lem:forcing_S:totality} implies that for a sufficiently generic~$G$, $\bigcup \left\{ u^p \,:\,  p\in G \right\}$ is defined on all of~$T_\alpha$ (so in particular, $x_G\in 2^{\Leaves_\alpha}$), and $y_G\in 2^\w$. Hence, $(x_G,y_G)\in \XXX_\alpha$; the proof that $p\in G$ iff $(x_G,y_G)\in [p]$ is as for \cref{lem:simple_forcing:filters_are_points}, using \cref{lem:forcing_S:compatibility}. 

  To show that in fact $(x_G,y_G)\in \WWW_\alpha$, let $m<n^{x_G}$. There is some $p= (u,\zeta)\in G$ such that $m<n^u$; the fact that $(u,\zeta)\in \SSS$ and $\zeta \prec y_G$ implies that $k^{x_G}(m) = k^u(m)$ codes $y_G\rest{(m+1)}$. 
\end{proof}

The definition of a $\beta$-complete condition is the same as for~$\PP$. 

\begin{lemma} \label{lem:forcing_S:extending_to_beta_complete}
  For every $\beta \le \alpha+1$, every $p\in \SSS$ can be extended to a $\beta$-complete condition in~$\SSS$. 
\end{lemma}

\begin{proof}
  Write $p = (u,\zeta)$. Suppose that $p(\s)=1$, and $\rk(\s)>\beta$ is a limit. If $|\s|\ge 1$, then adding 0-labels to those $\s\conc k$ of rank $<\beta$ (as in the proof of \cref{lem:simple_forcing:beta_complete_is_dense}) does not affect $n^u$, and hence being in~$\SSS$.  If $\s$ is the root~$\rooot$, we apply \cref{lem:forcing_S:extension_lemma} to add 0-labels to finitely many children of the root (first extend~$\zeta$ by~\ref{item:forcing_S:extension_lemma:extending_zeta}, then extend as in~\ref{item:forcing_S:extension_lemma:extending_a_void}). 
\end{proof}

\Cref{lem:restrictions_of_conditions_in_tilde_Q} implies that for all $p\in \SSS$ and $\beta \le \alpha$, $p\rest{\beta}\in \SSS$. We obtain the analogue of \cref{lem:K_alpha:untagging:extension_lemma}:

\begin{lemma} \label{lem:forcing_S:untagging:extension_lemma}
  Suppose that $p\in \SSS$ is $\beta$-complete and $r\in \SSS$ extends $p\rest{(\beta+1)}$. Then~$p$ and $(r\rest \beta)$ are compatible in~$\SSS$. 
\end{lemma}

\begin{proof}
  As in the proof of \cref{lem:K_alpha:untagging:extension_lemma}, let $v = u^p\cup (u^r\rest{\beta})$. By \cref{lem:simple_forcing:untagging:extension_lemma}, $v\in \QQ$. By \cref{lem:tilde_Q_is_closed_under_unions},  $v\in \tilde \QQ$. By \cref{lem:forcing_S:compatibility},  $(v,\zeta^r)\in \SSS$. 
\end{proof}

Again, the proof of the untagging lemma is the same, so we obtain the analogue of \cref{prop:simple_forcing:Sigma_untagging_lemma,lem:K_alpha:Sigma_untagging_lemma}: if $\vphi$ is $\Sigma_\beta$ and $G\subset\SSS$ is sufficiently generic, then $(x_G,y_G) \in [\vphi]\,\,\Iff \,\, (\exists p\in G\cap \SSS_\beta)\,\,\,p\force^*\vphi$, where $\SSS_\beta = \PP_\beta\cap \SSS$. 

\begin{theorem} \label{prop:L_alpha_is_non_colourable}
  Let $\alpha\ge 3$. There is no countable $\bSigma^0_\alpha$-colouring of~$L_\alpha$. 
\end{theorem}

\begin{proof}
  Follow the proof of \cref{prop:K_alpha:no_colouring}, using \cref{lem:forcing_S:extension_lemma}. By \ref{item:forcing_S:extension_lemma:starting_condition} of the lemma, we may start with the condition $p^*$ described. We extend this condition to a sufficiently generic $G\subset \SSS$. By the untagging lemma, we obtain $p\in G\cap \SSS_\alpha$ that strongly forces into one of the $\bSigma^0_\alpha$ sets $C\in \+C$. Writing $p = (u,\zeta)$, we have $u(\rooot)\diverge$, so $u(n^u)\diverge$. By \ref{item:forcing_S:extension_lemma:extending_a_void}, we may assume that $|\zeta| = n^u$. Extend~$u$ to~$v$ by setting $v(n^u)=1$ and $v(\rooot)=0$. By \ref{item:forcing_S:extension_lemma:adding_a_1} and \ref{item:forcing_S:extension_lemma:extending_zeta}, $q_0 = (v,\zeta\conc 0)$ and $r_0 = (v,\zeta\conc 1)$ are both in~$\SSS$. We then use \ref{item:forcing_S:extension_lemma:extending_a_1} to extend~$q_0$ and~$r_0$ to filters~$Q$ and~$R$ as in the proof of \cref{prop:K_alpha:no_colouring}, to show that the set~$C$ contains an edge of~$L_\alpha$. 
\end{proof}

\begin{remark} \label{rmk:why_this_coding_scheme}
  In our coding scheme, instead of using a number~$k$ to code all initial segments of a real $M_k$, we could instead have fixed a single finite binary string coded by~$k$ (as long as each finite string has infinitely many codes). We will require our more flexible coding scheme in the next section. 
\end{remark}

\subsection{Approximating values at heights 1 and 2}

Our next goal is \cref{prop:no_embedding_of_K_3_into_L_3} below: there is no continuous embedding of $K_3$ into~$L_3$. To prove this, we will diagonalise against a given continuous function from $\XXX_3$ to $\WWW_3$. In this kind of argument, we construct approximations to various points in $\XXX_3$, and somehow ensure that there is a pair of points $a$, $b$ that we construct that is connected by an edge in~$K_3$, but such that $F(a)$ and~$F(b)$ cannot be connected by an edge in~$L_3$. To do that, we will need to observe, during the construction, approximations to points such as $F(a)$ and $F(b)$ built by our opponent, and make guesses about whether they are connected by an edge or not. We now describe this guessing procedure. 

Recall that an isomorphism between $2^{{\Leaves}_\alpha}$ and $2^\w$ can be determined by an $\w$-ordering of the leaves of~$T_\alpha$. Fix such a computable ordering $<^*$. This gives us a notion of ``initial segments'' of elements of $2^{\Leaves_\alpha}$:

\begin{definition} \label{def:initial_segments_of_leaves}
  We let $2^{<\Leaves_\alpha}$ be the collection of all finite $r\colon \Leaves_\alpha \to \{0,1\}$ such that if $\s<^* \tau$ and $r(\tau)\converge$ then $r(\s)\converge$. 
\end{definition}

For the rest of this section, we fix $\alpha=3$; we will later generalise the guessing machinery that we develop now for $\alpha=3$. We now define, for each $t\in 2^{<\Leaves_3}$,
\begin{itemize}
  \item $n^t$; and
  \item For $m<n^t$, $k^t(m)$, 
\end{itemize}
that will serve as an approximation for $n^x$ and $k^x(m)$ for $x\in 2^{\Leaves_3}$ extending~$t$. 
The idea is similar to Dekker's ``deficiency stages''. Suppose that at the previous stage, we are guessing that $k^x(m)$ is some~$k$, but we have just discovered that this is wrong: there is some~$a$ such that $t(m,k,a)=1$, so we cannot have $T^x(m,k)=1$ for any $x\succ t$. Then $k^t(m)$ is increased compared to the previous stage, telling us that there's a chance it will go to~$\infty$. That is, at that stage, we are guessing that $T^x(m)=1$, so $n^x\le m$. We therefore arrange that $n^t\le m$ as well. 

\begin{definition} \label{def:alpha_is_three:n_t_and_k_t_m}
  For $t\in 2^{<\Leaves_3}$ we define numbers $n^t$, and for $m<n^t$, numbers $k^t(m)$, by recursion on $|t|$. If $t = \seq{}$ then we set $n^t=0$. 

  Suppose that $t\ne \seq{}$; let $t^-$ be the initial segment of~$t$ of length $|t|-1$. 
  \begin{itemize}
    \item If there is some $m<n^{t^-}$ and some~$a$ such that $t(m,k,a)=1$, where $k = k^{t^-}(m)$, then we let $n^t$ be the least such~$m$. 
    \item Otherwise, we let $n^t = n^{t^-}+1$. 
  \end{itemize}

  Then, for each $m<n^t$, we let $k^t(m)$ be the least~$k$ such that there is no~$a$ with $t(m,k,a)=1$. 
\end{definition}

\begin{lemma} \label{lem:basics_on_n_r_and_k_r_m_for_alpha_is_three}
For all $x\in 2^{\Leaves_3}$, 
\[
  n^x = \liminf \left\{ n^t \,:\,  t\prec x \right\}. 
\]
\end{lemma}

\begin{proof}
Let $x\in 2^{\Leaves_3}$. By induction on $m<n^x$ we show that for all but finitely many $t\prec x$ we have $n^t>m$ and $k^t(m) = k^x(m)$. Suppose that $t_0\prec x$ is such that for all~$t$ with $t_0\preceq t\prec x$, we have $n^t \ge m$ and for all $m'<m$, $k^t(m')= k^x(m')$. Let $t_1$ such that $t_0 \prec t_1 \prec x$ and for all $k<k^x(m)$ there is some~$a$ such that $t_1(m,k,a)=1$. Then for all $t$ with $t_1\preceq t \prec x$, if $n^t>m$ then $k^t(m)= k^x(m)$. This implies that for all~$t$ with $t_1\prec t \prec x$ we have $n^t>m$. 

On the other hand, for any $s\prec x$ there is some~$t$ with $s\preceq t\prec x$ and $n^t\le n^x$. For suppose that $n^s>n^x$; let $k = k^s(n^x)$. Since $T^x(n^x,k)=0$, there is some $t\prec x$ with $t(n^x,k,a)=1$ for some~$a$. For the least such~$t$ we must have $s\prec t$. Suppose that for all $r$ with $s\preceq r \prec t$ we have $n^r>n^x$. Then the minimality of~$t$ implies that $k^{t^-}(n^x) = k$; now the definition implies that $n^t\le n^x$. 
\end{proof}

\begin{lemma} \label{lem:alpha_is_three:extending_s_to_have_desired_n}
  Let $s\in 2^{<\Leaves_3}$. 
  \begin{sublemma}
    \item There is some $t\succ s$ in $2^{<\Leaves_3}$ such that $n^t = 0$. 
    \item For all $m\ge n^s$, there is some $t\succ s$ in $2^{<\Leaves_3}$ such that $n^t = m$, and for all~$r$ with $s\preceq r \preceq t$ we have $n^r\ge n^s$.
  \end{sublemma}
\end{lemma}

\begin{proof}
The main point is that if $r$ is obtained from~$s$ by only adding 0's, then $n^r\ge n^s$ and for all $m<n^s$, $k^r(m)= k^s(m)$. For~(a), keep extending~$s$ by 0's, until we get to place a 1 at a location $(0,k^s(0),a)$ for some~$a$. For~(b), keep extending~$s$ by 0's, until we get to place a 1 at a location $(m,k,a)$, where $k = k^r(m)$, where~$r$ is a sufficiently long extension of~$s$ by adding $0$'s. 
\end{proof}

\subsection{$K_3$ is not minimal}

Together with \cref{prop:L_alpha_is_non_colourable}, the following implies that $K_3$ is not a least graph with no $\bSigma^0_3$ colouring (with respect to continuous homomorphisms). 

\begin{proposition} \label{prop:no_embedding_of_K_3_into_L_3}
  There is no continuous embedding of $K_3$ into~$L_3$. 
\end{proposition}

\begin{proof}
Before we prove this proposition, let us explain informally how it is done. As described above, we are given a continuous function $F\colon \XXX_3\to \WWW_3$, and our goal is to build a pair of points in~$\XXX_3$ that are connected by an edge of $K_3$, but whose images under~$F$ are not, thus, showing that~$F$ is not a graph homomorphism. 

We start by building two points $a$ and~$b$ and plan for them to be connected in~$K_3$: the intention will be to have $a = (z,w^a)$ and $b = (z,w^b)$, where we currently plan that $n^z = 0$ and $w^a\symdiff w^b = \{0\}$ (we can set in advance $w^a = 0^\w$ and $w^b = 1\conc 0^\w$). The opponent, playing~$F$, will have to show us how $F(a)$ and $F(b)$ are connected by an edge: it will show us finite pieces of $F(a) = (x^a,y^a)$ and $F(b) = (x^b, y^b)$. To counter our move, the opponent better ensure that $x^a= x^b = x$ and $y^a\symdiff y^b = \{n^*\}$ for some~$n^*$; and then ensure that $n^{x}= n^*$. 

The main point of the construction, our advantage over the opponent, is that since he has to ensure that $F(a), F(b)\in \WWW_3$, if he showed us this much, he cannot have $n^x > n^*$; otherwise, at least one of $F(a)$ or $F(b)$ cannot be in $\WWW_3$. Thus, when we see this information, we can defeat the opponent by creating a new point~$c$, very close to~$a$, and connected to~$a$ by an edge, which requires us to reset $n^z$ to be some large number $\ell^*$. This is sufficiently large so that the symmetric difference of the second coordinates of $F(a)$ and~$F(c)$ must be greater than~$n^*$, ruling out the possibility that $F(a)$ and~$F(c)$ are connected by an edge. 

\medskip

We can now give the details of the construction. Let $F\colon \XXX_3\to \WWW_3$ be continuous. 

\smallskip

At stages $\ell = 0,1,\dots$ of the construction, we define $s_\ell\in 2^{<\Leaves_3}$; we will ensure that 
\[
  s_0 \prec s_1 \prec s_2 \prec \cdots
\]

For each~$\ell$, we let $t^a_\ell$ be the longest $t\in 2^{<\Leaves_3}$ such that $|t|\le \ell$ and
\[
  F[s_\ell,0^\ell] \subseteq [t]\times 2^\w;
\]
here $[s,\xi] = [s]\times [\xi]\subseteq \XXX_3$ and $F[s,\xi]$ is the pointwise image under~$F$ of this set. We similarly let $t^b_\ell$ be the longest $t$ such that $|t|\le \ell$ and
\[
   F[s_\ell,1\conc 0^{\ell-1}] \subseteq [t]\times 2^\w.
\]
We let $\zeta^a_\ell$ be the longest $\zeta\in 2^{<\w}$ such that $|\zeta|\le \ell$
\[
  F[s_\ell,0^\ell] \subseteq [t^a_\ell,\zeta], 
\]
 and similarly we let $\zeta^b_\ell$ be the longest $\zeta\in 2^{<\w}$ such that $|\zeta|\le \ell$ and 
\[
  F[s_\ell,1\conc 0^{\ell-1}] \subseteq [t^b_\ell,\zeta].
\]

Note that the fact that $(s_\ell)$ is increasing implies that $(t^a_\ell)$, $(t^b_\ell)$, $(\zeta^a_\ell)$, and $(\zeta^b_\ell)$ are all (weakly) increasing with~$\ell$.

\medskip

We start with $s_0 = \seq{}$. Let $\ell\ge 0$, and suppose that we have defined~$s_\ell$. There are three ``phases'' that we consider. 

\smallskip
\noindent\textit{Phase 1: $t^a_\ell$ and $t^b_\ell$ are comparable, and $\zeta^a_\ell$ and $\zeta^b_\ell$ are also comparable.} We let $s_{\ell+1}$ be an extension of $s_\ell$ satisfying $n^{s_{\ell+1}}=0$ (use \cref{lem:alpha_is_three:extending_s_to_have_desired_n}).

\smallskip
\noindent\textit{Phase 2: $t^a_\ell$ and $t^b_\ell$ are comparable, and $\zeta^a_\ell$ and $\zeta^b_\ell$ are incomparable.} In this case, let $\ell^*\le \ell$  be the least stage at which this case applied. We will ensure that for all~$\ell$, $n^{s_\ell}\le \ell^*$.\footnote{If $\ell = \ell^*$ then $n^{s_\ell} = 0$.} Extend $s_\ell$ to $s_{\ell+1}\in 2^{<\Leaves_3}$ so that $n^{s_{\ell+1}} = \ell^*$, and $n^r\ge n^{s_\ell}$ for all~$r$ with $s_\ell\preceq r\preceq s_{\ell+1}$ (again, this is possible by \cref{lem:alpha_is_three:extending_s_to_have_desired_n}).

\smallskip
\noindent\textit{Phase 3: $t^a_\ell$ and $t^b_\ell$ are incomparable.} We let $s_{\ell+1}$ be an extension of $s_\ell$ satisfying $n^{s_{\ell+1}}=0$. 

\smallskip

This completes the construction of the sequence $(s_\ell)$. The ``phases'' are named so because we never ``go back'': when the construction starts, phase~1 applies. If we ever move to phase 2, we never later go back to phase~1. If we ever move to phase~3, we never later go back to phase 1 or phase 2. 

We now argue that there are two points in~$\XXX_3$ that are connected by an edge of $K_3$, but whose images under~$F$ are not connected by an edge. 

\medskip

Let $z = \bigcup_\ell s_\ell$, and let 
\begin{itemize}
  \item $a = (z,0^\w)$;
  \item $b = (z,1\conc 0^\w)$; and
  \item if the construction ever enters phase~2, at stage $\ell^*$, then we let $c = (z,0^{\ell^*}\conc 1 \conc 0^\w)$. 
\end{itemize}

Write $F(a) = (x^a, y^a)$ and $F(b) = (x^b,y^b)$; if the construction enters phase~2, write $F(c) = (x^c, y^c)$. We note that $x_a = \bigcup_\ell t^a_\ell$, $y^a = \bigcup_\ell \zeta^a_\ell$, and similarly for~$b$. We consider which is the last phase that the construciton reaches. 

\smallskip
\noindent\textit{The construction is always in phase~1}. In this case, the sequence $(n^{s_\ell})$ shows that $n^z = \liminf  \left\{ n^r \,:\,  r\prec z \right\} = 0$ (\cref{lem:basics_on_n_r_and_k_r_m_for_alpha_is_three}), so $a$ and~$b$ are connected by an edge in~$K_3$. We claim that $F(a)$ and $F(b)$ are not connected by an edge; indeed, in this case, $F(a)=F(b)$.

\smallskip
\noindent\textit{The construction eventually settles in phase~2}. Let $\ell^*$ be the stage at which the construction enters stage~2. In this case, the sequence $(n^{s_\ell})_{\ell \ge \ell^*}$, as well as the instructions of how to pass from $s_\ell$ to $s_{\ell+1}$, show that $n^z = \ell^*$. This means that~$a$ and~$c$ are connected by an edge. We claim that $F(a)$ and $F(c)$ are not connected by an edge.

Suppose that they were. Then $x^a = x^c$ (call this common value~$x$), and $y^a\symdiff y^c$ is the singleton $\{n^x\}$. Also, since we never leave phase~2, $x^b = x$. Since we did enter phase~2, $y^a\ne y^b$; let $n_{a,b} = \min (y^a\symdiff y^b)$. 

We obtain a contradiction by showing that $n^x>n_{a,b}$ and $n^x\le n_{a,b}$. 
\begin{itemize}
  \item Since $F[s_{\ell^*},0^{\ell^*}]\subseteq [t^a_{\ell^*},\zeta^a_{\ell^*}]$, and the second coordinates of~$a$ and~$c$ both extend $0^{\ell^*}$, we must have that $y^a$ and $y^c$ both extend $\zeta^a_{\ell^*}$, so their point of difference, $n^x$, must be greater than $|\zeta^a_{\ell^*}|$, and so greater than $n_{a,b}$. 

  \item On the other hand, if $n^x>n_{a,b}$, then since the strings $y^a\rest{(n_{a,b}+1)}$ and $y^b\rest{(n_{a,b}+1)}$ are incomparable, they cannot be both coded by the number $k^x(n_{a,b})$. This would imply that $F(a)\notin \WWW_3$ or $F(b)\notin \WWW_3$. 
\end{itemize}

\smallskip
\noindent\textit{The construction eventually settles in phase~3}. In this case, as in the first case, $n^z = 0$, so $a$ and~$b$ are connected by an edge. But in this case, $x^a \ne x^b$, so $F(a)$ and~$F(b)$ are not connected by an edge.
\end{proof}

\section{Candidates for minimal graphs} \label{sec:candidates_for_minimal_graphs}

Unfortunately, even the graph $L_3$ is not a least graph with no $\bSigma^0_3$ colouring, with respect to continuous homomorphisms. We will present a graph $H_3$, which is in fact a least graph with no $\bSigma^0_3$ colouring, and observe that there is no continuous homomorphism from $L_3$ into~$H_3$. The graph $H_3$ is a variant of the minimal graph constructed by Lecomte and Zeleny in \cite{LecomteZeleny}. We show how this graph can potentially be generalised to levels $\alpha>3$, all using an elaboration of the machinery presented so far. 

The difference between $H_3$ and~$L_3$ is the location at which we connect two points. Instead of requiring that $y\symdiff y' = \{n^x\}$, we will require $y\symdiff y' = \{c^x(n^x)\}$, where $c^x(n^x)$ is a number (likely larger than~$n^x$) that is computed using the approximation to~$n^x$ introduced above. First, we explain how to similarly approximate~$n^x$ even when $\alpha>3$. To do so, we need to approximate values of $T^x(\s)$ for various~$\s$ (eventually, we will be interested in $|\s|=3$).

\subsection{True stages for $T_\alpha$}

Fix a computable $\alpha\ge 1$. Recalling that $\alpha$ is actually a concrete computable ordinal (a nice computable well-ordering of a computable subset of~$\Nat$), for $\beta<\alpha$ (considered as a sub-ordering of~$\alpha$) we let $a_\beta\in \Nat$ be the least upper bound of~$\beta$ in~$\alpha$.

For each $\beta<\alpha$, we (uniformly) fix a computable $\w$-ordering $<^*_\beta$ on the set of nodes of $T_\alpha$ of rank~$\beta$. This is done in a reasonable way, in particular, if $\s\conc k$ and $\s\conc k'$ are two sibling nodes of the same rank~$\beta$, and $k<k'$, then $\s\conc k <^*_\beta \s\conc k'$. Thus, if $\alpha$ is a successor ordinal (so all children of the root have the same rank $\alpha -1$), then $<^*_{\alpha-1}$ agrees with the ``natural'' ordering on the children of the root. 

As above, we use the ordering $<^*_0$ on the leaves to give us a notion of an initial segment of an element of $2^{\Leaves_\alpha}$, that is, we define the collection $2^{<\Leaves_\alpha}$ as in \cref{def:initial_segments_of_leaves}. 

We will define, for each $t\in 2^{<\Leaves_\alpha}$, a partial labelling $T^t$ of $T_\alpha$, to serve as an approximation (or guess) for $T^x$ for some $x\in 2^{\Leaves_\alpha}$ extending~$t$. This will be done by induction on~$|t|$. We start with:

\begin{orderedlist}
  \item \label{item:true_stages:definition:empty_string_empty_tree} 
  If $t = \seq{}$ then $T^t$ is empty (no nodes are labelled). 
\end{orderedlist}

Now let $t\in 2^{<\Leaves_\alpha}$ be nonempty. To define $T^t$, by induction on $\beta \le \alpha$, we will define $T^t(\s)$ for nodes of rank~$\beta$. If this is done for all $\gamma<\beta$, then we let $T^t_\beta$ be the partial labelling defined so far (the restriction of $T^t$ to nodes of rank $<\beta$). 

\begin{orderedlist}[resume]
  \item \label{item:true_stages:definition:0} 
  We start by setting $T^t_1 = t$.
\end{orderedlist}

That is, for nodes~$\s$ of rank~0, $T^t(\s) = t(\s)$. Now, let $\beta\le \alpha$ be nonzero, and suppose that $T^t_\beta$ is already defined. 

\begin{orderedlist}[resume]
  \item \label{item:true_stages:definition:saying_nothing} 
  Suppose that there are at most $a_\beta$-many $s\prec t$ such that $T^s_\beta \subseteq T^t_\beta$. We then let $T^t(\s)\diverge$ for all~$\s$ of rank~$\beta$. 

\item \label{item:true_stages:definition:injury_or_not} 
Suppose that this is not the case. Let~$s$ be the longest $s\prec t$ such that $T^s_\beta \subseteq T^t_\beta$. Let~$\tau$ be the $<^*_\beta$-least node such that either:
\begin{itemize}
  \item $T^s(\tau)\diverge$; or
  \item $T^s(\tau)=1$, and there is some~$k$ such that $T^t(\tau\conc k)=1$. 
\end{itemize}
  We define $T^t(\s)$ for~$\s$ of rank~$\beta$ as follows: 
  \begin{itemize}
    \item For $\s<^*_\beta \tau$, $T^t(\s) = T^s(\s)$. 
    \item For $\s >^*_\beta \tau$, $T^t(\s)\diverge$. 
    \item We let $T^t(\tau)=1$, unless there is some~$k$ with $T^t(\tau\conc k)=1$, in which case we let $T^t(\tau)=0$. 
  \end{itemize}
  \end{orderedlist}

This completes the definition of $T^t$ for all $t\in 2^{<\Leaves_\alpha}$. For $x\in 2^{\Leaves_\alpha}$, $T^x$ is already defined (\cref{def:the_tree_determined_by_a_labelling_of_the_leaves}); for $\beta \le\alpha+1$, we let $T^x_\beta$ be the restriction of~$T^x$ to nodes of rank~$<\beta$.

\begin{definition} \label{def:true_stage_relations}
  For $s,t\in 2^{\le \Leaves_\alpha}$, and nonzero $\beta \le \alpha+1$, we write $s\preceq_\beta t$ if $T^s_\beta \subseteq T^t_\beta$. 
\end{definition}

We list some basic properties of these partial labellings and relations. 

\begin{lemma} \label{lem:true_stages:basic_properties}
  Let $\beta \le \alpha+1$ be nonzero, and let $r,s,t\in 2^{\le\Leaves_\alpha}$. 
    \begin{sublemma}
      \item \label{item:true_stages:basic_properties:zero}
       $s\prec_1 t$ if and only if $s\prec t$. 

      \item \label{item:true_stages:basic_properties:empty}
       For all~$t$, $\seq{}\preceq_\beta t$.
      
      \item \label{item:true_stages:basic_properties:nested}
      If $\gamma \le \beta$ and $s\prec_\beta t$ then $s\prec_\gamma t$. 

      \item \label{item:true_stages:basic_properties:continuity}
      If $\beta$ is a limit, then  $s\prec_\beta t$ if and only if for all $\gamma<\beta$, $s\prec_\gamma t$.

      \item \label{item:true_stages:basic_properties:transitivity}
      If $r\prec_\beta s\prec_\beta t$ then $r\prec_\beta t$. 

      \item \label{item:true_stages:basic_properties:initial_segment_convergence_at_level_beta}
      If $T^t(\tau)\converge$ then for all $\s<^*_{\rk(\tau)} \tau$, $T^t(\s)\converge$. 

      \item \label{item:true_stages:basic_properties:internally_consistent_and_witnessed}
      For every non-leaf $\s$, if $T^t(\s)\converge$, then $T^t(\s) = 0$ if and only if there is some~$k$ such that $T^t(\s\conc k)=1$.

      \item \label{item:true_stages:basic_properties:only_ones_into_zeros}
      If $s\prec_\beta t$, $\trk(\s)=\beta$, and $T^s(\s)=0$,  then it is not the case that $T^t(\s)=1$. 

      \item \label{item:true_stages:basic_properties:finiteness}
      If $t\in 2^{<\Leaves_\alpha}$ (i.e., is finite), then $T^t$ is finite, and computable from~$t$.

      \item \label{item:true_stages:basic_properties:diamond}
      If $r\prec s\prec t$ and $r,s\prec_\beta t$ then  $r\prec_\beta s$.

      \item \label{item:true_stages:basic_properties:beta_plus_1_extension}
      Suppose that $\beta\le  \alpha$, $s\prec_\beta t$, and $s\nprec_{\beta+1} t$. Let~$\tau$ be  the $<^*_\beta$-least node such that it is not the case that $T^s(\tau)\converge = T^t(\tau)$. Then:
      \begin{itemize}
        \item  $T^s(\tau)=1$ and $T^t(\tau)=0$; and
        \item For all~$r$ with $s\preceq_\beta r \preceq_\beta t$, for all $\s<^*_\beta \tau$, we have $T^r(\s)\converge = T^s(\s)$. 
      \end{itemize}

      \item \label{item:true_stages:basic_properties:club}
      If $\beta\le  \alpha$,  $r\prec_\beta s\prec_\beta t$ and $r\prec_{\beta+1} t$, then  $r\prec_{\beta+1} s$. 
    \end{sublemma}
\end{lemma}

\begin{proof}
  \ref{item:true_stages:basic_properties:zero}---\ref{item:true_stages:basic_properties:transitivity} are immediate from the definitions. 

  \medskip

  \ref{item:true_stages:basic_properties:initial_segment_convergence_at_level_beta} and \ref{item:true_stages:basic_properties:internally_consistent_and_witnessed} hold by definition when~$t$ is infinite. For finite~$t$, they are proved by a straightforward induction on $|t|$. 

  \smallskip

  For~\ref{item:true_stages:basic_properties:only_ones_into_zeros}: by \ref{item:true_stages:basic_properties:internally_consistent_and_witnessed}, there is some~$k$ such that $T^s(\s\conc k)=1$. Since $\trk(\s\conc k)<\beta$ and $s\prec_\beta t$, we have $T^t(\s\conc k)=1$. By \ref{item:true_stages:basic_properties:internally_consistent_and_witnessed} again, we cannot have $T^t(\s)=1$. 

  \medskip

  \ref{item:true_stages:basic_properties:finiteness} follows from the fact that there are only finitely many~$\beta$ with $a_\beta <|t|$, and induction on~$|t|$; the entire construction is computable. 

  \medskip

  \ref{item:true_stages:basic_properties:diamond}, \ref{item:true_stages:basic_properties:beta_plus_1_extension} and  \ref{item:true_stages:basic_properties:club}  are proved by simultaneous induction on~$\beta$. Let $\beta \le  \alpha+1$, and suppose that these have been verified for all $\gamma<\beta$. 

  \smallskip

  If $\beta =1$, then \ref{item:true_stages:basic_properties:diamond}$_\beta$ is immediate. If $\beta$ is a limit, then \ref{item:true_stages:basic_properties:diamond}$_\beta$ follows by induction, and \ref{item:true_stages:basic_properties:continuity}. If $\beta = \gamma+1>1$ is a successor, then \ref{item:true_stages:basic_properties:diamond}$_\beta$ follows from \ref{item:true_stages:basic_properties:diamond}$_\gamma$ and \ref{item:true_stages:basic_properties:club}$_\gamma$ (and~\ref{item:true_stages:basic_properties:nested}). 

  \smallskip

  For~\ref{item:true_stages:basic_properties:beta_plus_1_extension}$_\beta$, fix some $s\in 2^{<\Leaves_\alpha}$. If $T^s(\s)\diverge$ for all~$\s$ of rank~$\beta$, then $s\prec_{\beta+1} t$ whenever $s\prec_\beta t$. Hence, suppose that $T^s(\s)\converge$ for some $\s$ of rank~$\beta$. By induction on the length of $t\succeq_\beta s$, we prove that if $s\nprec_{\beta+1} t$ then \ref{item:true_stages:basic_properties:beta_plus_1_extension}$_\beta$ holds between~$s$ and~$t$. This is vacuous when $s = t$. 

  Fix some $t\succ_\beta s$, and suppose that $s\nprec_{\beta+1} t$. Let~$r$ be the $\prec_\beta$-predecessor of~$t$. By \ref{item:true_stages:basic_properties:diamond}$_\beta$, $s\preceq_\beta r \prec_\beta t$. Since $T^s(\s)\converge$ for some~$\s$, $s$ has more than $a_\beta$-many $\prec_\beta$-predecessors. By \ref{item:true_stages:basic_properties:transitivity}, each such is a $\prec_\beta$-predecessor of~$t$, so case \ref{item:true_stages:definition:saying_nothing} does not apply in the definition of $T^t_{\beta+1}$. 

  Let $\tau_{s,t}$ be the $<^*_\beta$-least~$\tau$ such that $T^s(\tau)\converge$ and $\lnot (T^t(\tau) = T^s(\tau))$. If $s\nprec_{\beta+1} r$, let $\tau_{s,r}$ be the $<^*_\beta$-least~$\tau$ such that $T^s(\tau)\converge$ and $\lnot (T^r(\tau) = T^s(\tau))$. Similarly, if $r\nprec_{\beta+1} t$, we let  $\tau_{r,t}$ be the $<^*_\beta$-least~$\tau$ such that $T^r(\tau)\converge$ and $\lnot (T^t(\tau) = T^r(\tau))$.

  If $s\nprec_{\beta+1} r$, then $\tau_{s,t}\le^*_\beta \tau_{s,r}$. For suppose otherwise. By induction, $T^s(\tau_{s,r})=1$ and $T^r(\tau_{s,r})=0$. By minimality of $\tau_{s,t}$, $T^t(\tau_{s,r})=T^s(\tau_{s,r})=1$. This contradicts~\ref{item:true_stages:basic_properties:only_ones_into_zeros} (between $r$ and~$t$). Hence, for all $\s<^*_\beta \tau_{s,t}$, $T^r(\s)\converge = T^s(\s)$. 

  Similarly, if $r\nprec_{\beta+1} t$ then $\tau_{s,t}\le^* \tau_{r,t}$. By the definition of $T^t$, in this case, $T^r(\tau_{r,t})=1$ and $T^t(\tau_{r,t})=0$. If $\tau_{r,t}<^*_\beta \tau_{s,t}$ then by the minimality of $\tau_{s,t}$, $T^s(\tau_{r,t})= 0$, contradicting \ref{item:true_stages:basic_properties:only_ones_into_zeros} between~$s$ and~$r$. 

  Hence, if $\s<^*_\beta \tau_{s,t}$, then $T^r(\s)\converge = T^t(\s)$. This confirms the second part of \ref{item:true_stages:basic_properties:beta_plus_1_extension}$_\beta$ between~$s$ and~$t$. Also, by the definition of $T^t$, this implies that $T^t(\tau_{s,t})\converge$. By \ref{item:true_stages:basic_properties:only_ones_into_zeros}, we indeed have $T^t(\tau_{s,t})=0$ and $T^s(\tau_{s,t})=1$.

  \smallskip

  This concludes the proof of \ref{item:true_stages:basic_properties:beta_plus_1_extension}$_\beta$ when~$t$ is finite. Suppose that $x\in 2^{\Leaves_\alpha}$, $s\prec_\beta x$ but $s\nprec_{\beta+1} x$. Define $\tau$ as in \ref{item:true_stages:basic_properties:beta_plus_1_extension}. Since $T^x(\tau)\converge$, by \ref{item:true_stages:basic_properties:only_ones_into_zeros}, we must have $T^x(\tau)=0$ and $T^s(\tau)=1$. For the second part of \ref{item:true_stages:basic_properties:beta_plus_1_extension}$_\beta$ between~$s$ and~$x$, let $r$ with $s\preceq_\beta r\prec_\beta x$, and suppose that there is some $\s<^*_\beta\tau$ such that it is not the case that $T^r(\s) = T^s(\s)$. Choosing the $<^*_\beta$-least such~$\s$, we have $T^r(\s)=0$ and $T^x(\s)=T^s(\s)=1$, as usual, contradicting \ref{item:true_stages:basic_properties:only_ones_into_zeros} between~$r$ and~$x$.

  \smallskip

  Finally, \ref{item:true_stages:basic_properties:club}$_\beta$ follows from \ref{item:true_stages:basic_properties:beta_plus_1_extension}$_\beta$: suppose that $r\prec_\beta s\prec_\beta t$, and $r\nprec_{\beta+1} s$. Let~$\tau$ be $<^*_\beta$-minimal with $T^s(\tau)\ne T^r(\tau)$. By \ref{item:true_stages:basic_properties:beta_plus_1_extension}$_\beta$, $T^r(\tau)=1$ and $T^s(\tau)=0$. By \ref{item:true_stages:basic_properties:only_ones_into_zeros}, we cannot have $T^t(\tau)=1$. So $r\nprec_{\beta+1} t$. 
\end{proof}

\begin{corollary} \label{cor:you_wouldnt_think_this_would_be_useful} 
  Let $\beta \le  \alpha+1$. Suppose that $s,t\in 2^{<\Leaves_\alpha}$ and that $T^s_\beta$ and $T^t_\beta$ are compatible  ($T^s_\beta \cup T^t_\beta$ is a function).  Then $s\preceq_\beta t$ or $t\preceq_\beta s$. 
\end{corollary}

\begin{proof}
  We have $s\preceq t$ or $t\preceq s$; say $s\preceq t$. Let~$\gamma$ be the greatest with $s\prec_\gamma t$. Suppose that $\gamma<\beta$. By \ref{item:true_stages:basic_properties:beta_plus_1_extension} of \cref{lem:true_stages:basic_properties} , there is some~$\tau$ with $\rk(\tau)=\gamma$ and $T^s(\tau)=1$, $T^t(\tau)=0$. 
\end{proof}

\begin{lemma} \label{lem:true_stages:existence_of_true_stages}
  For $x\in 2^{\Leaves_\alpha}$ and nonzero $\beta \le \alpha+1$, there are infinitely many $s\prec_\beta x$, and 
  \[
    T^x_\beta = \bigcup \left\{ T^s_\beta \,:\,  s\prec_\beta x  \right\}. 
  \]
\end{lemma}

\begin{proof}
  We prove this by induction on~$\beta$. It is immediate for $\beta = 1$. 

  \medskip

  Suppose that~$\beta$ is a limit, and that the lemma holds for all $\gamma<\beta$. It suffices to show that there are infinitely many $t\prec_\beta x$. Well, let $r\prec x$; let $\gamma = \max \{\delta<\beta\,:\, a_\delta \le |r|\}$. Let $t$ be the shortest with $t\succeq r$ and $t\prec_{\gamma+1} x$. Let $\delta\in [\gamma+1,\beta)$, and suppose that $s\prec_\delta t$. Since $\delta > \gamma$,  $s\prec_{\gamma+1} t$. Since $t\prec_{\gamma+1} x$, we have $s\prec_{\gamma+1} x$. The minimality of~$t$ implies that $s\preceq r$. Hence, $t$ has at most $|r|+1$ many $\prec_\delta$-predecessors. Since $a_\delta > |r|$, $t$ has at most $a_\delta$-many $\prec_\delta$-predecessors. By \ref{item:true_stages:definition:saying_nothing} of the definition of $T^t_\delta$, we have $T^t(\s)\diverge$ for all~$\s$ of rank~$\delta$. Hence, $T^t_\beta = T^t_{\gamma+1}$, so $t\prec_\beta x$. 

  \medskip

  For the successor case, suppose that $\beta \le \alpha$ and that the lemma has been verified for~$\beta$. We first show:

  \begin{description}
    \item[$(*)$] For all $\s$ of rank~$\beta$, for all but finitely many $s\prec_\beta x$ we have $T^s(\s)\converge = T^x(\s)$. 
  \end{description}
  First, note that it suffices to show that for all but finitely many $s\prec_\beta x$ we have $T^s(\s)\converge$. This is because by \cref{lem:true_stages:basic_properties}\ref{item:true_stages:basic_properties:only_ones_into_zeros}, if $s\prec_\beta x$ and $T^s(\s)=0$ then $T^x(\s)=0$; on the other hand, if $T^x(\s)=0$ then there is some~$k$ with $T^x(\s\conc k)=1$, and then, by induction, for all but finitely $s\prec_\beta x$ we have $T^s(\s\conc k)=1$, whence $T^s(\s)=0$ if $T^s(\s)\converge$. 

  We prove $(*)$ by induction on $<^*_\beta$. Let~$\s$ be a node of rank~$\beta$; by the induction on~$<^*_\beta$, we fix a long $s\prec_\beta x$ so that for all $t$ with $s\preceq_\beta t\prec_\beta x$, for all $\rho<^*_\beta \s$, we have $T^t(\rho) = T^x(\rho)$. Also assume that $s$ has more than $a_\beta$-many $\prec_\beta$-predecessors (this is required when $\s$ is the $<^*_\beta$-least node, otherwise, this is implied by $T^s(\rho)\converge$ for some~$\rho$). Then our construction ensures that for all~$t$ with $s\prec_\beta t\prec_\beta x$, $T^t(\s)\converge$. This completes the proof of~$(*)$. 

  \smallskip

  So to prove the lemma for~$\beta+1$, it suffices to show that there are infinitely many $t\prec_{\beta+1} x$. Let $s\prec_\beta x$ and suppose that $s\nprec_{\beta+1} x$. Let~$\tau$ be $<^*_\beta$-least with $T^s(\tau)\converge\ne T^x(\tau)$. By \cref{lem:true_stages:basic_properties}\ref{item:true_stages:basic_properties:beta_plus_1_extension}, $T^x(\tau)=0$ and $T^s(x)=1$, and for all $\s<^*_\beta \tau$, for all~$t$ with $s\preceq_\beta t\prec_\beta x$, we have $T^t(\s)\converge = T^x(\s)$. Let~$t$ be the shortest $t\prec_\beta x$ such that $T^t(\tau) = 0$; then $s\prec_\beta t$. Let~$r$ be the $\prec_\beta$-predecessor of~$t$. Since it is not the case that $T^t(\tau) = T^{r}(\tau)$, by the definition of $T^t$ (case \ref{item:true_stages:definition:injury_or_not}), for all $\s>^*_\beta \tau$ we have $T^t(\s)\diverge$. Hence, $t\prec_{\beta+1} x$.
\end{proof}

\begin{remark} \label{rmk:true_stages:unique_path}
  In fact, for each $x\in 2^{\Leaves_\alpha}$, 
    \[
      \left\{ s \,:\,  s\prec_\beta x \right\}
    \]
    is the unique infinite path in the tree $(\left\{ s \,:\, s\prec x   \right\},\prec_\beta)$. This is proved by induction on~$\beta$; for the successor step, use~\ref{item:true_stages:basic_properties:club}.
\end{remark}

\subsection{Dynamic coding locations}

For the rest of this section, we assume that $\alpha$ is the successor of a successor ordinal. We let 
\[
  \alpha^* = \alpha -2.
\]
The point is that every node of length~2 has rank $\alpha^*$, and so for all $t\in 2^{\le\Leaves_\alpha}$, 
\[
  T^t_{\alpha^*} = T^t\rest{\{ \s\in T_\alpha\,:\, |\s|\ge 3 \}}. 
\]

Using $T^t_{\alpha^*}$, we can generalise \cref{def:alpha_is_three:n_t_and_k_t_m} to $\alpha >3$.

\begin{definition} \label{def:n_t_and_k_t_m:general_case}
  For $t\in 2^{<\Leaves_\alpha}$ we define numbers $n^t$, and for $m<n^t$, numbers $k^t(m)$, by induction on $|t|$. We let $n^{\seq{}}=0$. 

  Suppose that $t\ne \seq{}$; let~$s$ be the longest such that $s\prec_{\alpha^*} t$.\footnote{Such an~$s$ exists; by \cref{lem:true_stages:basic_properties}\ref{item:true_stages:basic_properties:empty}, $\seq{}\prec_{\alpha^*} t$.}
  \begin{itemize}
    \item If there is some $m<n^{s}$ and some~$a$ such that $T^t(m,k^s(m),a)=1$, then we let $n^t$ be the least such~$m$. 
    \item Otherwise, we let $n^t = n^{s}+1$. 
  \end{itemize}

  Then, for each $m<n^t$, we let $k^t(m)$ be the least~$k$ such that there is no~$a$ with $T^t(m,k,a)=1$. 
\end{definition}

\begin{lemma} \label{lem:basics_on_n_r_and_k_r_m_for_general_alpha}
Let $s,t\in 2^{\le\Leaves_\alpha}$. 
\begin{sublemma}
  \item \label{item:n_and_k:n_t_bounded_by_length_of_t}
  If $t\in 2^{<\Leaves_\alpha}$, then $n^t \le |t|$. 

  \item \label{item:n_and_k:k_only_increases}
    If $s\preceq_{\alpha^*} t$ then for all~$m<\min \{n^s, n^t\}$, $k^s(m)\le k^t(m)$. 

  \item \label{item:n_and_k:when_there_is_a_defect}
    Suppose that $s\preceq_{\alpha^*} t$, $m < \min \{n^s,n^t\}$, and $k^s(m)\ne k^t(m)$. Then there is some~$r$ with $s\prec_{\alpha^*} r \prec_{\alpha^*} t$ such that $n^r\le m$. 
  \end{sublemma}
\end{lemma}

\begin{proof}
  \ref{item:n_and_k:n_t_bounded_by_length_of_t} is proved by induction on $|t|$, since $n^t\le n^s+1$, where~$s$ is the longest with $s\prec_{\alpha^*} t$. 

  \ref{item:n_and_k:k_only_increases} is immediate from the definition of $k^t(m)$; if $s\preceq_{\alpha^*} t$ and there is no~$a$ with $T^t(m,k,a)=1$, then there is no~$a$ with $T^s(m,k,a)=1$. 

  For \ref{item:n_and_k:when_there_is_a_defect}, suppose that $s\preceq_{\alpha^*} t$, $m < \min \{n^s,n^t\}$, and $k^s(m)\ne k^t(m)$. By \ref{item:n_and_k:k_only_increases}, $k^t(m)> k^s(m)$. Choose $r\preceq_{\alpha^*} t$ shortest such that $n^r>m$ and $k^r(m)> k^s(m)$. By \cref{lem:true_stages:basic_properties}\ref{item:true_stages:basic_properties:diamond}, $r$ and~$s$ are $\preceq_{\alpha^*}$-comparable; by \ref{item:n_and_k:k_only_increases} again, $s\prec_{\alpha^*} r$. Let $\bar r$ be the longest with $\bar r\prec_{\alpha^*} r$; by \cref{lem:true_stages:basic_properties}\ref{item:true_stages:basic_properties:diamond} again, $s\preceq_{\alpha^*} \bar r \prec_{\alpha^*} r$. By minimality of~$r$, either $n^{\bar r}\le m$, or $k^{\bar r}(m) = k^s(m)$. The latter case is impossible, since then we would have $k^{\bar r}(m)\ne k^r(m)$, whence by definition we would have $n^r\le m$. 
\end{proof}

\begin{definition} \label{def:square_extension}
  Let $r, s\in 2^{\le \Leaves_\alpha}$. We write
  \[
    r \sqsubseteq_m s
  \]
  if $r\preceq_{\alpha^*} s$, $m\le \min \{n^r,n^s\}$,  and for all $\ell<m$, $k^s(\ell) = k^r(\ell)$. 
\end{definition}

\begin{lemma} \label{lem:basics_of_square_extensions}
Let $s,t\in 2^{<\Leaves_\alpha}$. 
\begin{sublemma}
  \item \label{item:basics_of_square_extensions:transitive}
  $\sqsubseteq_m$ is transitive. 

  \item \label{item:basics_of_square_extensions:nested}
  If $s\sqsubseteq_m t$ and $m'<m$ then $s\sqsubseteq_{m'} t$.

  \item \label{item:basics_of_square_extensions:monotony}
  If $s\sqsubseteq_m t$ then for all~$r$ with $s\preceq_{\alpha^*} r \preceq_{\alpha^*} t$ we have $n^r\ge m$. 

  \item \label{item:basics_of_square_extensions:club}
   If $r\preceq_{\alpha^*} s \preceq_{\alpha^*} t$ and $r\sqsubseteq_m t$ then $r\sqsubseteq_m s$. 

   \item \label{item:basics_of_square_extensions:diamond}
   If $r\preceq s$ and $r,s\sqsubseteq_m t$ then $r\sqsubseteq_m s$. 

  \item \label{item:basics_of_square_extensions:hit_then_s_precisely}
   If $s$ is the shortest with $s\sqsubseteq_m t$, then $n^s = m$. 
\end{sublemma}
\end{lemma}

\begin{proof}
  \ref{item:basics_of_square_extensions:transitive} follows from the transitivity of $\preceq_{\alpha^*}$. \ref{item:basics_of_square_extensions:nested} is immediate from the definition. 

  \smallskip

For~\ref{item:basics_of_square_extensions:monotony}, suppose that $s\prec_{\alpha^*} t$, that $n^s,n^t\ge m$, and that there is some $r$ such that $s\prec_{\alpha^*} r \prec_{\alpha^*} t$ and $n^r<m$. Let~$r$ be shortest such; let $\bar r$ be the $\prec_{\alpha^*}$-predecessor of~$r$. Again, $s\preceq_{\alpha^*} \bar r$. The fact that $n^r < n^{\bar r}$ implies that there is some~$a$ such that $T^r(n,k,a)=1$, where $n = n^r$ and $k = k^{\bar r}(n)$. Since $r\preceq_{\alpha^*} t$, we have $T^t(n,k,a)=1$, so $k^t(n)\ne k$. By \cref{lem:basics_on_n_r_and_k_r_m_for_general_alpha}\ref{item:n_and_k:k_only_increases}, $k^t(n)> k^{\bar r}(n)\ge k^s(n)$, so $k^t(n)>k^s(n)$. This shows that $s\nsqsubset_{m} t$. 

\smallskip

  For \ref{item:basics_of_square_extensions:club}, suppose that $r\prec_{\alpha^*} s \prec_{\alpha^*} t$ and $r\sqsubseteq_m t$. By \ref{item:basics_of_square_extensions:monotony}, $n^s\ge m$. Let $\ell<m$. By \cref{lem:basics_on_n_r_and_k_r_m_for_general_alpha}\ref{item:n_and_k:k_only_increases}, $k^t(\ell)\ge k^s(\ell)\ge k^r(\ell)$. By assumption, $k^t(\ell) = k^r(\ell)$; so $k^s(\ell) = k^r(\ell)$ as well. This shows that $r\sqsubseteq_m s$. 

  \smallskip

  \ref{item:basics_of_square_extensions:diamond}: By \cref{lem:true_stages:basic_properties}\ref{item:true_stages:basic_properties:diamond}, $r\preceq_{\alpha^*} s$. Now $r\sqsubseteq_m s$ follows from~\ref{item:basics_of_square_extensions:club}. 

  \smallskip

  For \ref{item:basics_of_square_extensions:hit_then_s_precisely}, suppose that $s\sqsubseteq_m t$ and that $n^s>m$. Since $n^s>0$, $s\ne \seq{}$, so let $\bar s$ be the $\prec_{\alpha^*}$-predecessor of~$s$. There are two possibilities: either $n^s = n^{\bar s}+1$, in which case $\bar s\sqsubset_{n^{\bar s}} s$; or $n^s<n^{\bar s}$, in which case $\bar s \sqsubset_{n^s} s$. In either case, $\bar s \sqsubseteq_{m} s$, hence $\bar s\sqsubset_m t$, so $s$ is not the shortest with $s\sqsubseteq_m t$. 
\end{proof}

\begin{lemma} \label{lem:_true_initial_segments_for_approximating_n_x}
  Let $x\in 2^{\le\Leaves_\alpha}$. 
  \begin{sublemma}
    \item \label{item:n_x_initial_segments:k_stabilises}
    For all finite $m\le n^x$,  for all but finitely many $s\prec_{\alpha^*} x$, $s\sqsubset_m x$.

    \item \label{item:n_x_initial_segments:n_liminf}
    $n^x = \liminf \left\{ n^t \,:\,  t\prec_{\alpha^*} x \right\}$.

    \item \label{item:n_x_initial_segments:true_stages}
    There are infinitely many $s$ such that $s\sqsubset_{n^s} x$.
  \end{sublemma}
\end{lemma}

\begin{proof}
  \ref{item:n_x_initial_segments:k_stabilises} and \ref{item:n_x_initial_segments:n_liminf} are similar to the proof of \cref{lem:basics_on_n_r_and_k_r_m_for_alpha_is_three}.

  We prove \ref{item:n_x_initial_segments:k_stabilises} by induction on~$m$. It is immediate for $m=0$. Suppose that it holds for~$m$, and that $n^x>m$. Let $s_0\prec_{\alpha^*} x$ such that for all~$s$ with $s_0\preceq_{\alpha^*} s\prec_{\alpha^*} x$, $s\sqsubset_m x$. By \cref{lem:true_stages:existence_of_true_stages}, there is some~$s_1$ such that $s_0\preceq_{\alpha^*} s_1 \prec_{\alpha^*} x$ and for all $k<k^x(m)$, there is some~$a$ such that $T^{s_1}(m,k,a)=1$. On the other hand, for all $s\prec_{\alpha^*} x$, there is no~$a$ with $T^s(m,k^x(m),a)=1$. Hence, for all~$s$ with $s_1\preceq_{\alpha^*} s\prec_{\alpha^*}x$, if $n^s>m$ then $k^s(m)= k^x(m)$, so $s\sqsubset_{m+1} x$. We argue that for all~$s$ with $s_1\prec_{\alpha^*} s \prec_{\alpha^*} x$ we have $n^s>m$. Let $s$ be such, and let $\bar s$ be the $\prec_{\alpha^*}$-predecessor of~$s$. So $s_1\preceq_{\alpha^*}\bar s \prec_{\alpha^*} x$. Since $n^{\bar s}\ge m$ and $\bar s\sqsubseteq_m x$, we have $\bar s\sqsubseteq_m s$. If $n^{\bar s} = m$ then by definition, $n^s = m+1$. If not, then $k^{\bar s}(m) = k^x(m)$ implies that $n^s > m$ as well. 

  \smallskip

  One direction of \ref{item:n_x_initial_segments:n_liminf} follows from \ref{item:n_x_initial_segments:k_stabilises}: if $m\le n^x$, then for all but finitely many $s\prec_{\alpha^*} x$ we have $n^s\ge m$. On the other hand, suppose that $n^x$ is finite, $s\prec_{\alpha^*} x$, and $n^s>n^x$; and that $r\sqsubset_{n^x} x$ for all~$r$ with $s\preceq_{\alpha^*} r \prec_{\alpha^*}x$. Let $t\prec_{\alpha^*} x$ be shortest with $T^t(n^x, k^s(n^x),a)=1$ for some~$a$. By induction on the length of~$u$ with $s\preceq_{\alpha^*} u \prec_{\alpha^*} t$ we see that $s\sqsubseteq_{n^x+1} u$. Applying this to $u$ the $\prec^*$-predecessor of~$t$, we see that $n^t = n^x$. 

  \smallskip

  If $n^x$ is finite, then \ref{item:n_x_initial_segments:true_stages} follows from \ref{item:n_x_initial_segments:k_stabilises} and \ref{item:n_x_initial_segments:n_liminf}: there are infinitely many $s\prec_{\alpha^*}$ with $n^s = n^x$, and for all but finitely many of these, $s\sqsubseteq_{n^x} x$. 

  Suppose that $n^x= \w$. Let $m<\w$; by \ref{item:n_x_initial_segments:k_stabilises}, let $s$ be the shortest such that $s\sqsubset_m x$. By \cref{lem:basics_of_square_extensions}\ref{item:basics_of_square_extensions:hit_then_s_precisely}, $n^s = m$, so again $s\sqsubset_{n^s} x$.   
\end{proof}

\begin{definition}
Let $t\in 2^{\le \Leaves_\alpha}$. For finite $m\le n^t$ we let 
  \[
    c^t(m) = \min \left\{ |s| \,:\,  s\sqsubseteq_m t \right\} .
  \]
\end{definition}

\begin{lemma} \label{lem:basics_on_coding_lengths_c}
Let $t\in 2^{\le \Leaves_\alpha}$, and let $m \le n^t$ be finite. 
 \begin{sublemma}
   \item \label{item:basics_on_coding_lengths:at_most_length_of_t}
   $c^t(m) \le |t|$. 

   \item \label{item:basics_on_coding_lengths:at_least_m}
   $c^t(m)\ge m$. 

   \item \label{item:basics_on_coding_lengths:is_finite}
   $c^t(m)<\w$. 

   \item \label{item:basics_on_coding_lengths_c:lengths_fixed_under_m_extensions}
   If $s\sqsubseteq_m t$ then $c^s(m) = c^t(m)$. 

   \item \label{item:basics_on_coding_lengths:lengths_are_strictly_increasing}
   If $m<m'\le n^t$ then $c^t(m)< c^t(m')$. 
 \end{sublemma}
\end{lemma}

\begin{proof}
  Let $r$ be shortest with $r\sqsubseteq_m t$, so $c^t(m)= |r|$.

  \smallskip

  \ref{item:basics_on_coding_lengths:at_most_length_of_t} is immediate. For~\ref{item:basics_on_coding_lengths:at_least_m},  by \cref{lem:basics_of_square_extensions}\ref{item:basics_of_square_extensions:hit_then_s_precisely}, $n^r = m$; by \cref{lem:basics_on_n_r_and_k_r_m_for_general_alpha}\ref{item:n_and_k:n_t_bounded_by_length_of_t}, $m\le |r|$. 

  \ref{item:basics_on_coding_lengths:is_finite} follows from \cref{lem:_true_initial_segments_for_approximating_n_x}\ref{item:n_x_initial_segments:k_stabilises}. 

  \ref{item:basics_on_coding_lengths_c:lengths_fixed_under_m_extensions} follows from \cref{lem:basics_of_square_extensions}\ref{item:basics_of_square_extensions:diamond}, since we have $r\sqsubseteq_m s$, and~$r$ is shortest such. 

  \ref{item:basics_on_coding_lengths:lengths_are_strictly_increasing} follows from \cref{lem:basics_of_square_extensions}\ref{item:basics_of_square_extensions:hit_then_s_precisely}; if $r'$ is shortest with $r'\sqsubseteq_{m'} t$  then $n^{r'} = m' \ne m = n^r$; so $r'\ne r$. Since $r'\sqsubseteq_m t$ (\cref{lem:basics_of_square_extensions}\ref{item:basics_of_square_extensions:nested}), the minimality of~$r$ implies that $r\prec r'$. 
\end{proof}

\begin{lemma} \label{lem:the_longest_square_initial_segment}
  Let $t\in 2^{<\Leaves_\alpha}$, $t\ne \seq{}$. Suppose that $s\prec_{\alpha^*} t$ is the longest such that $n^s\le n^t$. 
  \begin{sublemma}
    \item \label{item:longest_square_initial_segment:square_extension}
    $s\sqsubseteq_{n^s} t$. 

    \item \label{item:longest_square_initial_segment:k_constant}
    The value $k^r(n^s)$ is constant for all $r$ with $s\prec_{\alpha^*} r \prec_{\alpha^*} t$. 

    \item \label{item:longest_square_initial_segment:when_n_increases}
    If $n^s<n^t$ then~$s$ is the $\prec_{\alpha^*}$-predecessor of~$t$, and $c^t(n^t) = |t|$. 
  \end{sublemma}
\end{lemma}

\begin{proof}
  \ref{item:longest_square_initial_segment:square_extension} follows from \cref{lem:basics_on_n_r_and_k_r_m_for_general_alpha}\ref{item:n_and_k:when_there_is_a_defect}. 

  For \ref{item:longest_square_initial_segment:k_constant}, let $r,r'$ with $s\prec_{\alpha^*} r \prec_{\alpha^*} r' \prec_{\alpha^*} t$. If $k^{r'}(n^s)\ne k^r(n^s)$ then by \cref{lem:basics_on_n_r_and_k_r_m_for_general_alpha}\ref{item:n_and_k:when_there_is_a_defect}, there is some $u$ with $r\prec_{\alpha^*} u \prec_{\alpha^*} r'$ and $n^u\le n^s\le n^t$, contradicting the maximality of~$s$. 

  \ref{item:longest_square_initial_segment:when_n_increases}: suppsoe that $n^s < n^t$. If there is some~$r$ with $s\prec_{\alpha^*} r \prec_{\alpha^*} t$, let $r$ be the shortest such, so that $s$ is the $\prec_{\alpha^*}$-predecessor of~$r$. By maximality of~$s$, $n^r > n^t > n^s$ so $n^r \ge n^s+2$, which is impossible. Note that this means that $n^t = n^s+1$. 

  To see that $c^t(n^t) = |t|$, suppose, for a contradiction, that there is some $r\sqsubset_{n^t} t$. Then $r\preceq_{\alpha^*} s$, so by \cref{lem:basics_of_square_extensions}\ref{item:basics_of_square_extensions:club}, $r\sqsubseteq_{n^t} s$, which is impossible since $n^s<n^t$. 
\end{proof}

\subsection{The graphs $H_\alpha$}

We continue to use the list $(M_k)$ of reals (\cref{def:M_k}) for coding finite strings. We modify \cref{def:W_alpha_and_L_alpha}: instead of coding an edge at $n^x$, we code it at the much higher number $c^x(n^x)$, and we similarly increase the lengths of the initial segments of~$y$ coded by $k^x(m)$ for $m<n^x$. 

\begin{definition} \label{def:Y_alpha} \
    \begin{sublemma}
    \item We let $\YYY_\alpha$ be the collection of $(x,y)\in \XXX_\alpha$ such that for all $m<n^x$, $k^x(m)$ is a code for $y\rest{(c^x(m)+1)}$.

    \item For $(x,y), (x',y')\in \YYY_\alpha$, we let $(x,y)$ be related to~$(x',y')$ in the directed graph~$H_\alpha$ if $x=x'$, $n^x<\w$, $y\symdiff y' = \{c^x(n^x)\}$, and $c^x(n^x)\notin y$.
  \end{sublemma}
\end{definition}

\begin{lemma} \label{lem:complexity_of_Y_alpha}
  The set $\YYY_\alpha$ is $\Pi^0_2(\tau_{\alpha^*})$. 
\end{lemma}

\begin{proof}
  Similar to the proof of \cref{prop:complexity_of_WWW_alpha}. We define $A_{\bar k}$ as in that proof, except that the third bullet point is:
  \begin{itemize}
    \item for the shortest $r\prec_{\alpha^*} x$ such that $n^r = m$ and for all $\ell< m$, $k^r(\ell) = k_\ell$, $k_m$ is a code for $y\rest{(|r|+1)}$.  \qedhere
  \end{itemize}
\end{proof}

In particular, $\YYY_3$ is $\Pi^0_2$ in the standard topology.

\subsection{Non-colorability when $\alpha$ is the successor of a successor}

Our aim is to design a notion of forcing similar to $\SSS$ from the previous section (\cref{def:notion_of_forcing_SSS}), with the purpose of obtaining generic points of $\YYY_\alpha$. The extra complication comes from needing to code at the correct lengths. Specifically, suppose that we are building a generic point $(x_G,y_G)\in \YYY_\alpha$, and let $m<n^{x_G}$. Let $k = k^{x_G}(m)$ and $c = c^{x_G}(m)$. Then $k$ must code $y_G\rest{(c+1)}$. Let $p = (u,\zeta)\in G$ be ``sufficiently large'', so that $n^u>m$, $|\zeta|>c$, and $k^u(m) = k$. Then the condition~$p$ should in some way ``know'' what~$c$ is, so that it can ensure that~$k$ codes $\zeta\rest{(c+1)}$. However, $c$ is computed by considering strings $s\sqsubset_m x_G$, in particular, strings $s\prec_{\alpha^*} x_G$. Thus, we will require that~$u$ has some witness that could play the role of such~$s$. Luckily, any two such witnesses will agree with each other, in that they will compute the same value of~$c$. 

For the following definition, recall the collection of conditions $\tilde \QQ$ (\cref{def:the_notion_of_forcing_tilde_Q}), and the numbers $n^u$ and $k^u(m)$ (\cref{def:k_u_and_n_u}).

\begin{definition} \label{def:supporting_conditions_in_tilde_QQ}
  We say that $s\in 2^{<\Leaves_\alpha}$ \emph{supports} a condition $u\in \tilde \QQ$ if:
  \begin{orderedlist}
    \item $T^s_{\alpha^*}\subseteq u$; and 
    \item $n^s\ge n^u$, and for all $m<n^u$, $k^s(m) = k^u(m)$.
  \end{orderedlist}
  We say that $u\in \tilde \QQ$ is \emph{supported} if some $s\in 2^{<\Leaves_\alpha}$ supports~$u$. 
\end{definition}

 Note that if $s\sqsubseteq_{n^u} t$ and~$t$ supports~$u$ then~$s$ supports~$u$ as well.

\begin{lemma} \label{lem:unique_supporter_class}
   Let $u,v\in \tilde \QQ$ be compatible, with $n^u\le n^v$. If~$s$ supports~$u$ and $t$ supports~$v$ then $s$ and~$t$ are $\sqsubseteq_{n^u}$-comparable.  
\end{lemma}

\begin{proof}
  Let $w = u\cup v$. We have $T^s_{\alpha^*}, T^t_{\alpha^*}\subseteq w$, and so $T^s_{\alpha^*}\cup T^t_{\alpha^*}$ is a function. By \cref{cor:you_wouldnt_think_this_would_be_useful}, $s$ and~$t$ are $\preceq_{\alpha^*}$-comparable. For all $m<n^u$, $k^s(m) = k^u(m) = k^v(m) = k^t(m)$, so~$s$ and~$t$ are $\sqsubseteq_{n^u}$-comparable. 
\end{proof}

By \cref{lem:basics_on_coding_lengths_c}\ref{item:basics_on_coding_lengths_c:lengths_fixed_under_m_extensions}, if $s$ and~$t$ both support~$u$, then for all $m\le n^u$, $c^s(m) = c^t(m)$. We therefore define:

\begin{definition} \label{def:c_u_m}
  If $u\in \tilde\QQ$ is supported, then for $m\le n^u$, we let $c^u(m) = c^s(m)$ for any (all) $s$ that support~$u$. 
\end{definition}

By \cref{lem:basics_on_coding_lengths_c}\ref{item:basics_on_coding_lengths_c:lengths_fixed_under_m_extensions} we get:

\begin{lemma} \label{lem:extension_of_supported_preserves_coding_lengths}
  If $u,v\in \tilde \QQ$ are supported and compatible, then for all $m\le \min\{n^u,n^v\}$, $c^u(m) = c^v(m)$. In particular, if $u\subseteq v$ then for all $m<n^u$, $c^u(m) = c^v(m)$. 
\end{lemma}

\begin{remark} \label{rmk:minimal_supporter}
  The \emph{minimal support} of a supported~$u$ is the $\sqsubseteq_{n^u}$-least $s$ that supports~$u$, which exists by  \cref{lem:unique_supporter_class}. By \cref{lem:basics_of_square_extensions}\ref{item:basics_of_square_extensions:hit_then_s_precisely}, for this~$s$ we have $n^s = n^u$. 
\end{remark}

\begin{lemma} \label{lem:union_and_restriction_of_supported_conditions_is_supported} \ 
\begin{sublemma}
  \item If $u,v\in \tilde \QQ$ are compatible, and are both supported, then $u\cup v$ is supported.
  \item If $u\in \tilde \QQ$ is supported, then for all $\beta\le \alpha$, $u\rest{\beta}$ is supported.
\end{sublemma}
\end{lemma}

\begin{proof}
  (a): By \cref{lem:tilde_Q_is_closed_under_unions}, $w = u\cup v\in \tilde \QQ$. Suppose that $n^u\le n^v$. The main point is that $n^w = n^v$; if $n^u<n^v$ then necessarily $u(n^u)\diverge$. Also, for all $m<n^v$, $k^w(m) = k^v(m)$ (note that for all $m<n^u$, $k^u(m) = k^v(m)$).  Hence, if~$t$ supports~$v$ then it also supports~$w$.

  (b): There are two cases. If $\beta \le \alpha^*$, then $n^{u\rest\beta}=0$: if $(u\rest{\beta})(\s)=1$ then $\rk(\s)<\alpha^*$, meaning that $|\s|\ge 3$; hence there is no~$m$ with $(u\rest{\beta})(m)\converge$. In this case the empty string supports $u\rest\beta$. 

  Suppose that $\beta > \alpha^*$. Then $n^{u\rest{\beta}}= n^u$: for all $m<n^u$, $\rk(m,k^u(m))\le \alpha^* <\beta$, so $(u\rest{\beta})(m,k^x(m))=1$, implying that $(u\rest{\beta})(m)=0$. This also shows that for all $m<n^u$, $k^{u\rest{\beta}}(m) = k^u(m)$. Further, if $T^s_{\alpha^*}\subseteq u$ then $T^s_{\alpha^*}\subseteq u\rest{\beta}$. Hence, any string that supports~$u$ also supports $u\rest{\beta}$. 
\end{proof}

\begin{lemma} \label{lem:supported_conditions_are_dense}
 For every $u\in \tilde \QQ$ there is some $v\supseteq u$ in~$\tilde \QQ$ such that:
 \begin{orderedlist}
    \item $v$ is supported;
    \item $n^v = n^u$; 
    \item If $u(n^u)\diverge$ then $v(n^v)\diverge$. 
  \end{orderedlist} 
\end{lemma}

\begin{proof}
  Choose any $x\in [u]$. By \cref{lem:_true_initial_segments_for_approximating_n_x}\ref{item:n_x_initial_segments:k_stabilises} and \ref{item:n_x_initial_segments:n_liminf}, find some $t\prec_{\alpha^*} x$ such that $n^t\ge n^u$ and for all $m<n^u$, $k^t(m) = k^x(m)$. Note that since $x\in [u]$, for $m<n^u$, $k^x(m) = k^u(m)$. Since $T^t_{\alpha^*}\subseteq T^x$ and $u\subseteq T^x$, $w = u\cup T^t_{\alpha^*}$ is a function and does not contain any contradictions (there is no~$\s$ with $w(\s) = 1 = w(\s)\conc k$ for some~$k$). By \cref{lem:true_stages:basic_properties}\ref{item:true_stages:basic_properties:finiteness}, $\dom w$ is finite. Since $u\in \tilde \QQ$ and every $\s\in \dom T^t_{\alpha^*}$ has height $\ge 3$, we can extend~$w$ to a condition in $v\in \QQ$ without adding any labels to nodes of height 0 or 1, or any 1-labels to  any nodes of height~2. This ensures that $v\in \tilde \QQ$ and has the desired properties; $t$ supports~$v$. 
\end{proof}

\begin{definition} \label{def:the_notion_of_forcing_R}
  Let $\RR$ be the collection of pairs $(u,\zeta)\in \tilde \QQ\times 2^{<\w}$  such that~$u$ is supported, $|\zeta|\ge c^u(n^u)$, and for all $m<n^u$, $k^u(m)$ is a code of $\zeta\rest{(c^u(m)+1)}$. 
\end{definition}

As above, pairs are ordered by co-ordinatewise extension. By \cref{lem:basics_on_coding_lengths_c}\ref{item:basics_on_coding_lengths:lengths_are_strictly_increasing}, the requirement $|\zeta|\ge c^u(n^u)$ implies $|\zeta|>c^u(m)$ for all $m<n^u$, so the string $\zeta\rest{(c^u(m)+1)}$ makes sense.

\begin{lemma} \label{lem:forcing_R:extension_lemma} \  
  \begin{sublemma}
    \item \label{item:forcing_R:extension_lemma:starting_condition}
    The condition~$p^*$ defined by $\dom u^{p^*} = \{ \rooot \}$, $u^{p^*}(\rooot)=1$, and $\zeta^{p^*} = \seq{}$ is in~$\RR$.
    \end{sublemma}

Suppose that $(u,\zeta)\in\RR$. 

    \begin{sublemma}[resume]
    \item \label{item:forcing_R:extension_lemma:extending_zeta}
    For all $\xi\succeq \zeta$, $(u,\xi)\in \RR$. 

    \item \label{item:forcing_R:extension_lemma:extending_a_1}
    Suppose that $u(n^u)= 1$.  Then for all $v\supseteq u$ in~$\tilde\QQ$,  $(v,\zeta)\in \RR$.

    \item \label{item:forcing_R:extension_lemma:adding_a_1}
    Suppose that $u(\rooot)\diverge$. Let~$v$ extend~$u$ by defining $v(n^u)=1$ (and $v(\rooot)=0$). Then $(v,\zeta)\in \RR$. 

    \item \label{item:forcing_R:extension_lemma:extending_a_void}
    Suppose that $u(n^u)\diverge$. Then there is some $(w,\xi)\in \RR$ extending $(u,\zeta)$ such that $w(n^w)\diverge$ and $|\xi|= c^w(n^w)$. 
  \end{sublemma}
\end{lemma}

\begin{proof}
Most are the same as in the proof of \cref{lem:forcing_S:extension_lemma}; for example, for  \ref{item:forcing_R:extension_lemma:extending_a_1}, again the point is that $n^u = n^v$, and $k^v(m) = k^u(m)$ for all $m<n^u$, so any~$s$ that supports~$u$ also supports~$v$. 

The argument for~\ref{item:forcing_R:extension_lemma:extending_a_void} is a little more elaborate. Let $(u,\zeta)\in \RR$. Let~$l$ be a large number. Note that since we are assuming that $\alpha = \alpha^*+2$ is the successor of a successor, the rank of every node of height~2 is~$\alpha^*$, in particular, requirement~(iv) of \cref{def:the_notion_of_forcing_tilde_Q} holds automatically. (This is not really crucial, but simplifis notation.) We choose a large number~$k$ such that $\zeta \prec M_k$. 

We first define~$v$ to be the extension of~$u$ defined by letting, for every $m$ with $n^u\le m <l$, 
\begin{itemize}
  \item $v(m)=0$, and  $v(m,k)=1$;
  \item For all $k'<k$,  $v(m,k')=0$ and $v(m,k',l)=1$. 
\end{itemize}
Since $u\in \tilde \QQ$ and $u(l)\diverge$, this is consistent with~$u$, and $v\in \tilde \QQ$. Note that $v(n^v)\diverge$. 

By \cref{lem:supported_conditions_are_dense}, extend~$v$ to a supported $w\in \tilde \QQ$ with $n^w = n^v=l$ and $w(n^w)\diverge$. Let $\xi  = M_k\rest{c^w(l)}$. Then $(w,\xi)$ is as required: if $m<n^u$, then as $c^w(m) = c^u(m)$ (\cref{lem:extension_of_supported_preserves_coding_lengths}), and $k^w(m) = k^u(m)$ codes $\zeta\rest{(c^u(m)+1)} = \xi\rest{(c^w(m)+1)}$. If $n^u\le m < l$, then $k^w(m)=k$ codes every initial segment of~$\xi$, in particular, $\xi\rest{(c^w(m)+1)}$. 
\end{proof}

\begin{remark} 
  The proof of \cref{lem:forcing_R:extension_lemma}\ref{item:forcing_R:extension_lemma:extending_a_void} is where we need the flexible coding mechanism given by the sequence $(M_k)$: we need to choose some~$k$ that will code $\xi\rest{(c^w(k)+1)}$, without knowing what $c^w(k)$ is going to be. 
\end{remark}

\begin{lemma} \label{lem:forcing_R:totality}
  For all $\s\in T_\alpha$, the collection of $q\in \RR$ such that $u^q(\s)\converge$ is dense in~$\RR$. 
\end{lemma}

\begin{proof}
  Let $p = (u,\zeta)\in \RR$, and let $\s\in T_\alpha$. As in the proof of \cref{lem:forcing_S:totality}, if $\s = \rooot$, then \cref{lem:forcing_R:extension_lemma}\ref{item:forcing_R:extension_lemma:adding_a_1} allows us to add~$\s$ to the domain of~$u$. Also, if $u(n^u)\converge$ then \ref{item:forcing_R:extension_lemma:extending_a_1} allows us to similarly add any~$\s$. Suppose that $u(n^u)\diverge$. If $|\s|\ge 2$ then we can extend~$u$ to $v$ by setting $v(\s)=0$ and $v(\s\conc l)=1$ for some large~$l$, so that $n^v = n^u$, and $q=(v,\zeta)$ is as required. So to complete the proof of this lemma, it suffices to show that there is some $q\in \RR$ extending $p$ with $n^{u^q}>n^u$. By \ref{item:forcing_R:extension_lemma:extending_a_void}, we may assume that $|\zeta| = c^s(n^u)$. By \ref{item:forcing_R:extension_lemma:extending_zeta}, extend~$\zeta$ to any longer string, say $\zeta\conc 0$. Now apply \ref{item:forcing_R:extension_lemma:extending_a_void} to the condition $(u,\zeta\conc 0)$ to obtain $q = (w,\xi)\in \RR$ with $|\xi| = c^t(n^w)$. Since $c^w(n^u) = c^u(n^u) = |\zeta|<|\xi|$ (\cref{lem:extension_of_supported_preserves_coding_lengths}), we must have $n^w>n^u$, as required. 
\end{proof}

\begin{lemma} \label{lem:forcing_R:compatibility}
  Two conditions $p,q\in \RR$ are compatible in~$\RR$ if and only if $u^p\cup u^q$ is a function and $\zeta^p$ and $\zeta^q$ are comparable. 
\end{lemma}

\begin{proof}
  Similar to the proof of \cref{lem:forcing_S:compatibility}, using \cref{lem:union_and_restriction_of_supported_conditions_is_supported,lem:extension_of_supported_preserves_coding_lengths}. 
\end{proof}

Similarly, we obtain the analogue of \cref{lem:forcing_S:end_up_in_WWW}:

\begin{lemma} \label{lem:forcing_R:end_up_in_YYY}
  If $G\subset \RR$ is sufficiently generic, then:
\begin{orderedlist}
  \item For all finite $m\le n^{x_G}$, $c^{x_G}(m) = c^{u^p}(m)$ for some (or any) $p\in G$ with $n^{u^p}\ge m$.
  \item $(x_G,y_G)\in \YYY_\alpha$.
  \item For all $p\in \RR$, $(x_G,y_G)\in [p]$ if and only if $p\in G$.
\end{orderedlist}
\end{lemma}

\begin{proof}
  As in the proof of \cref{lem:forcing_S:end_up_in_WWW}, we have $(x_G,y_G)\in \XXX_\alpha$, and property~(iii) above holds. 

  For~(i) and~(ii),  let $m\le n^{x_G}$ be finite; let $p = (u,\zeta)\in G$ with $n^{u}\ge m$. Let $s$ support~$p$. Since $T^s_{\alpha^*}\subseteq u\subset T^{x_G}$, we have $s\prec_{\alpha^*} x_G$. For each $\ell<n^u$, since $u\subset T^{x_G}$, we have $k^u(\ell) = k^{x_G}(\ell)$. Since $s$ supports~$u$, $k^u(\ell) = k^s(\ell)$ for such~$\ell$. This shows that $s\sqsubseteq_{n^u} x_G$. By \cref{lem:basics_on_coding_lengths_c}\ref{item:basics_on_coding_lengths_c:lengths_fixed_under_m_extensions}, for all $m\le n^u$,  $c^u(m)$, which is defined to be $c^s(m)$, equals $c^{x_G}(m)$. 

  For all $\ell<n^u$, since $u\in \RR$, $k^{x_G}(\ell) = k^u(\ell)$ codes $\zeta\rest{(c^{x_G}(\ell)+1)}$, and $\zeta \prec y_G$. 
\end{proof}

We define $\beta$-complete conditions as usual. 

\begin{lemma} \label{lem:forcing_R:extending_to_beta_complete}
  For every $\beta \le \alpha+1$, every $p\in \RR$ can be extended to a $\beta$-complete condition in~$\RR$. 
\end{lemma}

\begin{proof}
  The proof of \cref{lem:forcing_S:extending_to_beta_complete} applies; since we are assuming that $\alpha = \alpha^*+2$, we do not need to consider the case $|\s|\le 1$. 
\end{proof}

We define $p\rest{\beta}$ as usual. 

\begin{lemma} \label{lem:forcing_R:restrictions_of_conditions_stay_in_R}
  For all $\beta \le \alpha+1$, for every $p\in \RR$, $p\rest{\beta}\in \RR$.
\end{lemma}

\begin{proof}
  Let $p = (u,\zeta)\in \RR$. By \cref{lem:union_and_restriction_of_supported_conditions_is_supported}, $u\rest{\beta}$ is supported. If $m\le n^{u\rest\beta}$ then by \cref{lem:extension_of_supported_preserves_coding_lengths} (or the proof of \cref{lem:union_and_restriction_of_supported_conditions_is_supported}), $c^{u\rest{\beta}}(m) = c^u(m)$, and when $m<n^{u\rest\beta}$, $k^{u\rest{\beta}}(m) = k^u(m)$, so codes correctly; we conclude that $p\rest\beta = (u\rest\beta,\zeta)\in \RR$. 
\end{proof}

We can now continue with analogues of \cref{lem:forcing_S:untagging:extension_lemma} and of \cref{prop:simple_forcing:Sigma_untagging_lemma,lem:K_alpha:Sigma_untagging_lemma}. Using the tools above, we can mimic the proof of \cref{prop:L_alpha_is_non_colourable} to obtain:

\begin{theorem} \label{prop:H_alpha_is_not_colourable}
  If $\alpha\ge 3$ is the successor of a successor ordinal, then the graph $H_\alpha$ does not have a countable $\bSigma^0_\alpha$ colouring. 
\end{theorem}

\subsection{Non-minimality of $L_3$}

We show that $L_3$ is not least among graphs with no $\bSigma^0_3$ colourings:

\begin{proposition} \label{prop:L_3_is_not_minimal}
  There is no continuous homomorphism from $L_3$ into $H_3$. 
\end{proposition}

The argument will be an elaboration on the proof of \cref{prop:no_embedding_of_K_3_into_L_3}. The difficulty is that we must construct elements of $\WWW_3$, rather than $\XXX_3$; that is, we need to make sure that we are coding correctly. Our advantage over the opponent is that we can keep changing our mind between, say, coding at $n=0$ and coding at some large value~$\ell^*$, whereas if the opponent stops coding at some current version of $c(n^*)$, and later wants to go back to $c(n^*)$, the new version of $c(n^*)$ is forced to be large. 

We need an elaboration on \cref{lem:alpha_is_three:extending_s_to_have_desired_n}. 

\begin{lemma} \label{lem:alpha_is_three:more_on_possibilities_for_extensions}
  Let $s\in 2^{<\Leaves_3}$; let $m\ge n^s$, and let $\zeta\in 2^{<\w}$ with $|\zeta|= m$. There is some $t\in 2^{<\Leaves_3}$ such that:
  \begin{itemize}
    \item $n^t = m$; 
    \item $s\sqsubseteq_{n^s} t$; 
    \item For all $\ell$ with $n_s \le \ell < m$, $k^t(\ell)$ codes $\zeta\rest{(\ell+1)}$.
  \end{itemize}
\end{lemma}

The proof is similar; we can always just extend by mostly~0's, and place 1's in the correct places to increase $k^t(\ell)$ to a desired value if necessary, and bring down~$n^t$ to the required value.

\begin{proof}[Proof of \cref{prop:L_3_is_not_minimal}]
  Let $F\colon \WWW_3\to \YYY_3$ be continuous. As in the proof of \cref{prop:no_embedding_of_K_3_into_L_3}, we define a sequence $(s_\ell)$, define $t^a_\ell, \zeta^a_\ell$, and $t^b_\ell, \zeta^b_\ell$ so that $F[s_\ell,0^\ell]\subseteq [t^a_\ell, \zeta^a_\ell]$ and $F[s_\ell,1\conc0^{\ell-1}]\subseteq [t^b_\ell,\zeta^b_\ell]$. It will be convenient for us to have $|\zeta^a_\ell|\le |t^a_\ell|$, so we can define $t^a_\ell$  as above, and then let $\zeta^a_\ell$ be the longest~$\zeta$ such that $|\zeta|\le |t^a_\ell|$ and $F[s_\ell,0^\ell]\subseteq [t^a_\ell, \zeta]$.

  Further, if the construction ever leaves phase~1 below, say at stage~$\ell^*$, then we will analogously define $(t^d_\ell, \zeta^d_\ell)$ such that $F[s_\ell,0^{\ell^*}\conc 1\conc 0^\w]\subseteq [t^d_\ell, \zeta^d_\ell]$.

  To avoid repetition, we define the following. 
  \begin{itemize}
    \item \emph{Working toward~$0$ at stage~$\ell$} means extending $s_\ell$ to $s_{\ell+1}$ satisfying $n^{s_{\ell+1}}=0$. 

    \item Suppose that $\ell^*\in \Nat$, and that $w$ is a binary string of length~$\ell^*$. A string $t\in 2^{<\Leaves_3}$ is \emph{$(\ell^*,w)$-admissible} if $n^t = \ell^*$ and for all $m<\ell^*$, $k^t(m)$ codes $w\rest{(m+1)}$. 

    \emph{Working toward $(\ell^{*},w)$} (at stage~$\ell$ of the construction) means extending~$s_{\ell}$ to a string $s_{\ell+1}$ that is $({\ell^*},w)$-admissible, and further, if $s_{\ell}$ itself is already $(\ell^*,w)$-admissible, then $s_{\ell}\sqsubset_{\ell^*} s_{\ell+1}$. 
  \end{itemize}

  Below we describe the phases of the construction. To avoid clutter, instead of specifying a separate phase analogous to phase 3 of the previous construction, we just declare that if we ever see that $t^a_\ell$ and $t^b_\ell$ are incomparable, then we move to a terminal phase at which at every stage we work toward~$0$. Hence, for the rest of the description of the construction, we assume that  $t^a_\ell$ and $t^b_\ell$ are comparable. 

  Similarly, if the construction ever leaves phase~1, say at stage~$\ell^*$, and we later see that $t^a_\ell$ and $t^d_\ell$ are incomparable, then the construction enters a special, terminal phase at which we always work toward $(\ell^*,0^{\ell^*})$. Hence, for the rest of the description of the description of the construction, after we leave phase~1, we assume that  $t^a_\ell$ and $t^d_\ell$ are comparable. 

  \smallskip
\noindent\textit{Phase 1: $\zeta^a_\ell$ and $\zeta^b_\ell$ are comparable.} Work toward~0.

  \smallskip

  We leave phase~1 when we see that $\zeta^a_\ell$ and $\zeta^b_\ell$ are incomparable. If this happens, we define:
  \begin{itemize}
    \item $c_0 = \min (\zeta^a_\ell\symdiff \zeta^b_\ell)$; 
    \item $r_0 = t^a_\ell \rest{c_0}$; and
    \item $n_0 = n^{r_0}$. 
  \end{itemize}

  If there is some $r\sqsubset_{n_0} r_0$ then we enter a terminal phase in which we work toward~0. Also, if we ever see that $r_0 \nsqsubset_{n_0} t^a_\ell$, we enter a terminal phase in which we work toward~0. (Observe that by \cref{lem:basics_of_square_extensions}\ref{item:basics_of_square_extensions:club}, once we see that $r_0\nsqsubset_{n_0} t^a_{\ell}$, the same holds for all $\ell'>\ell$.)

  \smallskip
\noindent\textit{Phase 2: $\zeta^a_\ell$ and $\zeta^d_\ell$ are comparable.} Let $\ell^*$ be the stage at which the construction leaves phase~1. Work toward $(\ell^*, 0^{\ell^*})$. 

\smallskip

We leave phase~2 when we see that $\zeta^a_\ell$ and $\zeta^d_\ell$ are incomparable. If this happens, we define:
\begin{itemize} 
    \item $c_1 = \min (\zeta^a_\ell\symdiff \zeta^d_\ell)$; 
    \item $r_1 = t^a_\ell \rest{c_1}$; and
    \item $n_1 = n^{r_1}$. 
\end{itemize}

If there is some $r\sqsubset_{n_1} r_1$ then we enter a terminal phase in which we work toward $(\ell^*,0^{\ell^*})$. Also similarly to above, if we ever see that $r_1 \nsqsubset_{n_1} t^a_\ell$, we enter a terminal phase in which we work toward $(\ell^*,0^{\ell^*})$. If this is not the case, we stay in phase 3:

\smallskip

\noindent\textit{Phase 3.} Work toward~0. 

\smallskip

For the verification, we let $z= \bigcup_\ell s_\ell$, $a = (z,0^\w)$,  $b = (z,1\conc 0^\w)$, and
if the construction ever leaves phase~1, we let $d = (z,0^{\ell^*}\conc 1 \conc 0^\w)$. We observe: 
\begin{itemize}
  \item If in the final phase of the construction we work toward $(\ell^*, 0^{\ell^*})$ then $n^{z} = \ell^{*}$ and for all $m<\ell^*$, $k^z(m)$ codes $0^{m+1}$ (this follows from \cref{lem:_true_initial_segments_for_approximating_n_x}\ref{item:n_x_initial_segments:k_stabilises},\ref{item:n_x_initial_segments:n_liminf}). Hence, in this case, both $a$ and~$d$ are in~$\WWW_3$, and $a$ is connected to~$d$ by an edge of $L_3$.

  \item If in the final phase of the construction we work toward $0$, then $n^z =0$, and in this case, $a$ and~$b$ are in $\WWW_3$ and are connected by an edge of~$L_3$. 
\end{itemize}

We define $F(a)= (x^a,y^a)$, $F(b)= (x^b,y^b)$. We obtain ``easy victory'' if $x^a\ne x^b$, or if we never leave phase~1.  Suppose that we do leave phase~1 at stage $\ell^*$, and that $x^a = x^b = x$. We observe:
\begin{itemize}
  \item If $F(a)$ and $F(b)$ are connected by an edge of $H_3$ then $n^x = n_0$ and $c^x(n_0) = c_0$, so that $r_0$ is the shortest~$r$ with $r\sqsubset_{n_0} x$. This follows from $c^x(n^x)=c_0$, and \cref{lem:basics_of_square_extensions}\ref{item:basics_of_square_extensions:hit_then_s_precisely}. 
\end{itemize}
Hence, if there is some $r\sqsubset_{n_0} r_0$, then $a$ and~$b$ are connected by an edge, and $F(a)$ and $F(b)$ are not. Similarly, if $r_0\nsqsubset_{n_0} x$, i.e., if we ever see that $r_0\nsqsubset_{n_0} t^a_\ell$, then we win by the same outcome. Suppose that this is not the case. 

\smallskip

Since we are assuming that we do leave phase~1, let $F(d)= (x^d,y^d)$. If $x^d\ne x^a$, or if we never leave phase~2, then we again obtain easy victory. So we assume that $x^a = x^b = x^d = x$, and that we leave phase~2 at some stage. Similarly to above:
\begin{itemize}
  \item If $F(a)$ and $F(d)$ are connected by an edge then $n^x = n_1$ and $r_1$ is shortest with $r_1 \sqsubset_{n_1} x$.
\end{itemize}
Similarly to above, we may assume that $r_1\sqsubset_{n_1} x$, and is the shortest such. 

\smallskip

Now, the thing to observe is that $c_1> c_0$. This is because, as in the previous construction, $y^a$ and~$y^d$ must both extend $\zeta^a_{\ell^*}$, and $|\zeta^a_{\ell^*}|> c_0$. Hence, $r_0 \prec r_1 \prec x$.

\smallskip

So we are assuming that the final phase of the construction is phase~3;  $r_0 \sqsubset_{n_0} x$ and $r_1\sqsubset_{n_1} x$. By \cref{lem:basics_of_square_extensions}\ref{item:basics_of_square_extensions:club}, $r_0 \sqsubset_{n_0} r_1$. The minimality of~$r_1$ implies that $n_0 < n_1$. The fact that $r_1\sqsubset_{n_1} x$ implies $n^x \ge n_1$, in particular, $n^x\ne n_0$. Hence, $F(a)$ and $F(b)$ are not connected by an edge, whereas $a$ and~$b$ are. 
\end{proof}

\section{An embedding result} \label{sec:an_embedding_result}

In this section we show that $H_3$ is minimal for non-$\bSigma^0_3$-colourability. 

\begin{theorem} \label{thm:minimality_of_H_3}
  Let $X$ be a computably presented Polish space, let $G$ be a $\Sigma^1_1$ directed graph on~$X$, and suppose that there is no countable $\bSigma^0_3$ colouring of~$G$. Then there is a  continuous graph homomorphism from  $(\YYY_3, H_3)$ to $(X,G)$.
\end{theorem}

The argument is similar to the one given in \cite{LecomteZeleny}; we give it for completeness.

\subsubsection*{Preparation}

To define the continuous map (and further witnesses for its success), we define a collection of initial segments of elements of $\YYY_3$. 

\begin{definition} \label{def:TTT}
  We let $\TTT$ be the collection of pairs $(t,\zeta)\in 2^{<\Leaves_3}\times 2^{<\w}$ satisfying:
\begin{orderedlist}
  \item $|\zeta| =|t|+1$; 
  \item For all $m<n^t$, $k^t(m)$ is a code of $\zeta\rest{(c^t(m)+1)}$. 
\end{orderedlist}
\end{definition}

Observe that by \cref{lem:basics_on_coding_lengths_c}\ref{item:basics_on_coding_lengths:at_most_length_of_t}, for all $m<n^t$, $c^t(m)\le |t|$, so $|\zeta|>|t|$ ensures that the strings $\zeta\rest{(c^t(m)+1)}$ make sense.

For brevity, for $p=(t^p,\zeta^p)\in \TTT$ we write:
  \begin{itemize}
    \item $|p| = |t^p|$; 
    \item $n^p=n^{t^p}$;
    \item For $m<n^p$, $k^p(m)=k^{t^p}(m)$; 
    \item For $m\le n^p$, $c^p(m)=c^{t^p}(m)$. 
  \end{itemize}

We also define, for $p,q\in \TTT$,
\begin{itemize}
    \item $p\preceq q$ if $t^p\preceq t^q$ and $\zeta^p\preceq \zeta^q$; 
    \item $p\sqsubseteq q$ if $p\preceq q$ and further, $t^p\sqsubseteq_{n^p} t^q$. 
  \end{itemize}  

Note that there are two conditions $p\in \TTT$ with $|p|= 0$, namely, $(\seq{}, \seq{0})$ and $(\seq{}, \seq{1})$; $n^{\seq{}}=0$, so no coding is required.  For any $q\in \TTT$, If $\zeta^q\succeq \seq{i}$, then $(\seq{}, \seq{i})\sqsubseteq q$. By \cref{lem:basics_on_coding_lengths_c}\ref{item:basics_on_coding_lengths_c:lengths_fixed_under_m_extensions}, if $p\sqsubseteq q$ then for all $m\le n^p$, $c^p(m) = c^q(m)$.

\begin{lemma} \label{lem:TTT:sqsubset_extensions}
Let $p,q,r\in \TTT$. 
  \begin{sublemma}
     \item \label{item:TTT:sqsubset:transitive} 
    $\sqsubseteq$ is transitive.


     \item  \label{item:TTT:sqsubset:club}
     If $p\preceq q \preceq r$ and $p\sqsubseteq r$ then $p\sqsubseteq q$.
   \end{sublemma}
\end{lemma}

\begin{proof}
  \ref{item:TTT:sqsubset:transitive}: Suppose that $p\sqsubseteq q \sqsubseteq r$, so $p\preceq q \preceq r$ and $t^p\sqsubseteq_{n^p} t^q \sqsubseteq_{n^q} t^r$. Then $p\preceq r$. Since $n^p\le n^q$, we have $t^q \sqsubseteq_{n^p} t^r$ (\cref{lem:basics_of_square_extensions}\ref{item:basics_of_square_extensions:nested}); now apply \cref{lem:basics_of_square_extensions}\ref{item:basics_of_square_extensions:transitive}.

  \ref{item:TTT:sqsubset:club} follows from \cref{lem:basics_of_square_extensions}\ref{item:basics_of_square_extensions:club}.
\end{proof}

For the following lemma, we extend the relations $p\preceq q$ and $p\sqsubseteq q$ to elements of~$\YYY_3$.

\begin{lemma} \label{lem:points_have_initial_segments_in_TTT} \ 
  \begin{sublemma}
      \item \label{item:TTT:initial_segments_in_TTT:in_TTT}
      Suppose that $p\in \TTT$ or $p = (t^p,\zeta^p)\in \YYY_3$. If $s\sqsubset_{n^s} t^p$ then $q=(s,\zeta^p\rest{(|s|+1)})$ is in~$\TTT$ (and $q\sqsubseteq p$). 
  
      \item \label{item:TTT:initial_segments_in_TTT:existence_of_true_stages}
      For every $a\in \YYY_3$ there are infinitely many $p\sqsubset a$ in~$\TTT$. 
  \end{sublemma}
\end{lemma}

\begin{proof}
  For \ref{item:TTT:initial_segments_in_TTT:in_TTT}, suppose that $s\sqsubset_{n^s} t^p$. Let $m<n^s$. Then $k^s(m) = k^{t^p}(m)$, and by \cref{lem:basics_on_coding_lengths_c}\ref{item:basics_on_coding_lengths_c:lengths_fixed_under_m_extensions}, $c^s(m) = c^{t^p}(m)$. Since $p\in \TTT\cup \YYY_3$, $k^{t^p}(m)$ codes $\zeta^p\rest{(c^{t^p}(m)+1)}$. This shows that $(s,\zeta^p\rest{(|s|+1)})\in \TTT$. 

  \ref{item:TTT:initial_segments_in_TTT:existence_of_true_stages} follows from \ref{item:TTT:initial_segments_in_TTT:in_TTT} and  \cref{lem:_true_initial_segments_for_approximating_n_x}\ref{item:n_x_initial_segments:true_stages}.  
\end{proof}

\begin{definition} \label{def:TTT:v_p_and_p_hat}
Let $p\in \TTT$. 
\begin{sublemma}
  \item We let $\hat p$ be the pair $(t^p,\hat \zeta^p)$, where $|\hat \zeta^p| = |\zeta^p|$ and $\hat  \zeta^p\triangle \zeta^p  = \{c^p(n^p)\}$. 

  \item We call $p$ a \emph{lefty} if $\zeta(c^p(n^p))=0$, a \emph{righty} otherwise.   

  \item If $|p|>0$, we let $v(p)$ be the longest $q\sqsubset p$ in~$\TTT$. 
\end{sublemma}
\end{definition}

Note that $\hat{\hat{p}} = p$. The definition of $v(p)$ makes sense, since if $|p|>0$ then there is some $q\in \TTT$ such that $q\sqsubset p$, namely one of $(\seq{}, \seq{0})$ or $(\seq{},\seq{1})$. 

For the proof of the following lemma, and below, we recall that since we are working with $\alpha=3$, we have $\alpha^*=1$; so for $s,t\in 2^{<\Leaves_3}$, $s\prec_{\alpha^*} t$ means $s\prec t$ (\cref{lem:true_stages:basic_properties}\ref{item:true_stages:basic_properties:zero}).

\begin{lemma} \label{lem:TTT:v_p_and_p_hat_properties}
  Let $p\in \TTT$ with $|p|>0$. 
  \begin{sublemma}
    \item \label{item:TTT:v_p_p_hat:longest_with_n_q_below_n_p}
    If $v(p)\prec q \prec p$ then $n^q>n^p$. 

    \item \label{item:TTT:v_p_p_hat:same_parity_when_same_n}
    If $q\sqsubseteq p$ and $n^q = n^p$ then $p$ and~$q$ have the same orientation (they are both lefties or both righties). 

    \item \label{item:TTT:v_p_p_hat:when_v_p_is_same}
    If $n^{v(p)} = n^p$ then $v(\hat{p}) = \hat{v(p)}$, and either $v(p)$ is the $\prec$-predecessor of $p$, or $v(\hat{p})$ is the $\prec$-predecessor of $\hat{p}$. 

    \item \label{item:TTT:v_p_p_hat:when_v_p_is_smaller}
    If $n^{v(p)}<n^p$ then $v(\hat p) = v(p)$, and $v(p)$ is the $\prec$-predecessor of both~$p$ and~$\hat{p}$. 
  \end{sublemma}
\end{lemma}

\begin{proof}
  By \cref{lem:points_have_initial_segments_in_TTT}\ref{item:TTT:initial_segments_in_TTT:in_TTT}, $t^{v(p)}$ is the longest~$s$ with $s\sqsubset_{n^s} t^p$. Then~\ref{item:TTT:v_p_p_hat:longest_with_n_q_below_n_p} follows from \cref{lem:the_longest_square_initial_segment}\ref{item:longest_square_initial_segment:square_extension}.  Observe also that $|v(p)|$ only depends on $t^p$, hence, $t^{v(p)} = t^{v(\hat p)}$, in particular, $|v(p)| = |v(\hat{p})|$. 

  \smallskip

 \ref{item:TTT:v_p_p_hat:same_parity_when_same_n} follows from the fact that $c^q(n^q) = c^p(n^p)$ (\cref{lem:basics_on_coding_lengths_c}\ref{item:basics_on_coding_lengths_c:lengths_fixed_under_m_extensions}) and $\zeta^q\preceq \zeta^p$. 

 \smallskip

 \ref{item:TTT:v_p_p_hat:when_v_p_is_same}: suppose that $n^{v(p)}= n^p$. Let $n = n^p$ and $c = c^p(n)$. By~\ref{item:TTT:v_p_p_hat:same_parity_when_same_n}, $c^{v(p)}(n) = c$. Since $|v(p)| = |v(\hat{p})|$, we get that $v(\hat{p}) = \hat{v(p)}$. 

 For any $s$ with $t^{v(p)}\prec s \prec t^p$ we have $n^s>n$ and $t^{v(p)}\sqsubset_{n} s$ (\cref{lem:the_longest_square_initial_segment}\ref{item:longest_square_initial_segment:square_extension} and \cref{lem:basics_of_square_extensions}\ref{item:basics_of_square_extensions:club}), so $c^s(n) = c$ (\cref{lem:basics_on_coding_lengths_c}\ref{item:basics_on_coding_lengths_c:lengths_fixed_under_m_extensions} again). By \cref{lem:the_longest_square_initial_segment}\ref{item:longest_square_initial_segment:k_constant}, let~$k$ be the constant value $k^s(n)$ for all~$s$ with $t^{v(p)}\prec s \prec t^p$. 

 Suppose that $v(p)$ is not the $\prec$-predecessor of~$p$ in~$\TTT$, i.e., that there is some $q\in \TTT$ with $v(p)\prec q\prec p$. Since $q\in \TTT$, $n^q>n$, $c^q(n)=c$ and $k^q(n)= k$, $k$ codes $\zeta^p\rest{(c+1)}$. For any~$s$ with $t^{v(p)}\prec s \prec t^p$, $k = k^s(n)$ cannot code $\zeta^{\hat{p}}\rest{(c+1)}$, since $\zeta^p(c)\ne \zeta^{\hat{p}}(c)$. Hence, $(s,\zeta^{\hat{p}}\rest{(|s|+1)})$ cannot be in~$\TTT$. This shows that $v(\hat{p})$ is the $\prec$-predecessor of~$\hat{p}$ in~$\TTT$. 

 \smallskip

 \ref{item:TTT:v_p_p_hat:when_v_p_is_smaller}: suppose that $n^{v(p)}< n^p$. By \cref{lem:the_longest_square_initial_segment}\ref{item:longest_square_initial_segment:when_n_increases}, $c^p(n^p)= |p|$. Since $|v(p)|<|p|$, this implies that $\zeta^{p}\rest{(|v(p)|+1)} = \zeta^{\hat{p}}\rest{(|v(p)|+1)}$, so $v(\hat{p}) = v(p)$. Further, \cref{lem:the_longest_square_initial_segment}\ref{item:longest_square_initial_segment:when_n_increases} says that $t^{v(p)}$ is the $\prec$-predecessor of $t^p$, i.e., $|p| = |v(p)|+1$, whence there cannot be any~$q$ with $v(p)\prec q\prec p$ (or $\prec \hat{p}$). 
\end{proof}

\subsubsection*{Construction}

Suppose that $X$ is a computably presented Polish space, $G$ is a $\Sigma^1_1$ directed graph on~$X$, and that there is no countable $\bSigma^0_3$ colouring of~$G$.  Following \cite[Theorem~5.1]{LecomteZeleny}, we obtain a nonempty $\bPi^0_2$, $\Sigma^1_1$ set $Y\subseteq X$ with the property that for every closed, $\Sigma^1_1$ set $V\subseteq X$, if $V\cap Y\ne \emptyset$ then $V\cap Y$ is not $G$-independent. 

We replace $G$ by $G\cap Y^2$ (note that this keeps the edge relation $\Sigma^1_1$), so we assume that $G\subseteq Y^2$.

For each $p\in \TTT$ we will define:

\begin{itemize}
  \item A point $x_p\in X$;
  \item A rational open ball $X_p\subseteq X$;
  \item An effectively closed set $D_p\subseteq X^2\times \w^\w$. 
\end{itemize}
We will ensure that:
\begin{orderedlist} 
  \item \label{item:construction:edge_condition:x_p_in_X_p}
  $x_p\in X_p$.

  \item \label{item:construction:edge_condition:D_p_is_symmetric}
  $D_p = D_{\hat p}$.
  
  \item \label{item:construction:edge_condition:U_p_in_G}
  The projection~$U_p$ of $D_p$ onto~$X^2$ is a subset of~$G$.
  
  \item \label{item:construction:edge_condition:x_p_x_p_hat_is_in_U_p}
  If~$p$ is a lefty, then $(x_p,x_{\hat p})\in U_p$.
  
  \item \label{item:construction:edge_condition:U_p_contained_in_X_p_times_X_p_hat}
  If~$p$ is a lefty, then $U_p\subseteq X_p\times X_{\hat p}$.
  
  \item \label{item:construction:edge_condition:diameters}
  If $|p|>0$ then the diameters of $X_p$ and of~$D_p$ are $\le 1/|p|$. 

  \item \label{item:construction:edge_condition:simple_extension}
   If $p\prec q$ then $\overline{X}_q\subseteq X_p$. 

  \item \label{item:construction:2-extensions}
  If $q\sqsubseteq p$ and $n^p = n^q$ then $D_p\subseteq D_q$. 

  \item  \label{item:construction:main_closure_condition}
  If $q\sqsubseteq p$ and $n^q < n^p$, and~$q$ is a lefty, then every (left or right) endpoint of an edge in~$U_p$ is a limit point of points which are left endpoints of edges in~$U_q$; analogously if~$q$ is a righty. 
\end{orderedlist}

Note that \ref{item:construction:edge_condition:D_p_is_symmetric} and \ref{item:construction:edge_condition:x_p_x_p_hat_is_in_U_p} imply that if~$p$ is a righty, then $(x_{\hat{p}},x_p)\in U_p$, and similarly, \ref{item:construction:edge_condition:U_p_contained_in_X_p_times_X_p_hat} implies that when~$p$ is a righty, $U_p\subseteq X_{\hat{p}}\times X_p$. Note that~\ref{item:construction:edge_condition:x_p_x_p_hat_is_in_U_p} implies that for all~$p$, $x_p\in Y$, since we modified~$G$ so that $G\subseteq Y^2$. Also observe that~\ref{item:construction:edge_condition:x_p_x_p_hat_is_in_U_p} and~\ref{item:construction:main_closure_condition} imply that if $q\sqsubseteq p$ and $n^q < n^p$, then both $x_p$ and $x_{\hat{p}}$ are limits of points that are left endpoints of edges in~$U_q$ (or right, depending on the orientation of~$q$). 

\medskip

Let $p\in \TTT$; we suppose that the construction has been performed for all $q\prec p$ in~$\TTT$. We will consider both~$p$ and~$\hat{p}$ at the same time. There are several cases.

\medskip

\noindent{\textit{First case:}} $|p|=0$. We let $D_{p} = D_{\hat p}$ be an effectively closed set projecting to~$G$; $X_{p} = X_{\hat p} = X$ and we choose $x_p,x_{\hat{p}}$ so that $(x_p,x_{\hat{p}})\in G$, where $p = (\seq{}, \seq{0})$ is the lefty condition with $|p|=0$.

\medskip

Suppose that $|p|>0$, so $v(p)$ and $v(\hat p)$ are defined. 

\smallskip

\noindent{\textit{Second case:}} $n^{v(p)} < n^p$. This situation will be symmetric between~$p$ and~$\hat{p}$, so suppose that~$p$ is a lefty. 

For each~$q$ with $q \sqsubseteq v(p)$, 
\begin{itemize}
  \item If~$q$ is a lefty, let $K_q$ be the collection of left end-points of edges in~$U_q$;
  \item If~$q$ is a righty, let $K_q$ be the collection of right end-points of edges in~$U_q$. 
\end{itemize}

Let~$Z$ be a rational open ball centered at $x_{v(p)}$ whose closure is contained in $X_{v(p)}$. Let 
\[
R =  Y\cap \overline{Z} \cap \bigcap_{q\sqsubseteq v(p)} \overline{K_q}. 
 \]
Note that~$R$ is the intersection with~$Y$ of a closed $\Sigma^1_1$ set. We claim that $x_{v(p)}\in R$. Let $q\sqsubseteq v(p)$. To see that $x_{v(p)}\in \overline{K_q}$, there are two cases. If $n^q<n^{v(p)}$, then $x_{v(p)}\in K_q$ is guaranteed by requirement \ref{item:construction:main_closure_condition} of the construction, which by induction, holds for $v(p)$. Suppose that $n^q = n^{v(p)}$. Then by requirement~\ref{item:construction:2-extensions}, $U_{v(p)}\subseteq U_q$. By \cref{lem:TTT:v_p_and_p_hat_properties}\ref{item:TTT:v_p_p_hat:same_parity_when_same_n}, $v(p)$ and~$q$ have the same orientation. By~\ref{item:construction:edge_condition:x_p_x_p_hat_is_in_U_p}, $x_{v(p)}$ is a left / right end-point of an edge in $U_{v(p)}$, hence of $U_q$, and so $x_{v(p)}\in K_q$. 

By the main property of~$Y$, we can choose $x_p$ and $x_{\hat p}$ in~$R$ that are connected by an edge of~$G$. We then choose sufficiently small neighbourhoods $X_p$ and $X_{\hat p}$ of~$x_p$ and $x_{\hat p}$, subsets of~$Z$, and sufficiently small $D_p = D_{\hat p}\subseteq D_{v(p)}$ that projects to a subset of $G\cap R^2 \cap (X_p\times X_{\hat p})$, and whose projection contains the edge $(x_p,x_{\hat{p}})$.

Let us verify that the requirements of the construction hold for~$p$ and $\hat p$. The requirements \ref{item:construction:edge_condition:x_p_in_X_p}--\ref{item:construction:edge_condition:diameters} are immediate by our choices. To verify~\ref{item:construction:edge_condition:simple_extension}, let $q\in \TTT$, and suppose that either $q\prec p$ or $q\prec \hat p$. By \cref{lem:TTT:v_p_and_p_hat_properties}\ref{item:TTT:v_p_p_hat:when_v_p_is_smaller}, $q\preceq v(p)$. By induction, $X_{v(p)}\subseteq X_q$, and by construction, $\overline{X}_p, \overline{X}_{\hat{p}}\subseteq X_{v(p)}$. 

To verify~\ref{item:construction:2-extensions} and~\ref{item:construction:main_closure_condition}, suppose that $q\sqsubset p$ or $q\sqsubset \hat p$. By \cref{lem:TTT:sqsubset_extensions}\ref{item:TTT:sqsubset:club}, $q\sqsubseteq v(p)$. Since $n^q \le n^{v(p)}< n^p$, in this case \ref{item:construction:2-extensions} holds vacuously; \ref{item:construction:main_closure_condition} holds by construction, since we ensured that $U_p\subseteq R^2$, so all endpoints of edges in~$U_p$ are in $R\subseteq \overline{K_q}$. 

\medskip

In the third and fourth cases, we assume that $n^{v(p)} = n^p$. 

\medskip

\noindent{\textit{Third case:}} $v(p)$ is the $\prec$-predecessor of~$p$, and $v(\hat{p})$ is the $\prec$-predecessor of~$\hat{p}$. In this case we let $x_p = x_{v(p)}$, $x_{\hat{p}} = x_{v(\hat{p})}$, and choose $X_p\subseteq X_{v(p)}$, $X_{\hat{p}}\subseteq X_{v(\hat{p})}$, and $D_p\subseteq D_{v(p)}$, appropriately small so that requirements \ref{item:construction:edge_condition:x_p_x_p_hat_is_in_U_p}, \ref{item:construction:edge_condition:U_p_contained_in_X_p_times_X_p_hat}, and \ref{item:construction:edge_condition:diameters} are satisfied. All the other requirements follow by induction, using the fact that $\hat{v(p)} = v(\hat p)$ (\cref{lem:TTT:v_p_and_p_hat_properties}\ref{item:TTT:v_p_p_hat:when_v_p_is_same}).

\medskip

\noindent{\textit{Fourth case:}} The third case does not hold. Without loss of generality, suppose that $v(p)$ is not the $\prec$-predecessor of~$p$ in~$\TTT$ (we can therefore not assume that~$p$ is a lefty). Let $p^-$ be the $\prec$-predecessor of~$p$. By \cref{lem:TTT:v_p_and_p_hat_properties}\ref{item:TTT:v_p_p_hat:longest_with_n_q_below_n_p}, $n^{p^-}> n^p = n^{v(p)}$. By induction, $x_{p^-}$ is a limit point of endpoints of edges of $U_{v(p)}$, left or right depending on the orientation of $v(p)$, which is the same as the orientation of~$p$ (\cref{lem:TTT:v_p_and_p_hat_properties}\ref{item:TTT:v_p_p_hat:same_parity_when_same_n}). We therefore can choose $x_p$ and $x_{\hat{p}}$ such that:
  \begin{itemize}
    \item $x_p\in X_{p^-}$; and
    \item The edge $(x_p, x_{\hat p})$ (or the reverse, according to parity) is in $U_{v(p)}$. 
  \end{itemize}
We observe that $x_{\hat{p}}\in X_{v(\hat p)}$, since \ref{item:construction:edge_condition:U_p_contained_in_X_p_times_X_p_hat} holds for $v(p)$. As above, we choose small neighbourhoods $X_p$ and $X_{\hat{p}}$ of $x_p$ and~$x_{\hat p}$, and choose sufficiently small $D_p\subseteq D_{v(p)}$ so that $(x_p,x_{\hat p})$ belongs to $U_p$, and $U_p\subseteq X_p\times X_{\hat p}$ (or the reverse). 

For verifying that the requirements hold, the main fact is \cref{lem:TTT:v_p_and_p_hat_properties}\ref{item:TTT:v_p_p_hat:when_v_p_is_same}, that says that $\hat{v(p)} = v(\hat{p})$ is the $\prec$-predecessor of~$\hat{p}$ in~$\TTT$ (this is really the heart of the construction, the reason that~$H_3$ is minimal and~$L_3$ is not). This mainly impacts~\ref{item:construction:edge_condition:simple_extension}. Suppose that $q\prec p$. Then $q\preceq p^-$; so $\overline{X}_p\subseteq X_{p^-}\subseteq X_{q}$. On the other hand, if $q\prec \hat p$, then $q\preceq v(\hat{p})$, and $\overline{X}_{\hat{p}}\subseteq X_{v(\hat{p})}\subseteq X_q$. 

For \ref{item:construction:2-extensions} and~\ref{item:construction:main_closure_condition}, suppose that $q\sqsubset p$ or $q\sqsubset \hat{p}$; then $q\sqsubseteq v(p)$ or $q\sqsubseteq v(\hat{p})$. If $n^q = n^p = n^{v(p)}$ then by induction, $D_{v(p)} = D_{v(\hat{p})}\subseteq D_q$, and $D_p\subseteq D_{v(p)}$. If $n^q<n^p$ then by induction, every end point of an edge in $U_{v(p)}$ is a limit of left / right endpoints of edges in $U_q$, and $U_p\subseteq U_{v(p)}$. [Note that it is only here, to keep the induction going, that we use the full~\ref{item:construction:main_closure_condition}, rather than just assuming that $x_p$ is a limit points of endpoints of~$U_q$, which is what was used up until now.]

\subsubsection*{Verification}

Having performed the construction of $x_p$, $X_p$, $D_p$ for all $p\in \TTT$, we define $F\colon \YYY_3\to X$ as follows: 
\begin{itemize}
  \item For $a\in \YYY_3$, we let $F(a)$ be the limit of $\left\{ x_p \,:\,  p\in \TTT \andd p\prec a \right\}$. 
\end{itemize}

Here we use the properties of the construction, in particular~\ref{item:construction:edge_condition:simple_extension}, as well as \cref{lem:points_have_initial_segments_in_TTT}\ref{item:TTT:initial_segments_in_TTT:existence_of_true_stages}, that ensures that the diameters of $X_p$ for $p\prec a$ indeed go to~0, to see that~$F$ is well-defined and continuous. 

We show that~$F$ is a graph homomorphism. Suppose that $(a,b)$ is an edge of~$H_3$. Then $a = (z,w)$ and $b = (z,w')$, where $n^z<\w$ and $w\symdiff w' = \{c^z(n^z)\}$. For all but finitely many $p\sqsubset a$ in~$\TTT$  we have $n^p = n^z$, and for these~$p$ we have $\hat p \sqsubset b$. If $p,q \sqsubset a$ and $p\preceq q$ then $p\sqsubseteq q$ (\cref{lem:TTT:sqsubset_extensions}\ref{item:TTT:sqsubset:club}); if $n^p = n^z = n^q$ then by requirement \ref{item:construction:2-extensions}, $D_q\subseteq D_p$. Since each $D_p$ is closed and their diameters shrink to~0, $\bigcap \left\{ D_p \,:\,  p\sqsubseteq a \andd n^p = n^z \right\}$ is nonempty; by \ref{item:construction:edge_condition:U_p_contained_in_X_p_times_X_p_hat}, this intersection necessarily projects to $(F(a), F(b))$, so we get $(F(a), F(b))\in G$, as required. 

This completes the proof of \cref{thm:minimality_of_H_3}.

\begin{remark} \label{rmk:can_use_GH}
  Instead of defining shrinking, closed $D_p\subseteq X^2\times \w^\w$ that project to $U_p$, we can (as is done in \cite{LecomteZeleny})  define $U_p \subseteq  \{ (x,x')\in X^2 \,:\, \omega_1^{(x,x')} = \omega_1^{\textup{ck}} \}$. The latter set (call it $\Omega_{X^2}$) is $\Sigma^1_1$, and the restriction of the Gandy-Harrington topology to $\Omega_{X^2}$ is Polish (whereas the Gandy-Harrington topology on all of $X^2$ is not). We can then simply require that the sets $U_p$ are shrinking in a metric that gives the Gandy-Harrington topology on $\Omega_{X^2}$. The construction is essentially the same. 
\end{remark}

\begin{remark} \label{rmk:minimality_for_higher_alphas}
  The argument above shows the following. Let $\alpha \ge 3$. Let~$X$ be a computably presented Polish space, let $G$ be a $\Sigma^1_1$ graph on~$X$, and suppose that there is no countable $\Sigma^0_\alpha(\Delta^1_1)$ colouring of~$G$. Then there is a graph homomorphism $g\colon (\YYY_\alpha, H_\alpha)\to (X,G)$ such that the pullback by~$g$ of any $\Sigma^0_{\alpha^*}(\Delta^1_1)$ set is $\tau_{\alpha^*}$-open. 

  To see this, we can repeat the argument, except that we define $p\preceq q$ to mean $t^p\preceq_{\alpha^*} t^q$ and $\zeta^p\preceq \zeta^q$. There are no new ideas needed, so for length considerations, we omit the details. 
\end{remark}

\section{Separators of iterated Fr\'echet ideals} \label{sec:separators_of_iterated_fr_echet_ideals}

Here we give a new proof of a result of Debs and Saint Raymond \cite{DebsSaintRaymond}, using our forcing methods and untagging.  Day and Marks \cite{DayMarks} gave another proof using forcing, though theirs is a different forcing notion and does not make use of untagging. The theorem (\cref{thm:frechet_separation} below) is not stated explicitly in \cite{DebsSaintRaymond}, but follows immediately from Theorem 3.2 and the proof of Theorem 6.5 in that paper. 

Recall that the first Fr\'echet ideal is the ideal of finite sets; the second is the ideal of sets, all but finitely many of whose columns are finite, and in general, the $\alpha\tth$ iterate of the Fr\'echet ideal are those sets such that for almost all~$n$, their $n\tth$ column belongs to the $(\alpha_n)\tth$ iterate. Thus the natural ``playing ground'' of the $\alpha\tth$ ideal is $2^{\Leaves_\alpha}$. We cast the definition in these terms.

\begin{definition}
For $x \in 2^{\Leaves_\alpha}$, the {\em filter labelling} $F^x = F^x_\alpha$ of $T_\alpha$ is defined as follows:
\begin{itemize}
\item For each $\sigma \in \Leaves_\alpha$, $F^x(\sigma) = x(\sigma)$;
\item For each $\sigma \in T_\alpha\setminus \Leaves_\alpha$,
\[
F^x(\sigma) = 1 \Longleftrightarrow \{ k \in \omega : F^x(\sigma\cat k) = 1\} \text{ is cofinite}.
\]
\end{itemize}

The {\em ideal labelling} $I^x = I^x_\alpha$ of $T_\alpha$ is the dual:
\begin{itemize}
\item For each $\sigma \in \Leaves_\alpha$, $I^x(\sigma) = x(\sigma)$;
\item For each $\sigma \in T_\alpha\setminus \Leaves_\alpha$,
\[
I^x(\sigma) = 0 \Longleftrightarrow \{ k \in \omega : I^x(\sigma\cat k) = 0\} \text{ is cofinite}.
\]
\end{itemize}

The {\em $\alpha\tth$ iterate of the Fr\'echet filter} is the set
\[
\+F_\alpha = \{ x \in 2^{\Leaves_\alpha} : F^x(\seq{}) = 1\}.
\]

The {\em $\alpha\tth$ iterate of the Fr\'echet ideal} is the dual:
\[
\+I_\alpha = \{ x \in 2^{\Leaves_\alpha} : I^x(\seq{}) = 0\}.
\]
\end{definition}

Our objective is the following theorem.
\begin{theorem}[Debs \& Saint Raymond, \cite{DebsSaintRaymond}] \label{thm:frechet_separation}
The $\alpha\tth$ iterates of the Fr\'echet filter and ideal cannot be separated by a ${\mathbf \Delta}^0_{\alpha+1}$ set.
\end{theorem}

We will make use of a modified notion of forcing.

\begin{definition}
We let $\UUU$ be the collection of all finite partial functions $p: T_\alpha \to \{0, 1, \texttt{both}\}$ satisfying: if $\sigma \in \Leaves_\alpha \cap \dom p$, then $p(\sigma) \in \{0, 1\}$.

The set $\UUU$ is partially ordered as follows: for $p, q \in \UUU$, $q \le p$ if and only if:
\begin{itemize}
\item $p \subseteq q$; and
\item If $p(\sigma) \in \{0, 1\}$, then for all $k$ with $\sigma\cat k \in \dom q \setminus \dom p $, $q(\sigma\cat k) = p(\sigma)$.
\end{itemize}
Note that the set $\UUU$ is simpler than the set $\QQ$, but the extension relation is more complicated.
\end{definition}

For a filter $G\subset \UUU$, as above we define $x_G = \bigcup G\rest{\Leaves_\alpha}$. The labelling $\bigcup G$ is neither $F^x$ or $I^x$, but it does indicate which pieces of $x_G$ lie in the appropriate filters and ideals. 
The following is the analogue of \cref{lem:simple_forcing:filters_are_points}, proven by induction on the rank of $\sigma$. In analogy with \cref{def:simple_forcing:Dpoint}, we let $\Dpoint(\UUU)$ be the collection of the following dense subsets of~$\UUU$:
\begin{itemize}
  \item The sets $\{p\in \UUU\,:\, p(\sigma)\converge\}$ for $\sigma\in T_\alpha$;
  \item The sets $\{p\in \UUU\,:\, p\le q\lor p\perp q\}$ for $q\in \UUU$; and
  \item For non-leaf $\sigma\in T_\alpha$, $i\in\{0,1\}$, and $k\in \Nat$, the sets 
  \[
    \left\{ p\in \UUU  \,:\,  p(\s)= \texttt{both} \then (\exists m>k)\,\,p(\s\conc m)=i \right\}. 
  \]
\end{itemize}

\begin{lemma}\label{lem:frechet_forcing_equals_truth}
Suppose that $G\subset \UUU$ is $\Dpoint(\UUU)$-generic. Then $x_G\in 2^{\Leaves_{\alpha}}$, and for all~$\s$,  $F^{x_G}(\sigma) = 1 \iff \left(\bigcup G\right) (\sigma) = 1$, and $I^{x_G}(\sigma) = 0 \iff \left(\bigcup G\right)(\sigma) = 0$.
\end{lemma}

Define $p_0, p_1 \in \UUU$ by $p_0(\seq{}) = 0$, $p_1(\seq{}) = 1$, and both are undefined everywhere else.  Then $\+F_\alpha$ is equivalent to $[p_1]$ modulo a $\tau_{\UUU}$-meagre set, and the same for $\+I_\alpha$ and $[p_0]$.

We have a modified version of restriction.
\begin{definition}
For $p \in \UUU$ and $\beta \le \alpha$, we define a condition $p\rest \beta \subseteq p$ as follows: for $\sigma \in \dom p$, $\sigma \in \dom p\rest \beta$ if and only if either:
\begin{itemize}
\item $\rk(\sigma) < \beta$; or
\item $\rk(\sigma) = \beta$ and $p(\sigma) \neq \texttt{both}$.
\end{itemize}
\end{definition}

Observe that we ``gain an ordinal'' compared to the previous notion of restriction; for example, $p\rest{\Leaves_\alpha}$ is $p\rest{0}$, not $p\rest{1}$.  This will occur again in the modified untagging lemma.

We use the same definition of strong forcing for this new notion of forcing, and the analog of \cref{prop:simple_forcing:forcing_theorem} is by the same proof.

We again define a notion of $\beta$-completeness. Again the idea is that if $\rk(\s)$ is a limit, $\beta<\rk(\s)$, and $p(\s)$ tells us what the values of all undefined $p(\s\conc k)$ should be, then $p$ records these values where $\rk(\s\conc k)<\beta$. 
\begin{definition}
For $\beta \le \alpha$, we say that $p \in \UUU$ is {\em $\beta$-complete} if for all $\sigma\in \dom p$ with $\rk(\sigma) > \beta$ and $p(\sigma) \neq \texttt{both}$, $p(\sigma\cat k)\converge$ for all $k$ such that $\rk(\sigma\cat k) < \beta$.
\end{definition}

The following density of $\beta$-complete conditions is straightforward.
\begin{lemma}
Let $\beta \le \alpha$.  For all $p \in \UUU$, there is a $\beta$-complete $p'$ extending $p$.
\end{lemma}

We have our version of the key technical lemma.
\begin{lemma}\label{lem:frechet_technical}
Let $\gamma' < \gamma \le \alpha$.  If $p \in \UUU$ is $\gamma$-complete and $r$ extends $p\rest \gamma$, then $(r\rest \gamma') \cup p$ extends both $p$ and $r\rest \gamma'$.
\end{lemma}

\begin{proof}
Extending $r\rest \gamma'$ is immediate by definition.

To show extension of $p$, the only concern is that there might be $\sigma \in \dom p$ with $p(\sigma) \neq \texttt{both}$, and $k$ with $\sigma \cat k \in \dom (r\rest \gamma') \setminus \dom p$.  But this indicates that $\rk(\sigma \cat k) \le \gamma' < \gamma$, so by the $\gamma$-completeness of $p$ we know that $\rk(\sigma) \le \gamma$, and so $\sigma \in \dom (p\rest \gamma)$.  By definition of the extension relation $(r\rest \gamma')(\sigma\cat k) = r(\sigma\cat k) = p(\sigma)$, so there is no obstacle to extending $p$.
\end{proof}

Now we have our untagging lemma. Note that in contrast with \cref{prop:simple_forcing:Pi_untagging_lemma}, the somewhat different definition of $p\rest{\gamma}$ allows us to ``gain a quantifier''; to force a $\Pi_\gamma$ fact, $p\rest\gamma$ suffices, we don't need $p\rest{(\gamma+1)}$. 

\begin{proposition}
Let $\gamma \le \alpha$, and let $\varphi$ be a $\Pi_\gamma$ Borel code.  For $p \in \UUU$, if $p$ is $\gamma$-complete and $p \Vdash^* \varphi$, then $p\rest \gamma \Vdash^* \varphi$.
\end{proposition}

\begin{proof}
The proposition is proved by induction on $\gamma$.

The base case $\gamma = 0$ follows from the definition of strong forcing, along with the fact that $p(\sigma) \neq \texttt{both}$ for all $\sigma \in \Leaves_\alpha$.

Suppose that $\gamma > 0$, and that the proposition has been verified for all $\gamma' < \gamma$.  Let $\varphi$ be a $\Pi_\gamma$ Borel code, and let $p \in \UUU$ be $\gamma$-complete.  We prove the contrapositive.

Suppose $p\rest \gamma \not \Vdash^* \varphi$.  Since $\varphi$ is $\neg \bigvee_n \psi_n$ (where each $\psi_n$ is $\Pi_{\gamma'}$ for some $\gamma' < \gamma$), by definition this means there is some $n$ and some $r$ extending $p\rest \gamma$ such that $r \Vdash^* \psi_n$.  We may assume $r$ is $\gamma'$-complete.  By induction, $r\rest \gamma' \Vdash^* \psi_n$.  By \cref{lem:frechet_technical}, $(r\rest \gamma') \cup p$ extends $p$ and $r\rest \gamma'$, and thus witnesses that $p \not \Vdash^* \varphi$.
\end{proof}

As usual we obtain:

\begin{corollary}
Let $\gamma \le \alpha$, and let $\varphi$ be a $\Sigma_{\gamma+1}$ Borel code.  If $G\subset \UUU$ is sufficiently generic, then $x_G\in [\vphi]$ if and only if there is some $p\in (G\cap \UUU_\gamma)$such that $p \Vdash^* \varphi$. 
\end{corollary}

We are now ready to prove \cref{thm:frechet_separation}.

\begin{proof}
Towards a contradiction, suppose $Z$ were a ${\mathbf \Delta}^0_{\alpha+1}$ set with $\+I_\alpha\subseteq Z$ and $\+F_\alpha\subseteq Z^\complement$. Fix $\varphi_i$ and $\varphi_f$, $\Sigma_{\alpha+1}$ Borel codes for $Z$ and its complement, respectively.

Let $p_{\texttt{both}} \in \UUU$ be given by $p_{\texttt{both}}(\seq{}) = \texttt{both}$, and $p_{\texttt{both}}$ is undefined everywhere else.  Let $G$ be a sufficiently generic filter containing $p_{\texttt{both}}$, and let $x = x_G$.

Suppose $x \in Z$.  Fix $q \in G\cap \UUU_{\alpha}$ with $q \Vdash^* \varphi_i$. Since $q\in \UUU_\alpha$, $q$ is undefined at the root.  Define $q'$ extending $q$ by the definition $q'(\seq{}) = 1$.  Note that $q' \le q$, and thus $q' \Vdash^* \varphi_i$.

Let $H$ be a sufficiently generic filter containing $q'$.  Then $x_H \in Z$, since $q' \Vdash^* \varphi_i$.  But $F^{x_H}(\seq{}) = 1$ by \cref{lem:frechet_forcing_equals_truth}, and so $x_H\in \+F_\alpha$, a contradiction. If $x \not \in Z$, mutatis mutandis.
\end{proof}

\section{Questions} \label{sec:questions}

We state some open questions. 

\begin{question}
    Is there a way to define graphs $H_\alpha$ for $\alpha$ which are not successors of successors?
\end{question}

One possibility would be to replace the relation $s\prec_{\alpha^*}t$ in the definition of $n^t$ and $k^t(m)$ by the relation ``for all $\s$ with $|\s|\ge 3$, if $T^s(\s)\converge$ then $T^t(\s) = T^s(\s)$''. What quickly goes wrong is the property \cref{lem:true_stages:basic_properties}\ref{item:true_stages:basic_properties:diamond}. It would be interesting to see if the partial labellings $T^t$ can be modified so that this property is recovered. Even then, it is not clear how to prove that such a graph is not $\bSigma^0_\alpha$-colourable. 

\smallskip

For the following question, let $\AAA_3 = \left\{ (a,a) \,:\,  a= (x,y)\in \YYY_3 \andd T^x(\rooot)=1 \right\}$, and let $\BBB_3$ be the directed graph $H_3$. Recall that for Polish spaces $X$ and~$Y$, a set $C\subseteq X\times Y$ is $(\bSigma^0_\alpha\times \bSigma^0_\alpha)_\s$ if it is of the form $\bigcup_n (D_n\times E_n)$, where each $D_n$ and $E_n$ are $\bSigma^0_\alpha$. The results above show that there is no set $C\subseteq (\YYY_3)^2$ which is $(\bSigma^0_3\times \bSigma^0_3)_\s$ and separates $\AAA_3$ from $\BBB_3$. The question is whether this is a least example: 

\begin{question}
  Suppose that $X$, $Y$ are Polish spaces, that $A,B\subseteq X\times Y$ are $\bSigma^1_1$ and disjoint, and further, that there is no set $C\subseteq X\times Y$ which is $(\bSigma^0_3\times \bSigma^0_3)_\s$ such that $A\subseteq C$ and $B\subseteq C^\complement$. 

  Must there be continuous functions $f\colon \YYY_3\to X$ and $g\colon \YYY_3\to Y$ such that for all $(a,a)\in \AAA_3$, $(f(a),g(a))\in A$, and for all $(a,b)\in \BBB_3$, $(f(a), f(b))\in B$?
\end{question}

The background here is as follows. Lecomte \cite{L1} derived from the $\mathbb{G}_0$ dichotomy (\cref{thm:G_0_dichotomy}) a dichotomy result, characterising when two disjoint analytic sets can be separated by a countable union of Borel rectangles. He also showed that \cref{thm:G_0_dichotomy} is an easy corollary of this other dichotomy. In \cite{LecomteZeleny}, Lecomte and Zelen\'y found least examples of sets that are not separable by $(\bSigma^0_1\times \bSigma^0_1)_\s$ sets, and by $(\bSigma^0_2\times \bSigma^0_2)_\s$ sets; the problem for $\alpha =3$ and higher is still open.

\bibliography{mainbib}
\bibliographystyle{alpha}

\end{document}